\documentclass[11pt]{article}

\usepackage{etex}
\usepackage[margin={1.2in}]{geometry}
\usepackage[titletoc]{appendix}

\usepackage{tikz}
\usetikzlibrary{matrix}
\usepackage{eucal}
\usepackage{mathrsfs}
\usepackage{stmaryrd}

\usepackage{rotating}

\usepackage{longtable}

\usepackage[sc]{mathpazo}
\usepackage{courier}

\usepackage[utf8]{inputenc}
\usepackage[T1]{fontenc}

\usepackage[english]{babel}
\usepackage{blindtext}

\usepackage{amsmath}

\usepackage{microtype}

\usepackage{amsmath}
\usepackage{amssymb}
\usepackage{amsthm}
\usepackage{color}
\usepackage{latexsym,extarrows}

\usepackage[all]{xy}
\usepackage[stable,multiple]{footmisc}

\RequirePackage[font=small,format=plain,labelfont=bf,textfont=it]{caption}
\makeatletter
\newsavebox{\@brx}
\newcommand{\llangle}[1][]{\savebox{\@brx}{\(\m@th{#1\langle}\)}%
  \mathopen{\copy\@brx\kern-0.5\wd\@brx\usebox{\@brx}}}
\newcommand{\rrangle}[1][]{\savebox{\@brx}{\(\m@th{#1\rangle}\)}%
  \mathclose{\copy\@brx\kern-0.5\wd\@brx\usebox{\@brx}}}
\makeatother

\begin{document}
%%%%%%%%%%%%%%%%%%%%%%%%%%%%%%%%%%%%%%%%%%%%%%%%%%%%%%%%%%%%%%%%%%%%%%%%
%%%%%%%%%%%%%%%%%%%%%%%%%%     Macros      %%%%%%%%%%%%%%%%%%%%%%%%%%%%%
%%%%%%%%%%%%%%%%%%%%%%%%%%%%%%%%%%%%%%%%%%%%%%%%%%%%%%%%%%%%%%%%%%%%%%%%
\def\e#1\e{\begin{equation}#1\end{equation}}
\def\ea#1\ea{\begin{align}#1\end{align}}
\def\eq#1{{\rm(\ref{#1})}}
\theoremstyle{plain}% default
\newtheorem{thm}{Theorem}[section]
\newtheorem{lem}[thm]{Lemma}
\newtheorem{prop}[thm]{Proposition}
\newtheorem{cor}[thm]{Corollary}
\theoremstyle{definition}
\newtheorem{dfn}[thm]{Definition}
\newtheorem{ex}[thm]{Example}
\newtheorem{rem}[thm]{Remark}
\newtheorem{conjecture}[thm]{Conjecture}
\newtheorem{convention}[thm]{Convention}

\newcommand{\D}{\mathrm{d}}
\newcommand{\A}{\mathcal{A}}
\newcommand{\LL}{\llangle[\Big]}
\newcommand{\RR}{\rrangle[\Big]}
\newcommand{\LD}{\Big\langle}
\newcommand{\RD}{\Big\rangle}
\newcommand{\F}{\mathcal{F}}
\newcommand{\HH}{\mathcal{H}}
\newcommand{\X}{\mathcal{X}}
\newcommand{\PP}{\mathbb{P}}
\newcommand{\K}{\mathscr{K}}
\newcommand{\q}{\mathbf{q}}

%%%%%%%%%%%%%%%%%%%%%%%%%%%%
\newcommand{\op}{\operatorname}
\newcommand{\C}{\mathbb{C}}
\newcommand{\N}{\mathbb{N}}
\newcommand{\R}{\mathbb{R}}
\newcommand{\Q}{\mathbb{Q}}
\newcommand{\Z}{\mathbb{Z}}
\renewcommand{\H}{\mathbf{H}}

\newcommand{\Etau}{\text{E}_\tau}
\newcommand{\E}{{\mathcal E}}
\newcommand{\G}{\mathbf{G}}
\newcommand{\eps}{\epsilon}
\newcommand{\g}{\mathbf{g}}
\newcommand{\im}{\op{im}}

\newcommand{\h}{\mathbf{h}}

\newcommand{\Gmax}[1]{G_{#1}}
\newcommand{\AW}{E}
%%%%%%%%%%%%%%%%%%%%%%         Functions         %%%%%%%%%%%%%%%%%%%%%%%%%%%%%%%%%%%
\providecommand{\abs}[1]{\left\lvert#1\right\rvert}
\providecommand{\norm}[1]{\left\lVert#1\right\rVert}
\newcommand{\abracket}[1]{\left\langle#1\right\rangle}
\newcommand{\bbracket}[1]{\left[#1\right]}
\newcommand{\fbracket}[1]{\left\{#1\right\}}
\newcommand{\bracket}[1]{\left(#1\right)}
\newcommand{\ket}[1]{|#1\rangle}
\newcommand{\bra}[1]{\langle#1|}

\newcommand{\ora}[1]{\overrightarrow#1}

\providecommand{\from}{\leftarrow}
\newcommand{\bl}{\textbf}
\newcommand{\mbf}{\mathbf}
\newcommand{\mbb}{\mathbb}
\newcommand{\mf}{\mathfrak}
\newcommand{\mc}{\mathcal}
\newcommand{\cinfty}{C^{\infty}}
\newcommand{\pa}{\partial}
\newcommand{\prm}{\prime}
\newcommand{\dbar}{\bar\pa}
\newcommand{\OO}{{\mathcal O}}
\newcommand{\hotimes}{\hat\otimes}
\newcommand{\BV}{Batalin-Vilkovisky }
\newcommand{\CE}{Chevalley-Eilenberg }
\newcommand{\suml}{\sum\limits}
\newcommand{\prodl}{\prod\limits}
\newcommand{\into}{\hookrightarrow}
\newcommand{\Ol}{\mathcal O_{loc}}
\newcommand{\mD}{{\mathcal D}}
\newcommand{\iso}{\cong}
\newcommand{\dpa}[1]{{\pa\over \pa #1}}
\newcommand{\Kahler}{K\"{a}hler }
\newcommand{\0}{\mathbf{0}}

\newcommand{\B}{\mathcal{B}}
\newcommand{\V}{\mathcal{V}}

\newcommand{\M}{\mathfrak{M}}

\renewcommand{\Im}{\op{Im}}
\renewcommand{\Re}{\op{Re}}

%%%%%%%%%%%%%%%%%%%%%%%%%%%%%
\newcommand{\DD}{\Omega^{\text{\Romannum{2}}}}

%%%%%%%%%%%%%%%%%%%%%%%%%%%

\numberwithin{equation}{section}

\newcommand{\comment}[1]{\textcolor{red}{[#1]}} %for displaying red texts
%\newcommand{\comment}[1]{} %for not displaying red texts
%\newcommand{\com}[2]{\textcolor{blue}{[#2]}} %for displaying red texts

%%%
\makeatletter
\newcommand{\subjclass}[2][2010]{%
  \let\@oldtitle\@title%
  \gdef\@title{\@oldtitle\footnotetext{#1 \emph{Mathematics Subject Classification.} #2}}%
}
\newcommand{\keywords}[1]{%
  \let\@@oldtitle\@title%
  \gdef\@title{\@@oldtitle\footnotetext{\emph{Key words and phrases.} #1.}}%
}
\makeatother
%%%

\makeatletter
\let\orig@afterheading\@afterheading
\def\@afterheading{%
   \@afterindenttrue
  \orig@afterheading}
\makeatother

%%%%%%%%%%%%%%%%%%%%%%%%%%%%%%%%%%%%%%%%%%%%%%%%%%%%%%%%%%%%%%%%%%%%%%%%
%%%%%%%%%%%%%%%%%%%%%%%    Text of paper    %%%%%%%%%%%%%%%%%%%%%%%%%%%%
%%%%%%%%%%%%%%%%%%%%%%%%%%%%%%%%%%%%%%%%%%%%%%%%%%%%%%%%%%%%%%%%%%%%%%%%
%The following title can be changed according to needs.%
\title{\bf Genus Two Quasi-Siegel Modular Forms and Gromov-Witten Theory of Toric Calabi-Yau Threefolds}
\author{Yongbin Ruan, Yingchun Zhang and Jie Zhou}
\date{}
%\subjclass{14N35, 11Fxx}
\maketitle

\begin{abstract}

We first develop theories of differential rings of quasi-Siegel modular and quasi-Siegel Jacobi forms for genus two.
% that are of independent interest.
Then we apply them to the Eynard-Orantin topological recursion  of certain local Calabi-Yau threefolds equipped with
branes,
whose mirror curves are genus-two hyperelliptic curves.
By the proof of the Remodeling Conjecture, we prove that the corresponding open- and closed-
Gromov-Witten potentials are essentially quasi-Siegel Jacobi and quasi-Siegel modular forms for genus two, respectively.
 \end{abstract}

\setcounter{tocdepth}{2} \tableofcontents

\section{Introduction}

       One of the most exciting aspects of Gromov-Witten (GW) theory is its interaction with the theory of modular forms.
 Now it is well accepted from
 the B-model of mirror symmetry that the generating functions of GW theory carry certain non-holomorphic signature in the sense
 that they should be viewed as certain holomorphic limits of non-holomorphic generating functions. It is amazing that it coincides with the  non-holomorphic signature
  in the theory of the so-called quasi-modular forms.
 We should mention that there is similar and yet different kind of non-holomorphic signature in
 mock modular forms which made its appearance recently in quantum K-theory.

 There is a great deal of literature discussing
 this connection between GW theory and quasi-modular forms
since the works \cite{Dijkgraaf:1995, Kaneko:1995, Okunkov:2002} discussing the elliptic curve case.
    Despite of many surprising progresses that have been made\footnote{We refer the interested readers to the
    introduction part of \cite{Shen:2014} for an incomplete list of them.}, so far most of the results on quasi-modularity
  are established only for the cases where the dimension of the moduli space is one.
  One of the main difficulties in generalizing the quasi-modularity investigations to higher dimensional moduli spaces appearing as
  Hermitian symmetric spaces is that  it is not clear
what the appropriate notion of
 "quasi"
 is (see an early work \cite{Klemm:2015direct}).\\

The purpose of this work is two-fold.
 In the first half of this paper, we partially fill this gap for
 the case of the moduli space of genus two Riemann surfaces by building a theory of differential ring of quasi-Siegel modular forms.
One of the main results is that the differential ring is finitely generated by explicit generators.
Our technique is a combination of various perspectives of periods and quasi-periods
such as those in complex differential geometry, differential equations, representation theory and Hodge theory, and  is based on earlier results from
\cite{Igusa:1960arithmetic, Igusa:1962siegel, Igusa:1967modular, Bertrand:2000transcendence, VanderGeer:2008, Urban:2014nearly}.
The second half of the paper is devoted to the application of quasi-modular forms to the GW theories of toric Calabi-Yau (CY) 3-folds whose B-model mirror theories are governed by
Eynard-Orantin \cite{Eynard:2007invariants} topological recursion on genus two mirror curves,
thereby extending the earlier results in  \cite{Fang:2018open} from toric CYs with genus one mirror curves to those with genus two ones.

 \subsection{Differential ring of quasi-Siegel modular forms for genus two}

To state our results, we first need to recall some notations. Let
\begin{equation}
\mathcal{H}=\{\tau\in M_{2}(\mathbb{C}) \,|\, \mathrm{Im}\,\tau>0\,, \tau^{t}=\tau\}
\end{equation}
be the Siegel upper-half space of genus two.
Let $\Gamma$ be a congruence subgroup of the full Siegel modular group $\Gamma(1)=\mathrm{Sp}_{4}(\mathbb{Z})$
that acts on $\mathcal{H}$ by
\begin{equation}\label{eqntautransform}
\tau\mapsto \gamma \tau=(a\tau+b)(c\tau+d)^{-1}
\,,\quad
\forall\,
\gamma=
\begin{pmatrix}
a & b\\
c & d
\end{pmatrix}\in \Gamma\,.
\end{equation}
Here to avoid technicality we assume that it has no torsion, later this assumption will be lifted.
Let $S=\Gamma\backslash\mathcal{H}$ be the corresponding Siegel modular variety
which by Baily-Borel \cite{Baily:1966} admits the structure of a quasi-projective algebraic variety, and
$\pi:\mathcal{Z}\rightarrow S$ be the universal family of Abelian varieties with level structure determined by $\Gamma$.
Define the following locally free sheaves (or vector bundles) on $S$
\begin{equation}
\underline{\omega}=R^{0}\pi_{*}\Omega^{1}_{\mathcal{Z}|S}\,,
\quad
\mathcal{V}_{S}=R^{1}\pi_{*}\mathbb{C}_{\mathcal{Z}}\otimes \mathcal{O}_{S}\,,
\end{equation}
where $\Omega^{1}_{\mathcal{Z}|S}$ is the sheaf of relative differentials.

Generalizing the discussions in \cite{VanderGeer:2008, Urban:2014nearly},
we arrive at a sheaf-theoretic definition of quasi-modular forms.
\begin{dfn}[Definition \ref{dfnalgebraicdefinition}]
Let the notations be as above.
A weakly holomorphic\footnote{Here we are following the terminology used for
the genus one case (see e.g.
\cite{Zagier:2008}) in which "weak holomorphicity" stands for holomorphicity away from $\bar{S}-S$ where $\bar{S}$
is a certain compactification of $S$.} quasi-modular form
of weight $k$ and order $m\geq 0$ for the modular group $\Gamma$ is a holomorphic section of
$\mathrm{Sym}^{\otimes (k-m)}\underline{\omega}\otimes \mathrm{Sym}^{\otimes m} \mathcal{V}_{S}$.
\end{dfn}

By working with a particular frame adapted to the Hodge filtration on the pull-back sheaves
along $\mathcal{H}\rightarrow S$, global sections
of $\mathrm{Sym}^{\otimes (k-m)}\underline{\omega}\otimes \mathrm{Sym}^{\otimes m} \mathcal{V}_{S}$
can be described in terms of the coordinates of the corresponding pull-back sections.
As explained in \cite{Urban:2014nearly}, this particular frame is naturally motivated from the monodromy weight filtration
around a singular fiber in the family $\pi:\mathcal{Z}\rightarrow S$.
This then yields
an equivalent analytic definition of quasi-modular forms in the spirit of \cite{Kaneko:1995}.

\begin{dfn}[Definition \ref{dfnanalyticdefinition}]
Let the notations be as above.
A weakly holomorphic quasi-modular form of weight $k$ and order $m$ for the modular group $\Gamma$ is a collection of holomorphic functions $\{f_{i}(\tau)\}_{i,=0,1\cdots,m}$ on $\mathcal{H}$ which
 satisfy the transformation law
\begin{equation*}
f_{i}(\tau)=
\sum_{\ell= i}^{m}
{\ell\choose i} f_{\ell}(\gamma\tau)  ((c \tau+d)^{-1})^{\otimes (k-\ell)} (- c^t)^{\otimes (\ell-i)} (( c\tau+d)^t )^{\otimes i}\,,
\quad
\forall \gamma\in \Gamma\,.
\end{equation*}
\end{dfn}

%%%
\iffalse
The equivalence for the above two definitions follows from the familiar fact that sections of bundles
on $S$ with trivial pull backs
to
$\mathcal{H}$
corresponds to $\Gamma$-invariant sections of the pull-back bundles, and by fixing a particular global trivialization
of the pull-back bundle upstairs (distinguished by monodromy), these sections
are functions on $\mathcal{H}$ with given automorphy factor determined by the transition function of the bundle downstairs.
\fi
%%%

In particular, weakly holomorphic
quasi-modular forms with $m=0$ are weakly holomorphic, vector-valued modular forms
corresponding to global sections of the sheaves $\mathrm{Sym}^{\bullet\geq 0}\underline{\omega}:
=\{ \mathrm{Sym}^{\otimes k}\underline{\omega}\}_{k\geq 0}$.\\

Using the further structure of $\mathcal{H}=\mathrm{Sp}_{4}(\mathbb{R})/\mathbb{U}_{2}$
as a Hermitian symmetric domain, the bundles above can be regarded as homogeneous vector bundles
defined by representations of $\mathbb{U}_{2}$.
The bundle $\underline{\omega}$ corresponds to the fundamental representation, while $\mathcal{V}_{S}$ the trivial representation.
Moreover, one has the following elementary fact
\begin{equation}
\underline{\omega}^{\otimes 2}=\mathrm{Sym}^{\otimes 2}\underline{\omega}
\oplus \wedge^{2} \underline{\omega}\,,
\quad
\wedge^{2} \underline{\omega}=\det \underline{\omega}\,.
\end{equation}
This then tells that meromorphic
scalar-valued modular forms correspond to global sections of the sheaves $\det^{\bullet} \underline{\omega}$, which are studied thoroughly by Igusa in
\cite{Igusa:1960arithmetic, Igusa:1962siegel, Igusa:1964siegel, Igusa:1967modular}.

By the Weyl character formula, any holomorphic representation
of $\mathbb{U}_{2}$
decomposes into a direct sum whose summands lie in $\mathrm{Sym}^{\bullet}\underline{\omega},\det^{\bullet} \underline{\omega}$.
Therefore, the set of representations
$\{ \det^{\bullet}{\underline{\omega}} \otimes \mathrm{Sym}^{\bullet}\underline{\omega},\oplus,\otimes \}$ carries a module structure
over $\{\det^{\bullet}{\underline{\omega}},\oplus,\otimes\}$. See \cite{VanderGeer:2008} for related discussions.
Hence in addition to a direct examination by using the analytic definition,
the representation-theoretic description leads to another way of showing the
ring structure of meromorphic vector-valued modular forms graded by the representations
$\mathrm{Sym}^{\bullet}\underline{\omega}, \det^{\bullet}{\underline{\omega}}$.
Furthermore, since $\mathcal{V}_{S}$ corresponds to the trivial representation,
similar module structure holds for
$\{ \det^{\bullet}{\underline{\omega}} \otimes \mathrm{Sym}^{\bullet}\underline{\omega}
\otimes \mathrm{Sym}^{\bullet \geq 0} \mathcal{V}_{S}\}$.\\

For practical reasons, in this paper we shall work analytically with a collection of subspaces of
$H^{0}(S,  \det^{\bullet}{\underline{\omega}} \otimes
\mathrm{Sym}^{\bullet}\underline{\omega}
\otimes \mathrm{Sym}^{\bullet \geq 0} \mathcal{V}_{S}) $ given by\footnote{Here $H^{0}(X,-)$ denotes the global section functor.}
\begin{equation}
 H^{0}( S, \mathrm{det}^{\bullet}\underline{\omega})
\otimes
\mathrm{Sym}^{\bullet\geq 0} H^{0}(S, \underline{\omega})\otimes
 \mathrm{Sym}^{\bullet \geq 0} H^{0}(S,\mathcal{V}_{S})\,.
\end{equation}
\begin{dfn}[Definition \ref{defringofmodularforms}, Definition \ref{defringofquasimodularforms}]\label{dfnsubring}
Let $R(\Gamma)$ be the fractional ring of the graded ring of scalar-valued modular forms for $\Gamma$.
Define a subring of vector-valued modular forms to be the $R(\Gamma)$-module given by
\begin{equation}
\mathcal{R}(\Gamma)=R(\Gamma)[ \mathrm{Sym}^{\bullet\geq 0} H^{0}(S, \underline{\omega})]\,.
\end{equation}
Define a subring of quasi-modular forms to be the $\mathcal{R}(\Gamma)$-module given by
\begin{equation}
\widetilde{\mathcal{R}}(\Gamma)
=
R(\Gamma) [\mathrm{Sym}^{\bullet\geq 0} H^{0}(S, \underline{\omega}),
 \mathrm{Sym}^{\bullet \geq 0} H^{0}(S,\mathcal{V}_{S})
]\,.
\end{equation}
\end{dfn}
\begin{convention}
Hereafter unless stated otherwise, by modular forms we mean elements in $\mathcal{R}(\Gamma)$
and by quasi-modular forms elements in $\widetilde{\mathcal{R}}(\Gamma)$.
\end{convention}

The sheaf-theoretic origin of Definition \ref{dfnsubring}
above not only gives the ring structure by construction, but also allows to describe the differential structure in terms of very concrete geometric terms.
We now briefly explain this.
Denote the rational functional field of the quasi-projective algebraic variety $S$ by $k(S)$.
Up to algebraic extension, $k(S)$ is the same as
the rational functional field of $\Gamma(1)\backslash \mathcal{H}$
which admits very explicit presentation via scalar-valued Siegel modular forms
\cite{Igusa:1962siegel, Igusa:1964siegel}, as shall be reviewed in Section
\ref{secfamiliesofgenustwocurves} below.
We take a set of three algebraically independent generators $\{t_1,t_2,t_3\}$ from $k(S)$ and denote them
 collectively by $t$.
 Hence up to algebraic extension $k(S)$ coincides with $k(t_1,t_2,t_3)$.
Denote
the differentials $\{\partial_{\tau_{ij}}, i,j,=1,2\}$ collectively by $\partial_{\tau}$
and similarly for $\partial_{t}$.
Taking a weak Torelli marking $\{A,B\}$ and a frame $\omega,\eta$ adapted to the Hodge filtration of
 the universal family
$\pi:\mathcal{Z}\rightarrow S$ of Abelian varieties.
Let $\Pi_{A}(\omega), \Pi_{A}(\eta)$ be the $A$-cycle periods and quasi-periods.

\begin{thm}[Theorem \ref{thmdifferentialringofquasimodularforms}\cite{Bertrand:2000transcendence}, Theorem \ref{thmdifferentialGaloisfield}\cite{Bertrand:2000transcendence},
 Corollary \ref{cordifferentialringofquasimodularforms}, Corollary \ref{corBergmankerneldegreezeroterm}]
\label{thmmainthm1}

Let the notations be as above.
Let $\mathcal{D}$ be the differential ring obtained by  adjoining  to $k(S)$ all of the
$\partial_{\tau}^{n},n\geq 0$ derivatives  of rational functions in $k(S)$.

\begin{enumerate}
\item
The differential ring  $\mathcal{D}$
%is generated by
%$\partial_{\tau}t,\partial_{\tau}^{2}t$ over $k(S)$.
is also stable under $\partial_{t}$.
The fractional field of $\mathcal{D}$ has transcendental degree $10$ over $\mathbb{C}$.

\item
Up to algebraic extension,
the differential ring $\mathcal{D}$ is generated over $k(S)$ by the entries of
\begin{eqnarray}
\partial_{\tau}t \,,\quad \partial_{\tau}^{2}t\,.
\end{eqnarray}

Elements in $k(S)[\partial_{\tau}t]$ are modular forms for $\Gamma$
while $\partial_{\tau}^{2}t$ quasi-modular.

\item
Up to algebraic extension, the differential ring $\mathcal{D}$ is identical to the ring of quasi-modular forms:
\begin{equation}
\mathcal{D}=\widetilde{\mathcal{R}}(\Gamma)\,,
\quad
\textrm{up to algebraic extension}\,.
\end{equation}

Let $\Pi_{A}(\omega)$ and $\Pi_{A}(\eta)$
be the periods and quasi-periods associated to the family
$\pi:\mathcal{Z}\rightarrow S$, respectively.
The
generators of $\widetilde{\mathcal{R}}(\Gamma)$ over $k(S)$ can be taken to be the entries of
\begin{eqnarray}
\Pi_{A}(\omega) \,,\quad \Pi_{A}(\eta)\,,
\end{eqnarray}
 subject to the only relation given by the 1st Riemann-Hodge bilinear relation
\begin{equation}
\Pi_{A}(\omega) \Pi_{A}(\eta)^{t}
=
\Pi_{A}(\eta) \Pi_{A}(\omega)^{t}
\,.
\end{equation}

Elements in $k(S)[\Pi_{A}(\omega)]$ are modular forms for $\Gamma$,
while $\Pi_{A}(\eta)$ is quasi-modular with the obstruction of modularity given by the 2nd Riemann-Hodge bilinear relation/Legendre period relation.
\item
Up to algebraic extension,
the differential ring $\mathcal{D}$ is generated over $k(S)$ by
\begin{eqnarray}
\partial_{\tau}\log f\,,\quad [z^{0}]\partial^2_{z}\log \vartheta_{\delta}\,,
\end{eqnarray}
where $ f$ is any nonzero element in $ k(S)$,
$\vartheta_{\delta}$ is any Jacobi theta function with an odd characteristic $\delta$
defined on dimension two principally polarized Abelian varieties, and  $[z^{0}]\partial^2_{z}\log \vartheta_{\delta}$
stands for the degree zero term in the Laurent expansion of  $\partial^2_{z}\log \vartheta_{\delta}$.
Elements in $k(S)[\partial_{\tau}\log f]$ are  modular forms for $\Gamma$,
while $[z^{0}]\partial^2_{z}\log \vartheta_{\delta}$ is quasi-modular.

\end{enumerate}

\end{thm}

 We should mention that an attempt in developing a theory of quasi-modular forms for the genus two case
 was
 made in \cite{Klemm:2015direct} using differential operators on vector-valued modular forms. As far as we understand, there is no finite generation result for a theory of the ring of quasi-modular forms in their setting.
In addition to our finite generation theorem, our results differ from \cite{Klemm:2015direct}  and other previous  works in that we provide various realizations of the \emph{differential} ring from various perspectives
such as the connection to quasi-periods. These aspects are crucial to understand points such as the
reason behind the failure of modularity.
More importantly,
the ring of quasi-modular forms that we build is universal in a certain sense
(detailed in Section \ref{secdifferentialstructure}) and can be generalized to other cases such as
 the moduli space of principally polarized Abelian varieties
of higher dimensions.
We hope to explore this, as well as the
generalization to general lattice polarized K3 surfaces and Calabi-Yau threefolds such as the quintic family\footnote{See \cite{Zhou:2021} for relevant discussions.},
in a future investigation.

 \subsection{Modularity in topological recursion and in open-closed GW theory}

     The second half of the paper is a continuation of the paper \cite{Fang:2018open}.
     When the mirror curve of the toric CY is of genus one, a detailed analysis of topological recursion in terms of quasi-modular forms
     and their Jacobian cousins is performed therein.

     Using the theory of quasi-modular forms developed earlier and the results of \cite{Grant:1985theta, Grant:1990formal}
  for the differential field of rational functions of genus two hyperelliptic Jacobians, we can then
  define a ring of meromorphic quasi-Jacobi forms (see Definition \ref{eqndefinitionofJ}).
  Then we
  generalize the previous analysis in \cite{Fang:2018open} to topological recursion for the cases where the curve family is a 3-dimensional family of genus two curves equipped with hyperelliptic structures.
Specializing to a particular genus two curve family, namely the one arising from the mirror of the resolution $\mathcal{X}$ of $\mathbb{C}^{3}/\mathbb{Z}_{6}$,
 these results on quasi-modularity in topological recursion immediately apply to the open-closed GW theory of $\mathcal{X}$ thanks to the proof of the Remodeling Conjecture \cite{BKMP2009, FLZ16}.
 That is, the open and closed GW potentials produced from topological recursion
 are pull-backs of meromorphic quasi-Jacobi forms along the Abel-Jacobi map and quasi-modular forms, respectively.
 Part of the main results is phrased as follows.

\begin{thm}[Theorem \ref{thmholhighergenusWgn}]
Let   $\mathcal{X}$ be the resolution of the toric CY 3-fold $\mathbb{C}^{3}/\mathbb{Z}_{6}$.
Let $\pi:\mathcal{C}\rightarrow S$ be its compactified mirror curve family of genus two with generic fiber $C$.
		The open GW potentials $\omega_{g,n},2g-2+n>0, n>0$ are pull-backs of meromorphic quasi-Jacobi forms along the Abel-Jacobi map
		$\phi:C\rightarrow \mathrm{Pic}^{0}(C)$ based at a Weierstrass point on $C$.
		The closed GW potentials $F_{g}=\omega_{g,0}, g\geq 2$ are quasi-modular forms.
			\end{thm}

Our results not only offers an effective way to package the open- and closed- GW invariants in terms of nice
quasi-Jacobi forms and quasi-modular forms, but also
provides a promising tool in analyzing the global behavior (such as the singularities and the domain of convergence) of the open-closed GW potentials.
Furthermore, it offers some hints about the global property of
K\"ahler moduli space and of the open-moduli space in open GW theory whose rigorous definition
has so far resisted various attempts.

\subsection*{Organization of the paper}

In Section \ref{secgenustwocurveandJacobian}
 we review the basics on genus two algebraic curves and the function theory of their Jacobian varieties,
as well as some properties of different families of genus two algebraic curves that will be used in later sections.
Readers who are familiar with these can skip this section.

In Section  \ref{secringofmodularforms} we develop the theory of quasi-modular forms for genus two. We also define a
differential ring of meromorphic quasi-Jacobi forms, which includes the Bergman kernel as a generator and is suited for the purpose of studying topological recursion.

In Section \ref{secapplications} we extend the analysis in \cite{Fang:2018open} to prove the quasi-modularity in topological recursion for genus two algebraic curves.
We apply the discussions to the mirror curve of a specific toric CY 3-fold and
obtain the quasi-modularity of its open-closed GW theory using the Remodeling Conjecture.

Due to the computational nature of topological recursion, we have strived to make the paper self-contained by
collecting some useful facts and computations in the body of the paper and in Appendix \ref{secdependenceappendix} and Appendix \ref{appendixqgeometricperspective}.

\subsection*{Acknowledgments}

We would like to thank Bohan Fang for collaboration and discussions at the early stage of this work, and
Chiu-Chu Melissa Liu for helpful communications.
The bulk of this work was done during Y.~R.'s tenure at University of Michigan and he is forever grateful for the support and wonderful environment  in the department.
Y. ~R. would like to thank Albrecht Klemm for many stimulating discussion on the topics over the years.
He is partially supported by NSF grant DMS 1807079 and NSF FRG grant DMS 1564457.
J.~Z. would like to thank Kathrin Bringmann and Michael Mertens for discussions on modular forms.
J.~Z. is partially supported by a start-up grant from Tsinghua University,  the young overseas high-level talents introduction plan of China, and national key research and development program of China (NO. 2020YFA0713000).
Part of J.~Z.'s work was done while he was a postdoc at the
Mathematical Institute of University of Cologne and
was supported by German Research Foundation Grant
CRC/TRR 191.

\section{Review on genus two curves and their Jacobians}
\label{secgenustwocurveandJacobian}

We review some basics about genus two curves and their Jacobians, following
\cite{Mumford:1983tata, Mumford:1984tata, Grant:1985theta, Grant:1990formal, Mumford:1999curves, Birkenhake:2013complex, Griffiths:2014principles}.

\subsection{Hyperellipticity of genus two algebraic curves and their Jacobians}
\label{sechyperellipticJacobian}

Let $C$ be a smooth complex curve of genus $g$.
Let $C^{(n)}$ be the $n$th symmetric product of $C$.
Consider the map
\begin{equation}\label{eqnrhomap}
\rho_{n}: C^{(n)}\rightarrow  \mathrm{Pic}^{n}(C)\,,
\quad
(p_{1},p_{2},\cdots,p_{n} )\mapsto \mathcal{O}(p_{1}+p_{2}+\cdots +p_{n})\,.
\end{equation}
When $n=g$, $\rho_{n}$ gives a birational map. When $n=g-1$, the image of $\rho_{n}$ is a divisor.

Fix a degree $n$ divisor $D_{n}=nc$ on the curve $C$ for some point $c$ on $C$.
Consider the
Abel-Jacobi map
\begin{equation}\label{eqnAbelJacobimap}
\alpha_{n}: C^{(n)}\rightarrow J(C):= \mathrm{Pic}^{0}(C)\,,
\quad
(p_{1},p_{2},\cdots,p_{n} )\mapsto \mathcal{O}(p_{1}+p_{2}+\cdots +p_{n})\otimes \mathcal{O}(-D_{n})\,.
\end{equation}
When $n=1$, one can show that $\alpha_{1}$ is an embedding.
When $n=g-1$, Riemann's theorem says that one has
\begin{equation}
\rho_{g-1}(C^{(g-1)})=\kappa+\Theta
\end{equation}
for some theta divisor $\Theta$ on $J(C)$ defining a principal polarization on $J(C)$ and a line bundle $\kappa\in \mathrm{ Pic}^{g-1}(C)$.
If $\Theta$ is symmetric, that is $(-\mathbb{1})^{*}\Theta=\Theta$ where $-\mathbb{1}$ is the involution on
$J(C)$, then $\kappa$ is a theta characteristic satisfying $\kappa^2=K_{C}$.
See the textbooks \cite{Birkenhake:2013complex, Griffiths:2014principles} for more details on this.\\

We now restrict ourselves to the genus two case which is the main interest of this work.
It is a classical result that any genus two curve $C$ is hyperelliptic.
In fact, the curve can be realized as
\begin{equation}\label{eqnhyperellipticformfromcanonicalmap}
Y^{2}= F(X,Z):=\sum_{k=0}^{6}a_{k}X^{6-k}Z^{k}\,.
\end{equation}
for some degree $6$ polynomial $F$ in $X,Z$ with $a_{0}\neq 0$.
The polynomial $F$ can be factored to be
\begin{equation}
F(X,Z)=a_{0}\prod_{k=0}^{5} (X- r_{k}Z)\,.
\end{equation}
where $r_{k}, k=0,\cdots ,5$
are the roots.
In the above coordinates, the hyperelliptic cover $C\rightarrow \mathbb{P}^{1}$ is given by
$[X,Y,Z]\mapsto [X,Z]$.
The ramification points $[X,Y, Z]=[r_{k},0, 1]$ for the hyperelliptic cover have intrinsic meaning: they are exactly the Weierstrass points on the hyperelliptic curve $C$.

By applying a $\mathrm{PGL}_{2}(\mathbb{C})$-action on the base $\mathbb{P}^{1}$ of the above hyperelliptic cover,
we can assume that $r_{0}=\infty$.
Then the affine curve $C-\{r_{0}\}$ can be represented by an equation
\begin{equation}\label{eqngenustwomodelquintic}
y^2=f(x;b):=x^5+b_{1}x^4+b_{2}x^3+b_{3}x^2+b_{4}x+b_{5}=\prod_{k=1}^{5}(x-e_{k})\,.
\end{equation}
The hyperelliptic involution, denoted by  $\sigma$ below, is given by
\begin{equation}\label{eqnhyperellipticinvolution}
\sigma: C\rightarrow C\,,\quad (x,y)\mapsto (x,-y)\,.
\end{equation}
The model \eqref{eqngenustwomodelquintic} remains in its form under the change of variables
\begin{equation}\label{eqntransformationpreservingquintic}
x\mapsto s^2 x+r\,, \quad y\mapsto s^5 y\,,
\end{equation}
where $s\neq 0,r$ are complex numbers.
The curve $C$ is the normalization of the projective closure of $C-\{r_{0}\}$ in $\mathbb{P}^2$.
Its projective algebraic structure is the one obtained by gluing another patch
with $C-\{r_{0}\}$ in the usual way, see e.g.  \cite[Corollary 2.15]{Grant:1990formal}.
Alternatively, the curve can be embedded into the weighted projective space
$\mathbb{WP}^2 [1,3,1]$.
In the corresponding weighted homogeneous coordinates the point $r_{0}$, which we shall also regard it as a root of
the quintic $f$ in \eqref{eqngenustwomodelquintic}
 and denote by $e_{0}$, is given by $[1,0,0]$.

In the current genus two case, the map $\rho_{2}: C^{(2)}\rightarrow \mathrm{Pic}^2(C)$ reviewed in \eqref{eqnrhomap}
is the map that blows down the locus of unordered pairs
\begin{equation}
\{(p, \sigma(p)) | p\in C\}
\end{equation}
 to $K_{C}$.
The curve itself is embedded into $\mathrm{Pic}^2(C)$ via the map $\rho_{1}$.\\

We now make the Abel-Jacobi map in \eqref{eqnAbelJacobimap} more explicit
for the quintic model \eqref{eqngenustwomodelquintic}, following  \cite{Grant:1990formal}.
We take the following basis for the space $H^{0}(C,K_{C})$ of the 1st kind Abelian differentials
\begin{equation}\label{eqnbasisforfirstkind}
\omega_{1}={dx\over 2 y}\,,
\quad
\omega_{2}={xdx\over 2 y}\,.
\end{equation}
There is a standard and canonical way to pick a symplectic basis $\{A_1,A_2,B_1,B_2\}$ for $H_1(C,\mathbb{Z})$
, with the dual basis $\{\alpha_{1},\alpha_{2},\beta_{1},\beta_{2}\}$ for $H^{1}(C,\Z)$.
We shall denote these bases by  $\omega=(\omega_{1},\omega_{2})^{t}, (A,B)$ for simplicity.
Using these data one can form the period matrix
\begin{equation}\label{eqnperiodmatrixtheta}
\begin{pmatrix}
\int_{A_{1}} \omega_{1} & \int_{A_{2}} \omega_{1} &  \int_{B_{1}} \omega_{1} &
\int_{B_{2}} \omega_{1} \\
\int_{A_{1}} \omega_{2} & \int_{A_{2}} \omega_{2} &  \int_{B_{1}} \omega_{2} &
\int_{B_{2}} \omega_{2}
\end{pmatrix}
=
\begin{pmatrix}
\Pi_{A}(\omega) , \Pi_{B}(\omega)
\end{pmatrix}\,.
\end{equation}
It is a standard fact that $\Pi_{A}(\omega)$ is nondegenerate for the smooth curve $C$.
We shall then use the familiar notation for the normalized matrix
\begin{equation}\label{eqndefinitiontauelliptic}
\tau=\Pi_{A}(\omega)^{-1}\Pi_{B}(\omega)\,.
\end{equation}
It lies on the upper-half space $\mathcal{H}$ consisting of symmetric $2\times 2$ complex matrices with positive definite imaginary parts.

Let $\Lambda_{\tau}=\mathbb{Z}^2\oplus\tau\mathbb{Z}^2$ be the lattice in $\mathbb{C}^2$.
The Jacobian $J(C)$ is then identified with the Albanese variety $\mathbb{C}^2/\Lambda_{\tau}$  via integration with respect to the basis
$\omega_{1}, \omega_{2}$.
Choose the reference point $c$ in the definition of the Abel-Jacobi map in  \eqref{eqnAbelJacobimap} to be $e_{0}=\infty$. Then $\mathcal{O}_{C}(2e_{0})=K_{C}$.
That is, $e_{0}$ gives a theta-characteristic.
Then we have a morphism from $C^{(2)}$ to the Jacobian $\mathbb{C}^2/\Lambda_{\tau}$ via
\begin{equation}\label{eqnAJsymmetrictensorpower}
\Phi: (p_{1},p_{2})\mapsto
u=
\begin{pmatrix}
u_{1}\\
u_{2}
\end{pmatrix}=
\begin{pmatrix}
(\int_{\infty}^{p_{1}} +\int_{\infty}^{p_{2}} )\omega_{1}\\
(\int_{\infty}^{p_{1}} +\int_{\infty}^{p_{2}} )\omega_{2}
\end{pmatrix}\,.
\end{equation}

The curve $C$ is embedded into $J(C)$ by
\begin{equation}\label{eqnAJ}
\phi: p\mapsto \Phi(p, \infty)\,.
\end{equation}
By construction, the image $\Theta$ gives a principal polarization on $J(C)$.
The line bundle $\mathcal{O}_{J(C)} (\Theta)$ is ample and the linear system $|m\Theta|,m\gg 0$ gives an embedding
for the principally polarized Abelian variety $(J(C),\Theta)$. In fact, the latter is embedded into a projective space
when $m\geq 3$, see \cite{Mumford:1983tata}.

The holomorphic section of $\mathcal{O}_{J(C)} (\Theta)$, unique up to multiplication by a constant, is given by a theta function with certain characteristic.
Theta functions with characteristics $(a,b)$ are defined on the Jacobian $\mathbb{C}^2/\Lambda_{\tau}$
in the usual way, see \cite{Mumford:1983tata},
\begin{equation}\label{eqndefthetaab}
\vartheta_{(a,b)}(v,\tau)
=\sum_{n\in \mathbb{Z}^{2}} e^{\pi i (n+a)^{t} \tau (n+a) +2\pi i (n+a)^{t}(v+b) }\,,
\quad
a,b\in \mathbb{Q}^{2}\,.
\end{equation}
They are Jacobi forms with multiplier systems \cite{Eichler:1984} which enjoy nice transformation properties.
We refer the interested readers to our earlier work \cite{Fang:2018open} for a short exposition to its definitions and properties
which should be enough for the purpose of this work.

It turns out that with the above particularly chosen
frame \eqref{eqnbasisforfirstkind}
and carefully selected marking $(A,B)$, one has
\begin{equation}
\label{eqnspecialtheta}
\Theta=(\vartheta_{\delta})\,,
\quad
\delta=\left(
\begin{pmatrix}
{1\over 2}\\
{1\over 2}
\end{pmatrix},
\begin{pmatrix}
1\\
{1\over 2}
\end{pmatrix}
\right)\,.
\end{equation}
A different choice of the base point $c$ from the set of Weierstrass points corresponds to a different theta function with an odd characteristic from a total number of
$6$.
The above carefully chosen one
is convenient in performing many computations.

For any such choice of $c$,
the ramification points on $C$ are mapped to $2$-torsion points on $J(C)$ by construction.
In particular, the point
$e_{0}:=\infty$ is sent to the origin in $J(C)$.
These are the only $2$-torsion points lying on $\phi(C)$. In fact, the images generate the whole
group of 2-torsion points on $J(C)$.
See \cite{Grant:1985theta, Grant:1990formal} for detailed computations on the corresponding theta-characteristics.

\subsection{Weierstrass $\sigma$-function and functional field of hyperelliptic Jacobian}
\label{subsecfunctionalfieldofJacobian}

For later use, we also review some basics of hyperelliptic $\sigma$-functions for the genus two
curve in \eqref{eqngenustwomodelquintic}
following \cite{Grant:1990formal}.\\

We take the following 2nd kind Abelian differentials on $C$ due to Baker \cite{Baker:1907introduction}
\begin{equation}\label{eqn2ndkindframe}
\eta_{1}={ (3x^3+2b_{1} x^2+ b_{2}x)dx\over 2y} \,,
\quad
\eta_{2}={ x^2dx\over 2y} \,.
\end{equation}
These differentials\footnote{Strictly speaking, the above expressions
only describe the corresponding cohomology classes they represent in $H^{1}(C,\mathbb{C})$
in an affine patch of the curve: the full description can be obtained by
tracing the de-Rham C$\check{e}$ch resolution for the constant sheaf
via the machinery of the algebraic de-Rham cohomology.} are holomorphic except at $\infty$.
The differentials $\eta=(\eta_{1},\eta_{2})^{t}$ are dual to the ones $\omega$ in \eqref{eqnbasisforfirstkind} under the Poincar\'e residue pairing
\begin{equation}\label{eqnPoincarepairing}
\langle \omega, \eta \rangle=2\pi i\, \mathbb{1} \,.
\end{equation}
We organize the integrals over the $A$-cycles
by
\begin{equation}
\Pi_{A}(\eta)=
\begin{pmatrix}
\int_{A_{1}} \eta_{1} & \int_{A_{2}} \eta_{1}\\
\int_{A_{1}} \eta_{2} & \int_{A_{2}} \eta_{2}
\end{pmatrix}\,.
\end{equation}
The entries are called quasi-periods.

The 2-dimensional Weierstrass $\sigma$-function is defined by
\begin{equation}\label{eqndefinitionofsigma}
\sigma(u, \Pi_{A}(\omega),\Pi_{B}(\omega))=e^{-{1\over 2} u^t \Pi_{A}(\eta)\Pi_{A}(\omega)^{-1} u} \vartheta_{\delta}(\Pi_{A}(\omega)^{-1} u, \Pi_{A}(\omega)^{-1}\Pi_{B}(\omega))\,.
\end{equation}
Similar to the genus one curve, we define the 2-dimensional Weierstrass $\wp$-functions by
\begin{equation}\label{eqndefinitionofwp}
\wp_{ij}(z)=-\partial_{u_i}\partial_{u_j}\log\sigma(u, \Pi_{A}(\omega),\Pi_{B}(\omega))\,.
\end{equation}
One can consider furthermore higher derivatives such as $\wp_{ijk}=\partial_{u_{k}} \wp_{ij}$. \\

Let $k=\mathbb{C}$ and denote $\partial_u:=\{\partial_{u_{1}} ,\partial_{u_{2}}\}$ and $k_{\partial_{u}}\langle \wp_{ij}\rangle$
to be the differential closure of $k(\wp_{ij})$ obtained by adjoining all derivatives of $\wp_{ij}$ under $\partial_{u}$.
One then has the following result.
\begin{thm}[\cite{Grant:1985theta, Grant:1990formal, Buchstaber:1997, Onishi:1998complex}]\label{eqndifferentialfunctionfieldofhyperellipticJacobian}
The differential field $k_{\partial_{u}}\langle \wp_{ij}\rangle$
is finitely generated by $\wp_{ij}, \wp_{ijk}$.
Moreover, derivatives of $\wp_{ijk}$ are polynomials in the generators
$\wp_{ij}, \wp_{ijk}$.
\end{thm}
In fact, derivatives of $\wp_{ij}$ give rise to bases of the Riemann-Roch spaces $\{\mathcal{L}(m\Theta)\}_{m}$. For example, one has
\begin{eqnarray}\label{eqnbasisformTheta}
\mathcal{L}(\Theta)&=&\mathbb{C}\,,\nonumber\\
\mathcal{L}(2\Theta)/\mathcal{L}(\Theta)&=&\mathbb{C}\wp_{11}\oplus \mathbb{C}\wp_{12}\oplus \mathbb{C}\wp_{22}
\,,\nonumber\\
\mathcal{L}(3\Theta)/\mathcal{L}(2\Theta)&=&\mathbb{C}\wp_{111}\oplus \mathbb{C}\wp_{112}\oplus \mathbb{C}\wp_{122}
\oplus \mathbb{C}\wp_{222} \oplus \mathbb{C} (\wp_{11}\wp_{22}-\wp_{12}^2)\,.
\end{eqnarray}
The complete linear system $|3\Theta|$ defines a projective embedding of $J(C)$ into $\mathbb{P}^{8}$.
The set of defining equations are worked out in \cite{Grant:1990formal}.
As a consequence,
the differential field $k_{\partial_{u}}(J(C))$ of the functional field $k(J(C))$ is finitely generated by the generators in
$\mathcal{L}(3\Theta)$.
See \cite{Grant1988:generalization, Grant:1990formal, Buchstaber:1997, Onishi:1998complex, Buchstaber:2012multi} for some examples
where the
4th derivatives
are expressed in terms of polynomials of the generators in $\mathcal{L}(3\Theta)$.

\begin{rem}\label{remprincipalpartofBergmankernel}
An alternative explanation of the above statement is as follows.
From the $\vartheta_{\delta}$ or
$\sigma$-function above, one can construct the prime form and hence the bifundamental in the usual way.
Particularly for hyperelliptic curves, there is another algebraic way of constructing
this bifundamental (called the Klein bifundamental). The result then follows by comparing
the Laurent expansions of the bifundamental constructed in these two different ways.
See \cite{Baker:1898hyperelliptic, Buchstaber:1996hyperelliptic, Buchstaber:1997, Onishi:1998complex, Matsutani:2001hyperelliptic, Fay:2006theta, Buchstaber:2012multi} for related discussions.
The hyperellipticity will be essential in our work in studying topological recursion, as we shall  explain below in Remark \ref{remhyperellipticity}.
\end{rem}

Since $C^{(2)}$ is birational to $J(C)$, the pullback of rational functions in $k(J(C))$ are rational functions on $C\times C$ that are invariant under permutation.
It is shown in \cite{Grant:1985theta} that
\begin{equation}\label{eqnfunctionalfieldJacobian}
k(J(C))=\mathbb{C}(x_{1}+x_{2},x_{1}x_{2}, y_{1}+y_{2} )\,.
\end{equation}
Explicit formulae for $\wp_{ij},\wp_{ijk}\in k(J(C))\subseteq k_{\partial_{u}}(J(C))$ in terms of the generators
$x_{1}+x_{2},x_{1}x_{2}, y_{1}+y_{2} $ can be found in \cite{Grant:1985theta, Grant:1990formal}.
For example. one has
\begin{equation}\label{eqnrelationbetweenwpijandrationalfunctions}
\wp_{22}=(x_{1}+x_{2})\,,
\quad
\wp_{12}=-x_{1}x_{2}\,,
\quad
\wp_{11}= {G(x_{1},x_{2})-2y_{1}y_{2}\over (x_{1}-x_{2})^2}
\,,
\end{equation}
where
\begin{eqnarray}\label{eqnprincipalpartofBergmankernel}
G(x_{1},x_{2})&=&(x_{1}+x_{2})(x_{1}x_{2})^2+2b_{1}(x_{1}x_{2})^{2}+b_{2}(x_{1}+x_{2})(x_{1}x_{2})\nonumber\\
&&+2b_{3} (x_{1}x_{2}) +b_{4}(x_{1}+x_{2})+2b_{5}\,.
\end{eqnarray}

The pull-back of $\wp_{ij}$ to the curve $C$ along $\phi$, which is essentially equivalent to the restriction
to $\phi(C)$, are discussed in \cite{Grant1988:generalization, Grant:1990formal, Grant1991:generalization}.
This amounts to replacing the variables $(x_2,y_2)$ by the coordinates for $\infty$
which then reduces the map $\Phi$ in \eqref{eqnAJsymmetrictensorpower} to the map $\phi$ in \eqref{eqnAJ}.
In particular, it is found that
\begin{equation}\label{eqnexyintermsofwpij}
{\wp_{22}\over \wp_{12}}|_{\phi(C)}={\sigma_{2}\over \sigma_{1}}|_{\phi(C)}=-{1\over x}\,,
\quad
y=-{\wp_{11}\wp_{22}-\wp_{12}^2\over 2 \wp_{222}}|_{\phi(C)}\,.
\end{equation}

The Laurent expansions of $\sigma$ in terms of the algebraic local uniformizers made from $x,y$,
and the analytic coordinates from $u_1,u_2$ on $\phi(C)\cong C$  can be found in \cite{Grant:1990formal, Onishi:1998complex, Onishi:2002determinant}.
In fact, it is found that around a point $p$ a local uniformizer can be taken to be
$x-x(p)$ if $p\neq e_k$;
$y$ if $p=e_{k},k\neq 0 $; $1/x^{1\over 2}$ if $p=e_0$.
Also away from $e_{0}$, the analytic coordinate $u_{1}-u_{1}(p)$ serves as a local uniformizer,
and around $e_0$ the analytic coordinate $u_{2}$ does so.\\

The dependence of all of the above constructions on the choice of marking $(A,B)$ and the choice of frame $\omega$
are discussed in \cite{Grant:1985theta}.
This will be useful in our later discussions on quasi-modularity and
some necessary materials are collected in Appendix \ref{secdependenceappendix}.

\subsection{Families of genus two curves}
\label{secfamiliesofgenustwocurves}

\subsubsection{Scalar-valued Siegel modular forms}

It is known since Igusa \cite{Igusa:1962siegel, Igusa:1964siegel} that the ring of scalar-valued Siegel modular forms for
$\Gamma(1)=\mathrm{Sp}_{4}(\mathbb{Z})$ is generated by the Eisenstein series
\begin{equation}
E_{2k}(\tau)=\sum_{(c,d)=1} \det (c\tau+d)^{-2k}\,,\quad k\geq 2\,,
\end{equation}
which are defined as the summation over all coprime symmetric pairs.
They can be regarded as holomorphic sections of
$K_{\mathcal{A}}^{\otimes k}$, where $K_{\mathcal{A}}$
is the canonical bundle of the Siegel modular variety
$\mathcal{A}:=\Gamma(1)\backslash\mathcal{H}$.

The following normalized cusp forms (that is, those holomorphic ones vanishing at the boundary of
$\mathcal{A}$) $\chi_{10},\chi_{12}$ are convenient
\begin{eqnarray}
\chi_{10}
&=&
-43867\cdot 2^{-12}3^{-5}5^{-2}7^{-1}53^{-1}(E_{4}E_{6}-E_{10})\,,\nonumber\\
\chi_{12}&=&
131\cdot 593\cdot 2^{-13}3^{-7}5^{-3}7^{-2} 337^{-1} (3^2 7^2 E_{4}^3+2\cdot 5^2 E_{6}^2-691 E_{12})\,.
\end{eqnarray}
In the above we have omitted the argument $\tau$ in the Siegel modular forms.
It is shown in  \cite{Igusa:1962siegel, Igusa:1964siegel} that the ring $M(\Gamma(1))$ of scalar-valued Siegel modular forms of genus two for the Siegel modular group $\Gamma(1)$
is given by the ring with
five generators subject to one relation
\begin{equation}
M(\Gamma(1))=
\mathbb{C}[E_{4}, E_{6}, \chi_{10}, \chi_{12}, \chi_{35}]/ \chi^{2}_{35}=\chi_{10} \cdot P(E_{4}, E_{6}, \chi_{10}, \chi_{12})\,,
\end{equation}
where all of the generators and the polynomial $P$
are explicit in terms of Eisenstein series given in \cite{Igusa:1962siegel, Igusa:1964siegel}.
The relations of the above  Eisenstein series to theta constants are also worked out in \cite{Igusa:1967modular}.
The ring of modular forms for a class of congruence subgroups
of $\Gamma(1)$ are shown to be generated by theta constants \cite{Igusa:1964siegel}.

Limits of these Siegel modular forms
around singular points on the Siegel modular variety $\Gamma(1)\backslash \mathcal{H}$ are
also discussed \cite{Igusa:1962siegel}.
This will potentially be useful for the later purpose of
studying the enumerative content of the Gromov-Witten potentials near singular points in the moduli space.

\subsubsection{Binary invariants, Igusa absolute invariants, and Torelli map}

Let $\mathcal{M}$ be the moduli space of smooth sextics.
By assigning the (canonical) principal polarized Jacobian $J(C)$ to a smooth genus two curve $C$ defined by $Y^{2}=F(X;a)$
in \eqref{eqnhyperellipticformfromcanonicalmap}, we obtain the Torelli map
from the moduli space $\mathcal{M}$ of smooth genus two curves
to the Siegel modular variety $\mathcal{A}=\Gamma(1)\backslash \mathcal{H}$.
Recall that the Siegel modular variety is the moduli space of principally polarized Abelian varieties of dimension 2.
The Torelli theorem states that the Torelli map
 is an injective, birational map.
In particular rational functions on the moduli space $\mathcal{M}$ are identified with modular functions for the Siegel modular group $\Gamma(1)$.\\

Consider a family of genus two curves given by \eqref{eqnhyperellipticformfromcanonicalmap} which in inhomogenized coordinates of
$\mathbb{WP}^2[1,3,1]$ is given by
\begin{equation}\label{eqnhyperellipticsextic}
Y^{2}= F(X;a)=\sum_{k=0}^{6}a_{k}X^{6-k}=a_{0}\prod_{k=0}^{5} (X- r_{k})\,,
\quad
a_{0}\neq 0\,.
\end{equation}
The binary invariants $A,B,C,D$ for the sextic $F(X;a)$ are defined in terms of the roots of $F(x;a)$, see \cite{Igusa:1962siegel} (also \cite{Igusa:1967modular}).
For example, $D$ is the discriminant $a_{0}^{10}\prod_{0\leq i <j\leq 5}(r_{i}-r_{j})^2$.
They can be expressed in terms of the coefficients
$a_{k}, k=0,1,\cdots,6$
of the sextic $F(x;a)$.
For example, one has
\begin{equation}
A=6a_{3}^2-16 a_{2}a_{4}+40 a_{1}a_{5}-240 a_{0}a_{6}\,,\quad
\mathrm{Resultant}(F,\partial_{X}F)=-a_{0}D\,,
\end{equation}
where $Resultant$ represents the resultant.
Here we omit the detailed expressions and refer the interested readers to \cite{Klemm:2015direct} for a collection of the  relevant formulae.
The Igusa absolute invariants
are given in terms of the binary invariants by
\begin{equation}\label{eqnabsoluteinvariantsintermsofbinaryinvariants}
j_{1}={2^4 3^2 B\over A^2}\,,
\quad
j_{2}=2^{6} 3^{3} {(3C-AB)\over A^3}
\,,
\quad
j_{3}=2 \cdot 3^{5} {D\over A^5}
\,.
\end{equation}

The Torelli theorem concerns the period map
which assigns to a smooth genus 2 curve with marking its normalized period
$\tau$.
It turns out that the
absolute invariants \eqref{eqnabsoluteinvariantsintermsofbinaryinvariants} are identified \cite{Igusa:1962siegel} with the following modular functions $\mathfrak{j}_{1},\mathfrak{j}_{2},\mathfrak{j}_{3}$ for $\Gamma(1)$,
which are generators of the rational function field of the Siegel modular variety $\Gamma(1)\backslash \mathcal{H}$,
\begin{equation}
\mathfrak{j}_{1}= {E_4 \chi_{10}^2\over \chi_{12}^2}\,,
\quad
\mathfrak{j}_{2}={E_6 \chi_{10}^3\over \chi_{12}^3}
\,,
\quad
\mathfrak{j}_{3}={ \chi_{10}^6\over \chi_{12}^5}
\,.
\end{equation}
Explicit relations between the binary invariants
and the Eisenstein series for the modular group $\mathrm{Sp}_{4}(\mathbb{Z})$, that involve the Jacobian
from the $\tau$-variables to the absolute invariants,
are worked out in \cite{Igusa:1962siegel}. In particular, the relations between the absolute invariants
and Eisenstein series are laid out there.\\

We now look at the singular loci of the space of sextics in the model \eqref{eqnhyperellipticsextic}.
As explained in \cite{Igusa:1960arithmetic}, the binary invariants $B,C,D$ (as symmetric polynomials in the roots) vanish simultaneously at sextics with triple roots,
all of the genus two curves are mapped to the same point in $\mathcal{M}$. Blowing up this point then recovers the space parametrizing genus two curves with $A\neq 0$
and their degenerations.
From the perspective of $\mathcal{A}$, a Jacobian variety corresponding to a point on
$H_{1}:=\{\tau_{12}=0\}$, called the Humbert surface of degree $1$,
would satisfy $\chi_{10}=0, \chi_{12}\neq 0$
and
$[A,B,C,D]=[1,0,0,0]$.
It is also known that points on the locus $H_{4}=\{\tau_{11}=\tau_{22}\}$ in $\mathcal{A}$, called the Humbert surface of degree $4$, correspond to sextic curves
which have extra automorphism groups.
Hence this Humbert surface serves as the orbifold loci.

\subsubsection{Other models: quintic and Rosenhain normal form}
\label{secothermodels}

Assume that the discriminant $D$ of the curve in \eqref{eqnhyperellipticsextic} is non-vanishing and that $r_{0}\neq 0$.
The coordinate transformation
\begin{equation}\label{eqntransformationfromsextictoquintic}
X={ x+c_1\over x+c_2} r_{0}\,,
\quad
Y=  (a_{0} r_{0} (c_1-c_2)\prod_{k=1}^{5} (r_{0}-r_{k}))^{1\over 2} b_{0}^{-{1\over 2}}\cdot {y\over (x+c_2)^{3}} \,, b_0\neq 0\,,
\end{equation}
where $c_1,c_2$ are undetermined coefficients,
changes \eqref{eqnhyperellipticsextic} to the form in \eqref{eqngenustwomodelquintic}
\begin{equation} \label{eqnhyperellipticquintic}
y^{2}=b_0 \prod_{k=1}^{5}(x+ {c_1 r_0-c_2 r_k \over r_0-r_k}):=b_0 x^5+\sum_{k=1}^{5} b_{k}x^{5-k}=f(x;b)\,.
\end{equation}
The binary invariants in terms of the coefficients for the general quintic case
are  studied in \cite{Igusa:1960arithmetic, Krishnamoorthy:2005invariants}
(see also \cite{Grant1994:units} for the special case $b_0=1$).
Relations between theta constants and branch points are given by Rosenhain's and Thomae's formula,
see for example \cite{Mumford:1983tata, Grant:1985theta, Grant1988:generalization, Buchstaber:1997, Enolski:2007periods, Eilers:2018rosenhain}.\\

Now we apply the results in
\cite{Grant:1985theta, Grant:1990formal} to the curve in \eqref{eqnhyperellipticquintic}, which are reviewed in Section \ref{subsecfunctionalfieldofJacobian}.
 This yields explicit expressions for the rational functions $x,y$ on the curve \eqref{eqnhyperellipticquintic}
 and hence those on the curve \eqref{eqnhyperellipticsextic}
   in terms of pull-back of rational functions in
$\wp_{ij},\wp_{ijk}$ which are Jacobi forms.
%%%
\iffalse
Without changing the curve from the form in \eqref{eqnhyperellipticsextic} to the form in \eqref{eqnhyperellipticquintic},
one can still conclude that $X,Y$ are rational functions in $\wp_{ij},\wp_{ijk}$
since the latter generates the rational function field of the curve.
But explicit expressions in terms of pull-backs of Jacobi forms are not as direct to obtain.\\
\fi
%%%

As reviewed earlier, the model in \eqref{eqnhyperellipticquintic} remains its form under the transformation
\eqref{eqntransformationpreservingquintic}.
Note that this transformation is not same as the one discussed in
\eqref{eqntransformationofHodgeframeappendix} that changes the basis $\omega$ of the 1st kind  Abelian differentials.
A transformation in \eqref{eqntransformationpreservingquintic}
amounts to changing the values for $c_{1},c_{2}$ in the coordinate transformation
\eqref{eqntransformationfromsextictoquintic}.
Insisting that the equation takes the following Rosenhain normal form
\begin{equation} \label{eqnhyperellipticRosenhain}
y^{2}=x(x-1)(x-\lambda_1)(x-\lambda_2)(x-\lambda_3)
\end{equation}
fixes the transformation \eqref{eqntransformationpreservingquintic} up to permutations of the roots.
One such choice is such that the overall transformation in \eqref{eqntransformationfromsextictoquintic}
is given by
\begin{equation}
X=r_{0}{{r_{4}-r_{2}\over r_{4}-r_{0}} x -{r_{2}\over r_{0}}\over  {r_{4}-r_{2}\over r_{4}-r_{0}} x -1}\,.
\end{equation}
That is,
\begin{equation}\label{eqnquintictoRosenhain}
c_{1}=-{r_{2}\over r_{0}}\cdot {r_{4}-r_{0}\over r_{4}-r_{2}}\,,
\quad
c_{2}=-1\cdot {r_{4}-r_{0}\over r_{4}-r_{2}}\,.
\end{equation}
The roots $r_{k},k=2,4,0$ in the $X$-coordinate get changed to $0,1,\infty$ in the $x$-coordinate, while the rest of the roots $r_{1},r_{3},r_{5}$
become
\iffalse
\begin{equation}
\lambda_k=r_{0}{{r_{4}-r_{2}\over r_{4}-r_{0}} r_{2k-1} -{r_{2}\over r_{0}}\over  {r_{4}-r_{2}\over r_{4}-r_{0}} r_{2k-1} -1}\,,\quad k=1,2,3\,.
\end{equation}
\fi
\begin{equation}
\lambda_k={r_4-r_0\over r_4-r_2}\cdot{{r_{2k-1}}-{r_2}\over r_{2k-1}-{r_0}}\,,\quad k=1,2,3\,.
\end{equation}
Explicit relations between the binary invariants and the parameters $\lambda_{k},k=1,2,3$
can be found in \cite{Igusa:1960arithmetic}.
See also \cite{Krishnamoorthy:2005invariants} for a nice summary on the explicit formulae.
The works  \cite{Malmendier:2016universal, Malmendier:2017satake} also collect formulae about these relations.\\

It turns out that the Rosenhain normal form gives the universal family of smooth genus two curves with level two structures.
To be more precise, for any ordered triple $(\lambda_{1},\lambda_{2},\lambda_{3})$, where $\lambda_{1},\lambda_{2},\lambda_{3}$
are all distinct and are distinct from $0,1,\infty$,
the curve \eqref{eqnhyperellipticRosenhain}
determines a point in the coarse moduli space $\mathcal{M}(2)$
of genus two curves with level two structures.
The period matrix $\tau$ gives a point in the Siegel modular variety
$\mathcal{A}(2):=\Gamma(2)\backslash \mathcal{H}$,
here
\begin{equation}
\Gamma(2):=\{\gamma=
\begin{pmatrix}
a & b\\
c & d
\end{pmatrix}\in \Gamma(1)=\mathrm{Sp}_{4}(\mathbb{Z})\,|\,
\gamma\equiv
\begin{pmatrix}
\mathbb{1} & 0\\
0 & \mathbb{1}
\end{pmatrix}
\,\mathrm{mod}\,2\}\,.
\end{equation}
Conversely, for any point $\tau\in \mathcal{A}(2)$, there is a genus two curve
in the Rosenhain normal form  whose normalized period is $\tau$.
In particular, the parameters
$\lambda_{k},k=1,2,3$ are modular functions for the modular subgroup $\Gamma(2)$
whose field of rational functions consists of
rational functions in the level 2 theta functions
with characteristics \cite{Mumford:1983tata}.

It is a standard fact that the modular subgroup $\Gamma(2)$
fixes the theta characteristics and leaves $\lambda_{k}$'s invariant, and that the group $\Gamma(2)\backslash \Gamma(1)$
is the Galois group for the field extension $\mathbb{C}(\lambda_{1},\lambda_{2},\lambda_{3})/\mathbb{C}(j_{1} ,j_{2},j_{3})$.
The forgetful map $\mathcal{M}(2)\rightarrow \mathcal{M}$
is a Galois cover of degree $|\mathfrak{S}_{6}|=720$, where $\mathfrak{S}_{6}$
acts on the roots of a sextic by permutations and hence induces action on
the ordered triple $(\lambda_{1},\lambda_{2},\lambda_{3})$.
Therefore up to this action, the Rosenhain family \eqref{eqnhyperellipticRosenhain} gives the universal family
over $\mathcal{M}(2)$.

The relations between the parameters $\lambda_{k}$ and rational functions in the level two theta constants
 are determined up to the  Galois action.
One such relation is given by Rosenhain and can be found in
\cite{Igusa:1962siegel},
\begin{equation}\label{eqnRosenhainlambda}
\lambda_{1}=({\theta_{1100} \theta_{1000}\over
\theta_{0100} \theta_{0000}})^2\,,\quad
\lambda_{2}=({\theta_{1001} \theta_{1100}\over
\theta_{0001} \theta_{0100}})^2\,,\quad
\lambda_{3}=({\theta_{1001} \theta_{1000}\over
\theta_{0001} \theta_{0000}})^2\,.
\end{equation}
Here we have followed the convention in \cite{Igusa:1962siegel} for the theta constants, which is related to the one $\vartheta_{(a,b)}$ in \eqref{eqndefthetaab} by
\begin{equation}\label{eqnthetaconvention}
\theta_{xyzw}=\vartheta_{( a,b   )}\,,
\quad
a={1\over 2}(x,y)^{t}\,,\quad  b={1\over 2}(z,w)^{t}\,.
\end{equation}
The work
\cite{Malmendier:2017satake} provides a nice summary on the relations among the
parameters $\lambda_{k}$ in the Rosenhain normal form, Eisenstein series, and theta constants.

\section{Differential rings of quasi-modular and quasi-Jacobi forms for genus two}
 \label{secringofmodularforms}

Following the sheaf language in \cite{Urban:2014nearly},
we shall first define the ring  $\widetilde{\mathcal{R}}$ of quasi-modular forms. It is obtained by
adjoining extra generators to the ring $\mathcal{R}$
of modular forms valued in certain representations.
These extra generators are not modular, but quasi-modular in the sense generalizing the one in \cite{Kaneko:1995}.
Geometrically they are related to the periods of the 2nd kind Abelian differentials called
quasi-periods. The ring of quasi-modular forms turns out to be a differential ring
due to the connection to periods and quasi-periods whose differential closure can be studied via differential Galois theory  \cite{Bertrand:2000transcendence}.
Our  results reveal various facets of quasi-modular forms from different angles such as representation theory, Hodge theory and complex differential geometry.

We then construct a ring $\mathcal{J}$ of meromorphic Jacobi forms consisting of higher dimensional analogues of Weierstrass $\wp$-functions and their derivatives. Due to hyperellipticity this ring exhibits a very  simple  differential ring
structure
 \cite{Grant:1990formal, Onishi:1998complex}.
 Finally we define a ring  $\tilde{\mathcal{J}}$ of meromorphic quasi-Jacobi forms
 to be the $\widetilde{\mathcal{R}}$-module $\widetilde{\mathcal{R}}\otimes \mathcal{J}$.

\subsection{Quasi-modular forms}
\label{secquasimodularforms}

The theory of quasi-modular forms will be developed for universal families of smooth polarized Abelian varieties $\pi:\mathcal{Z}\rightarrow S$
which do not necessarily arise as families of Jacobian varieties attached to families of
genus two curves.\\

We shall only work with universal families in which $S$ is a Siegel modular variety
(more precisely, stack) $\Gamma\backslash \mathcal{H}$,
where $\Gamma$ is a congruence subgroup of ${\rm  Sp}_{4}(\mathbb{Z})$.
The resulting variety $S$ is algebraic and in fact quasi-projective, with the embedding given by theta constants.
The frequently studied cases in the literature are the full modular group
${\rm Sp}_{4}(\mathbb{Z})$ with the embedding given by the Igusa absolute invariants, and
the Igusa modular groups
 $\Gamma(2),\Gamma_{2,4}, \Gamma_{4,8}$ with the theta embeddings given by theta constants of level 4, 2 and 1 respectively.
See \cite{VanderGeer:2008, Clingher:2018normal} for a collection of the definitions of these groups and the corresponding theta constants.

The subtlety of compactification of the modular variety $\Gamma\backslash \mathcal{H}$ is in fact not important for our purpose, we can simply use Baily-Borel if needed.
The reason is that we shall only need to use the functional field of the variety $\Gamma\backslash \mathcal{H}$.
However, interested readers are referred to \cite{VanderGeer:2008, Urban:2014nearly, Liu:2016nearly}
for details on how the various constructions can be extended to the compactification.

We shall also ignore the subtlety on multiplier system in the definition of modular forms:
we may pass to a smaller congruence subgroup
or work freely with algebraic extensions if needed (cf. Lemma \ref{lemmodularityofparameters} below).
Most of the constructions in this section will be independent of the precise choice of $\Gamma$.
For this reason we shall often omit the symbol $\Gamma$ in the notations, if no confusion should arise.

\subsubsection{Quasi-periods: archetype
%: transcendental
}

For the given family $\pi: \mathcal{Z}\rightarrow S=\Gamma\backslash \mathcal{H}$,
the function field $k(S)$
has transcendental degree $3$ over $k=\mathbb{C}$.
We take three modular functions $t_{1},t_{2},t_{3}$
for $\Gamma$ which generate $k(S)$ up to algebraic extension.
We denote these generators collectively by $t$ and similarly the partial derivatives collectively by $\partial_{t}$.

The relevant locally free sheaves/bundles that enter the story are
the cotangent bundle
$\Omega_{S}^{1}$,
 $\underline{\omega}:=R^{0}\pi_{*}\Omega^{1}_{\mathcal{Z}|S}$, and the Hodge bundle
$\mathcal{V}_{S}:=R^{1}\pi_{*}\mathbb{C}_{\mathcal{Z}}\otimes\mathcal{O}_{S}$.
%One can lift these bundles from $S$ to $\mathcal{H}$.
For a universal family,
the Kodaira-Spencer map gives
\begin{equation}\label{eqnKSmorphism}
\Omega_{S}^{1}\cong \mathrm{Sym}^{\otimes 2} \underline{\omega}\,.
\end{equation}
See Appendix \ref{appendix:reps} for details.
The Hodge filtration is the relative version of the following sequence for an Abelian variety
$Z$
\begin{equation}\label{eqnHodgefiltration}
0\rightarrow H^0(Z,\Omega^{1}_{Z})\rightarrow H^{1}_{dR}(Z)\rightarrow H^{1} (Z,\mathcal{O}_{Z})\rightarrow 0\,.
\end{equation}

We fix a  locally constant frame $\{A,B\}$ for the local system $R^{1}\pi_{*}\mathbb{C}_{\mathcal{Z}}$
with dual frame $\{\alpha,\beta\}$.
We also choose a Hodge frame $\{\omega,\eta\}$
adapted for the Hodge filtration for the family $\pi: \mathcal{Z}\rightarrow S$:
$\omega$ is a frame for $\underline{\omega}$, and $\{\omega, \eta\}$ a frame for
$\mathcal{V}_{S}$ satisfying a relation similar to \eqref{eqnPoincarepairing}
\begin{equation}\label{eqndeRhampairing}
\langle \omega, \eta \rangle=2\pi i\, \mathbb{1} \,.
\end{equation}
The parings between elements in the frame $\{\omega,\eta\}$ with those in  $\{A,B\}$
are given by the corresponding period integrals.
In the case where
the family $\pi: \mathcal{Z}\rightarrow S$ does arise as the
family of Jacobian varieties of a genus two curve family $\mathcal{C}\rightarrow S$, these period integrals
coincide with the period integrals for the curves.
This is due to the fact that the Jacobian $J(C)$ of a curve $C$, defined earlier as
$\mathrm{Pic}^{0}(C)$ or the Albanese variety,  is isomorphic to the intermediate Jacobian of $C$.
See \cite{Birkenhake:2013complex} for details.
For this reason, we shall use the same notations
for the matrices of periods and quasi-periods for the family $\pi: \mathcal{C}\rightarrow S$.

Since the Kodaira-Spencer map \eqref{eqnKSmorphism} is an isomorphism for the universal family  $\pi: \mathcal{Z}\rightarrow S$, we can in fact take
\begin{equation}\label{eqnetaisGaussManin}
\eta=\nabla_{v_{0}}\omega
\end{equation} for some $v_{0}\in T_{S}$.
%In particular, the cohomology class of $\eta$ is an $\mathcal{O}_{S}$-linear combination of those of $\omega,\nabla_{v_{0}}\omega$.
The convention for the period matrix and the quasi-period matrix that we adapt in studying modularity is
\begin{equation}\label{eqnperiodmatrix}
\Pi:=\begin{pmatrix}
\Pi_{B}(\omega) & \Pi_{B}(\eta)\\
\Pi_{A}(\omega) & \Pi_{A}(\eta)
\end{pmatrix}
=
\begin{pmatrix}
\int_{B_{1}}\omega_{1} & \int_{B_{1}}\omega_{2} & \int_{B_{1}}\eta_{1} & \int_{B_{1}}\eta_{2} \\
\int_{B_{2}}\omega_{1} & \int_{B_{2}}\omega_{2} & \int_{B_{2}}\eta_{1} & \int_{B_{2}}\eta_{2} \\
\int_{A_{1}}\omega_{1} & \int_{A_{1}}\omega_{2} & \int_{A_{1}}\eta_{1} & \int_{A_{1}}\eta_{2} \\
\int_{A_{2}}\omega_{1} & \int_{A_{2}}\omega_{2} & \int_{A_{2}}\eta_{1} & \int_{A_{2}}\eta_{2}
\end{pmatrix}\,.
\end{equation}
This convention
% (for more details see Appendix \ref{appendixqgeometricperspective})
seems to be more convenient than the one given in Section \ref{secgenustwocurveandJacobian}.
For example, under this convention the action on
$\tau=  \Pi_{B} (\omega)\Pi_{A}(\omega)^{-1}$ is given by
\eqref{eqntautransform}, while the action on
 \eqref{eqndefinitiontauelliptic} would be less appealing.
Straightforward computation by using the Riemann-Hodge bilinear relations \eqref{eqnRiemannHodgePi} tells that the transformation law of $\Pi_{A}(\omega)^{-1}  \Pi_{A}(\eta) $ under $\gamma=(a ,b; c,d)\in \Gamma$
is
\begin{equation}\label{eqntransformationofquasiperiods}
\Pi_{A}(\omega)^{-1}  \Pi_{A}(\eta)\mapsto
\Pi_{A}(\omega)^{-1}  \Pi_{A}(\eta)
+(\Pi_{A}(\omega))^{-1} (c\tau+d)^{-1} 2\pi i c (\Pi_{A}(\omega))^{-t}\,.
\end{equation}
\iffalse If we set $\Pi_{A} (\omega)=2\pi i$ by normalizing $\omega$, we then see that
the quantity $\Pi_{A}(\eta)$
is a quasi-modular form according to our previous definition.
\fi
Henceforward we shall adapt this convention.
Explicit modular transformation rules will be computed in this one, and then transferred to the one
in Section \ref{secgenustwocurveandJacobian} if needed.
See Appendix \ref{secdependenceappendix} and Appendix \ref{appendixqgeometricperspective} for more details.\\

\iffalse If we set $\Pi_{A} (\omega)=2\pi i$ by normalizing $\omega$, we then see that
the quantity $\Pi_{A}(\eta)$
is a quasi-modular form according to our previous definition.
\fi
The normalized quasi-period matrix $\Pi_{A}(\omega)^{-1}  \Pi_{A}(\eta)$ is the archetype of the so-called quasi-modular form that we shall define below.
Furthermore, later in Section \ref{secquasiJacobi}
 we shall see that it is
nothing but the degree zero term in the Laurent expansion of the Bergman kernel defined in a way similar to
\eqref{eqndefinitionofwp}.\\

We now briefly review some explicit formulae for the periods and quasi-periods of hyperelliptic curves in terms of theta constants.

Consider the sextic model in \eqref{eqnhyperellipticformfromcanonicalmap} or \eqref{eqnhyperellipticsextic}, with the marking chosen as before and the frame chosen in a way similar to the ones in
\eqref{eqnbasisforfirstkind} and \eqref{eqn2ndkindframe}
\cite{Baker:1907introduction}.
The quasi-period matrix $\Pi_{A}^{(6)}(\omega)^{-1}\Pi_{A}^{(6)}(\eta)$ for the case $a_{0}=1$ can be computed
explicitly in terms of the coefficients $a_{k}$ and theta constants \cite{Eilers:2016modular}
to be\footnote{The convention in \cite{Eilers:2016modular} is different from the one in \cite{Grant:1990formal} that we utilize
in this work, and we have fixed the constants arising from this difference.
}
\begin{eqnarray}\label{eqnausimodularintermsofthetaderivative}
&&\Pi^{(6)}_{A}(\omega)^{-1}\Pi^{(6)}_{A}(\eta)\nonumber\\
&=&-{1\over 10}
\begin{pmatrix}
4a_{4} & a_{3}\\
a_{3} & 4a_{2}
\end{pmatrix}
+{1\over 10}
\sum_{1\leq i<j\leq 5}
{\begin{pmatrix}
\partial^{2}_{u_1 u_1} &\partial^{2}_{u_1 u_2} \\
\partial^{2}_{u_2 u_1}  & \partial^{2}_{u_2 u_2}
\end{pmatrix}|_{u=0}
\vartheta_{(a,b)} \left( \Pi_{A}^{(6)}(\omega)^{-t}\cdot u\right)\over \vartheta_{(a,b)}|_{u=0} \left( \Pi_{A}^{(6)}(\omega)^{-t}\cdot u \right)}
\,,
\end{eqnarray}
where $(a,b)$ is the theta-characteristic corresponding to
$r_{i}+r_{j}-2r_{0}$.
Here and in what follows the superscript $(6)$ stands for constructions for the sextic model
and similarly $(5)$ for the quintic model when confusion might arise.
The above result for the case with $a_{0}=1$ can be used to recover those for general $a_{0}$ easily.
That the above quasi-period matrix transforms according to \eqref{eqntransformationofquasiperiods} can be proved directly
by working out the transformation of the derivatives of theta functions, as shown in \cite{Grant:1985theta}
which we have included in Appendix \ref{appendixquasiperiods} for completeness.
The expressions for the periods themselves can be obtained by inverting
 Thomae's formula which relates the branch points to certain theta constants.\\

Consider also the quintic model \eqref{eqngenustwomodelquintic} or  \eqref{eqnhyperellipticquintic} with $b_0=1$.
Based on Thomae formula it is proved in  \cite{Grant1988:generalization} that\footnote{
Note that in \cite{Grant1988:generalization}, it is assumed that
$b_{0}=1$ and
the basis $\omega$ therein is $2$ times that in this work.}

\begin{equation}\label{eqndiscriminantintermsoftheta}
\prod_{1\leq k<\ell\leq 5}(e_{k}-e_{\ell})^2=\pm \left(\det ({\Pi_{A}^{(5)}(\omega) \over \pi ^2})\right)^{-10}
\prod_{\nu \,\mathrm{even}}\vartheta^2_{\nu}(0,\tau)\,,
\end{equation}
where $\Pi_{A}^{(5)}(\omega)$
is the matrix of $A$-cycle integrals of the 1st kind Abelian differentials
as explained earlier in Section \ref{sechyperellipticJacobian}, and the product is over the $10$ even theta-characteristics.

Further relations between theta constants and branch points are also given by Rosenhain's and Thomae's formulae,
see for example \cite{Mumford:1983tata, Grant:1985theta, Grant1988:generalization, Buchstaber:1997, Enolski:2007periods, Eilers:2018rosenhain}.
The expressions for the periods and quasi-periods can be found in
\cite{Enolski:2007periods, Eilers:2018rosenhain} for this quintic model.
To be more precise, let
$\epsilon_{k}=\phi(e_{k})-\delta,
\epsilon_{k\ell}=\phi(e_{k})-\delta+\phi(e_{\ell})-\delta+\delta$
and $\theta_{k}=\theta_{\epsilon_{k}}(v,\tau)|_{v=0},
\theta_{k\ell}=\theta_{\epsilon_{k\ell}}(v,\tau)|_{v=0}$, etc.
Then for any $k\neq \ell\neq p,q,r$ from the set $\{1,2,3,4,5\}$
one has\footnote{Note that in \cite{Enolski:2007periods}, it is assumed that
$b_{0}=4$ and
the basis $\omega$ therein is $2$ times that in this work.
Also the convention for the period matrix is different from ours.} \cite{Enolski:2007periods}

\begin{equation}\label{eqnperiodforquintic}
\Pi_{A}^{(5)}(\omega)
=
{\theta_{pq}  \theta_{pr} \theta_{qr}
\over \theta_{k\ell}  (e_{k}-e_{\ell})^{3\over 2}}
\begin{pmatrix}
\partial_{v_{1}}|_{v=0}  \theta_{\ell}& \partial_{v_{1}}|_{v=0}  \theta_{k}\\
\partial_{v_{2}}|_{v=0}  \theta_{\ell} & \partial_{v_{2}}|_{v=0}  \theta_{k}
\end{pmatrix}\,.
\begin{pmatrix}
{1\over \theta_{p\ell} \theta_{q\ell} \theta_{r\ell}}& 0\\
0 & {1\over \theta_{pk} \theta_{qk} \theta_{rk}}
\end{pmatrix}
\begin{pmatrix}
1 & -e_{\ell}\\
-1 & e_{k}
\end{pmatrix}\,.
\end{equation}
%%%
\iffalse
Choosing $k=2,\ell=4$
so that $\epsilon_{k},\epsilon_{\ell}$ are odd.
Let $\delta_{a},a=1,2,3,4$ be the remaining 4 odd characteristics.
Then the Rosenhain formula tells that
\begin{equation}
\det
\begin{pmatrix}
\partial_{v^{1}}|_{v=0}  \theta_{\ell}& \partial_{v^{1}}|_{v=0}  \theta_{k}\\
\partial_{v^{2}}|_{v=0}  \theta_{\ell} & \partial_{v^{2}}|_{v=0}  \theta_{k}
\end{pmatrix}
=
\end{equation}
\fi
%%%%

Choosing $k=2,\ell=4$
and normalizing the curve such that $e_{2}=0,e_{4}=1$, this leads to the
Rosenhain normal form \eqref{eqnhyperellipticRosenhain} with
a modular representation for the period matrix. In particular, one has
\begin{equation}\label{eqnperiodforRosenhain}
\left(\det \Pi_{A}^{(5)}(\omega)\right)^{-1}
=-{1\over \pi^2} {\theta_{1100}\theta_{0010}\theta_{1001} \theta_{0000}\theta_{0001}\theta_{1111}\theta_{0100}
\over
(\theta_{0110}\theta_{1000}\theta_{0011})^{3}}\,,
\end{equation}
with
\begin{equation}
{(\theta_{1100}\theta_{0000}\theta_{1001} )^{4}
\over
(\theta_{0110}\theta_{1000}\theta_{0011})^{4}}
=\lambda_{1}\lambda_{2}\lambda_{3}\,.
\end{equation}
Here we have followed the convention for $\theta_{xyzw}$ as in
\eqref{eqnthetaconvention}.
Note that the relation displayed above differs from the one obtained  form
\eqref{eqnRosenhainlambda} by an action in $\Gamma(2)\backslash \Gamma(1)$.

\subsubsection{Sheaf-theoretic and analytic definitions of quasi-modular forms}

In this subsection we shall define quasi-modular forms, following the analytic and sheaf-theoretic formulation in \cite{Kaneko:1995, Urban:2014nearly, Liu:2016nearly}.

For this purpose we need some elementary representation-theoretic aspects of the vector bundles involved which we now briefly review, see
 \cite{VanderGeer:2008, Pitale:2015representations, Pitale:2015lowest, Klemm:2015direct} for related studies.
Vector bundles on $S=\Gamma\backslash\mathcal{H}$ in consideration descend from homogeneous vector bundles on
$\mathcal{H}=\mathrm{Sp}_{4}(\mathbb{R})/\mathbb{U}_{2}$. Sections of the former bundles
 are equivalent to $\Gamma$-invariant sections of the homogeneous vector bundles.
The basic building block for the ring of modular forms is the bundle
$\underline{\omega}$.
Its pull-back from $S=\Gamma\backslash\mathcal{H}$ to $\mathcal{H}$ is the homogeneous vector bundle
associated to the fundamental representation of $\mathbb{U}_{2}$.
One can then
construct various bundles by taking tensor product $\otimes^{k} \underline{\omega} ,k\geq 0$, symmetric tensor product $\mathrm{Sym}^{\otimes k }\underline{\omega} ,k\geq 0$, etc on the corresponding representations.
The $\Gamma$-invariant sections of these bundles are modular forms valued in these representations.
See Appendix \ref{appendix:reps} for related discussions.\\

To motivate the sheaf-theoretic definition of quasi-modular forms  \cite{Urban:2014nearly},
we consider sections of
$\mathrm{Sym}^{\otimes (k-m)}\underline{\omega}\otimes \mathrm{Sym}^{\otimes m} \mathcal{V}_{S}$
with $m\geq 0$.
Although the local system $R^{1}\pi_{*}\mathbb{C}_{\mathcal{Z}}$ does not extend across cusps
on the Satake compactification of $S$, the sheaf $
 \mathcal{V}_{S}$ does by Schmid's nilpotent orbit theorem \cite{Schmid:1973variation} or Deligne's canonical extension.
We focus on a cusp singularity where the $A$-cycles
 vanish at the fiber over the singularity and the local monodromy is given by the Dehn twist.
Denote $p: \mathcal{H}\rightarrow S=\Gamma\backslash \mathcal{H}$.
Then the natural frame for the lift along $p$ of the extended sheaf  $\mathcal{V}_{S}$
is given by $\{\omega= e=\beta \tau+\alpha ,\beta\}$ by Deligne's canonical extension.
This frame is the one used in the study of mixed Hodge structure and is induced by the monodromy weight filtration. We shall therefore call this frame the monodromy weight frame.
This frame
 satisfies $\Pi_{A}(\omega)=1$ and
 has the following transformation under the action $\gamma$ on the marking
\begin{equation}\label{eqntransformationofDeligneframe}
\begin{pmatrix}
e,
\beta
\end{pmatrix}
\mapsto
\begin{pmatrix}
e,
\beta
\end{pmatrix}
\begin{pmatrix}
(c\tau+ d)^{-1} &- c^t\\
0&( c \tau+d)^t
\end{pmatrix}\,.
\end{equation}

Consider the transformation of the coordinate of the lifted section $p^{*}s$ for the section $s$ of
$\mathrm{Sym}^{\otimes (k-m)}\underline{\omega}\otimes \mathrm{Sym}^{\otimes m} \mathcal{V}_{S}$ in terms of this
frame. The lift $p^{*}s$ is given by
\begin{equation}
p^{*}s=\sum_{a=0}^{m} f_{a} (\tau) e^{\otimes (k-a)}\otimes \beta^{\otimes a}\,.
\end{equation}
The transformation law for $\{f_{a}\}_{a=0,\cdots, m}$ is hence
\begin{equation}\label{eqnsheafquasitransformation}
f_{a}(\tau)=
\sum_{\ell= a}^{m}
{l\choose a} f_{\ell}(\gamma\tau)  ((c \tau+d)^{-1})^{\otimes (k-\ell)} (-c^t)^{\otimes (\ell-a)} (( c\tau+d)^t )^{\otimes a}\,.
\end{equation}

Recall that the Hodge filtration induces an exact sequence
\begin{equation}
0\rightarrow
\mathrm{Sym}^{\otimes k}
\underline{\omega}
\rightarrow
\mathrm{Sym}^{\otimes (k-m)}\underline{\omega}\otimes \mathrm{Sym}^{\otimes m} \mathcal{V}_{S}
\rightarrow
\mathrm{Sym}^{\otimes ((k-2)-(m-1))}\underline{\omega}\otimes \mathrm{Sym}^{\otimes (m-1)} \mathcal{V}_{S}
\rightarrow 0\,.
\end{equation}
The middle map
\begin{equation}
\mathrm{Sym}^{\otimes (k-m)}\underline{\omega}\otimes \mathrm{Sym}^{\otimes m} \mathcal{V}_{S}
\rightarrow
\mathrm{Sym}^{\otimes ((k-2)-(m-1))}\underline{\omega}\otimes \mathrm{Sym}^{\otimes (m-1)} \mathcal{V}_{S}
\end{equation}
is
induced by the trace map $\mathrm{Tr}:\underline{\omega} \otimes \mathcal{V}_{S}
\rightarrow\mathcal{O}_{S}
$ whose kernel is
 $\mathrm{Sym}^{\otimes k}
\underline{\omega}$.
It is easy to see  that the quotient of
\begin{equation}
\mathrm{Sym}^{\otimes (k-m)}\underline{\omega}\otimes \mathrm{Sym}^{\otimes m} \mathcal{V}_{S}
\end{equation}
by the first piece $\mathrm{Sym}^{\otimes k}
\underline{\omega}$ in the Hodge filtration is the map
 sending the data
$\{f_{a}\}_{a=0,1,\cdots,m}$
to
$\{f_{a}\}_{a=1,\cdots, m-1}$.

The monodromy weight frame also induces a natural projection
\begin{equation}
\mathrm{Sym}^{\otimes (k-m)} \underline{\omega}\otimes
\mathrm{Sym}^{\otimes m}\mathcal{V}_{S}\rightarrow \mathrm{Sym}^{\otimes k} \underline{\omega}\,.
\end{equation}
This is the map sending the data $\{f_{a}\}_{a=0,1\cdots, m}$ to its top component $f_{0}$.
One can check that the data $\{f_{a}\}_{a=0,1\cdots ,m}$ related by \eqref{eqnsheafquasitransformation} is equivalent to the top component $f_{0}$.
See  \cite{Urban:2014nearly} for details.\\

The above discussions yield the following definition of weakly holomorphic ``quasi-modular" forms,
borrowing the terminology from the genus one case \cite{Kaneko:1995}.

\begin{dfn}[Quasi-modular forms by sheaf]\label{dfnalgebraicdefinition}
A weakly holomorphic quasi-modular form of weight $k$ and order $m\geq 0$ is a holomorphic section of
$\mathrm{Sym}^{\otimes (k-m)}\underline{\omega}\otimes \mathrm{Sym}^{\otimes m} \mathcal{V}_{S}$.
\end{dfn}

By working with the monodromy weight frame $\{e,\beta\}$ of $\mathcal{V}_{S}$ adapted to the Hodge filtration
which transforms according to \eqref{eqntransformationofDeligneframe}, the above definition is equivalent to the following
analytic definition of weakly holomorphic quasi-modular forms
 in the spirit of \cite{Kaneko:1995}.
\begin{dfn}[Quasi-modular forms by transformation]\label{dfnanalyticdefinition}
A weakly holomorphic quasi-modular form of weight $k$ and order $m$ is a collection of holomorphic functions $\{f_{a}(\tau)\}_{a,=0,1\cdots, m}$ which
 satisfy the transformation law
\begin{equation}
f_{a}(\tau)=
\sum_{\ell= a}^{m}
{\ell \choose a} f_{\ell}(\gamma\tau)  ((c \tau+d)^{-1})^{\otimes (k-\ell)} (- c^t)^{\otimes (\ell-a)} (( c\tau+d)^t )^{\otimes a}\,.
\end{equation}
\end{dfn}

We remark that sections of the lifts to $\mathcal{H}$ of the above vector bundles on $S$
are
by definition well behaved under transformations by the modular group.
It is the
local frame, naturally singled out from
the monodromy weight filtration,  which ruins the coordinates presentations of the $\Gamma$-invariant sections of the vector bundles.\\

The above algebraic definition
given in Definition \ref{dfnalgebraicdefinition} and the analytic one in Definition \ref{dfnanalyticdefinition}  provide different
perspectives on the weakly holomorphic quasi-modular forms,
and we shall work with both throughout this paper.

\begin{ex}

A weakly holomorphic quasi-modular form with $m=0$ is a weakly holomorphic vector-valued modular form.
For example, a section of the  $\underline{\omega}$ is  a vector-valued modular form with $k=1,m=0$.
 %%%
 \iffalse
 The induced determinant representation
 $\det(\Pi_{A}(\omega))$ is a scalar-valued modular form of weight $1$
 which is not a quasi-modular form according to our definition.
 \fi
 %%%
  The quantity $\eta =\alpha\Pi_{A}(\eta)+\beta \Pi_{B}(\eta)$ is a section of
$\mathrm{Sym}^{\otimes (k-m)}\underline{\omega}\otimes \mathrm{Sym}^{\otimes m} \mathcal{V}_{S}$
with $k=m=1$ and is a quasi-modular form.

To see this analytically,
 recall that
 the dual frame of $\{e,\beta\}$ is given by
 $\{A, B-\tau A\}$. In terms of this frame, the section $\omega$ is
given by
 \begin{equation}
 \omega=e\int_{A} \omega + \beta\int_{B-\tau A}\omega=e\cdot \Pi_{A}(\omega)\,.
 \end{equation}
That is, the coordinate of $\underline{\omega}$
in the frame $\{e\}$
is $\Pi_{A}(\omega)$ which is easily checked to be a modular form.
For the section $\eta$, one has
  \begin{equation}
 \eta=e\int_{A} \eta+ \beta\int_{B-\tau A}\eta=e\cdot \Pi_{A} (\eta)+\beta \cdot (\Pi_{B} (\eta)-\tau \Pi_{A} (\eta))\,.
 \end{equation}
 The coordinates are weakly holomorphic quasi-modular forms
whose transformation follows straightforwardly from
 \eqref{eqntransformationofquasiperiods}.
See Appendix \ref{appendixqgeometricperspective} for details.

\end{ex}

\subsubsection{Ring structure inherited from representations}

Let $M(\Gamma)$ be the graded ring of holomorphic scalar-valued modular forms for $\Gamma$, generated over $k$ by
 $\det^{\bullet\geq 0} \underline{\omega}$.
The ring $R(\Gamma)$ of meromorphic scalar-valued  modular forms
is defined to be the
fractional field of the ring $M(\Gamma)$.
It includes the functional field  of the projective closure of the variety $\Gamma\backslash \mathcal{H}$
since
rational functions on $S$ which constitute the field $k(S)$ correspond to
modular functions $k(\Gamma)$ (weight zero scalar-valued modular forms for $\Gamma$).

For the case in study where $\underline{\omega}$ has rank two corresponding to the fundamental
representation of $\mathbb{U}_{2}$, we have
$\wedge^{2} \underline{\omega}=\det \underline{\omega} $
whose sections are
scalar-valued modular forms of weight $1$.
Note that
\begin{equation}
\underline{\omega}^{\otimes 2}=\mathrm{Sym}^{\otimes 2}\underline{\omega}
\oplus \wedge^{2} \underline{\omega}\,.
\end{equation}
With the space of scalar-valued modular forms
well understood thanks to the work of \cite{Igusa:1962siegel},
working with $\mathrm{Sym}^{\otimes 2}\underline{\omega}$
is essentially equivalent to
working with $\underline{\omega}^{\otimes 2}$.
In fact, by the Weyl character formula, any holomorphic representation
decomposes into a direct sum of $\mathrm{Sym}^{k}\underline{\omega}$.
Therefore, the set of representations
$\{ \det^{\bullet}{\underline{\omega}} \otimes \mathrm{Sym}^{\bullet \geq 0}  \underline{\omega},\oplus,\otimes \}$ carries a module structure
over $\{\det^{\bullet}{\underline{\omega}},\oplus,\otimes\}$.
See \cite{VanderGeer:2008} for related discussions.\\

For the purpose of this work, we shall focus on a subclass of
weakly holomorphic modular forms and quasi-modular forms.

\begin{dfn}[Ring of modular forms by sheaf]\label{defringofmodularforms}
Let $R(\Gamma)$ be the ring of scalar-valued meromorphic modular forms for $\Gamma$.
Define a subring of modular forms to be the $R(\Gamma)$-module given by
\begin{equation}
\mathcal{R}(\Gamma)=R(\Gamma)[ \mathrm{Sym}^{\bullet\geq 0}H^{0}(S, \underline{\omega})]\,.
\end{equation}
\end{dfn}

From either the sheaf-theoretic or the analytic definition of quasi-modular forms, we can see
that the space of quasi-modular forms (including modular forms as a special case with $m=0$)
is a $R(\Gamma)$-module and
carries a graded ring structure, with the grading given by representations.

Recall that $\mathcal{V}_{S}$ corresponds to the trivial representation of
$\mathbb{U}_{2}$ and hence the decomposition of its tensor product
is also clear. We can then define the following ring of quasi-modular forms.

\begin{dfn}[Ring of quasi-modular forms by sheaf]\label{defringofquasimodularforms}
Let $R(\Gamma)$ be the ring of scalar-valued meromorphic modular forms for $\Gamma$.
Define the ring of quasi-modular forms to be the $\mathcal{R}(\Gamma)$-module given by
\begin{equation}
\widetilde{R}(\Gamma) =R(\Gamma)[\mathrm{Sym}^{\bullet\geq 0} H^{0}(S, \underline{\omega}),
 \mathrm{Sym}^{\bullet \geq 0} H^{0}(S,\mathcal{V}_{S})
]\,.
\end{equation}
\end{dfn}

The corresponding analytic definition
of quasi-modular forms is as follows.
\begin{dfn}\label{dfnringofquasimodularforms}[Ring of quasi-modular forms, analytic]
Define the ring of meromorphic quasi-modular forms to be
\begin{equation}
\widetilde{\mathcal{R}}(\Gamma)=R(\Gamma)[ \Pi_{A}(\omega),  \Pi_{A}(\eta)  ]\,.
\end{equation}

\end{dfn}
Here and in what follows, a notation  such as $\Pi_{A}(\omega)$ appearing in a ring stands for the set of its entries.
We have kept its matrix form as a book-keeping device to indicate that it is naturally valued in a representation.
Note that using our notation, $\Pi_{A}(\omega)^{\otimes k}$
represents entries of $\Pi_{A}(\omega)\otimes\cdots \otimes \Pi_{A}(\omega)$ instead of $\Pi_{A}(\omega)^{k}$, and it is the top component of a modular form valued in the representation $\mathrm{Sym}^{\otimes k}\underline{\omega}$
in the monodromy weight frame mentioned earlier.

\subsubsection{Almost-meromorphic modular forms}

By working with the real category
and choosing the frame $\{e,\bar{e}\}$ for $\mathcal{V}_{S}$, the notion of real-analytic modular forms can be defined.
See \cite{Katz:1976p} for details.
From its transformation law, it is easy to see that the quantity $(\tau-\bar{\tau})^{-1}$
is the coordinate description of a real-analytic modular form lying in
$\mathrm{Sym}(\underline{\omega})\otimes \mathrm{Sym}(\bar{\underline{\omega}})\otimes \mathcal{C}_{S}^{\infty}$ in the sheaf description.
In this work
we however shall not need this notion but only a more restricted one called almost-meromorphic modular forms \cite{Kaneko:1995}
or nearly holomorphic modular forms
\cite{Shimura:1986class, Shimura:1987, Urban:2014nearly}.\\

The discussions in Appendix \ref{appendixquasiperiods}
show that
one can make $\Pi_{A}(\omega)^{-1}\Pi_{A}({\eta})$ modular by adding a non-holomorphic term to it
\begin{equation}\label{eqnalmostholomorphicmodular}
\Pi_{A}(\omega)^{-1}\widehat{\Pi_{A}({\eta}) }:=
\Pi_{A}(\omega)^{-1}\Pi_{A}({\eta})+2\pi i \Pi_{A}(\omega)^{-1}(\tau-\bar{\tau})^{-1}\Pi_{A}(\omega)^{-t}\,.
\end{equation}

\begin{dfn}\label{dfnringofalmostholomorphicmodularforms}[Ring of almost-meromorphic modular forms, analytic]
Let the notations be as above. Define the ring of almost-meromorphic modular forms to be
\begin{equation}
\widehat{\mathcal{R}}(\Gamma)=R(\Gamma)[ \Pi_{A}(\omega), \widehat{\Pi_{A}({\eta}) } ]\,.
\end{equation}
\end{dfn}

The splitting of the Hodge filtration sequence defined by using Hodge decomposition induces the projection
\begin{equation}
\mathrm{Sym}^{\otimes (k-m)}\underline{\omega}\otimes \mathrm{Sym}^{\otimes m} \mathcal{V}_{S}
\rightarrow \mathrm{Sym}^{\otimes k}\underline{\omega}\otimes \mathcal{C}^{\infty}_{S}\,.
\end{equation}
One can show that it induces an isomorphism
between the ring $\widetilde{\mathcal{R}}$ and the ring
$
\widehat{\mathcal{R}}
$.
Therefore, the notions of quasi-modular forms
and almost-meromorphic modular forms are equivalent through the Hodge decomposition.
This is the generalization of
the map from quasi-modular forms to almost-meromorphic modular forms for the dimension one case.
The constant term map and modular completion for the genus one case \cite{Kaneko:1995} can be
defined either sheaf-theoretically or analytically in the same way.
Interested readers are referred to \cite{Kaneko:1995, Urban:2014nearly} for details.

\subsubsection{Differential structure from Picard-Vessiot extension%: algebraic
}
\label{secdifferentialstructure}

As mentioned in the previous section,
Definition  \ref{dfnalgebraicdefinition}  and
Definition \ref{dfnanalyticdefinition}
 of quasi-modular forms
offer both algebraic and analytic descriptions of quasi-modular forms.
Two immediate remarks are in order.
\begin{enumerate}

\item
Firstly, it is easy to check that the derivatives of a modular form of nonzero weight yield quasi-modular forms.
Ideally a ring of quasi-modular forms should be stable under differentiation $\partial_{\tau_{ij}},i,j=1,2$, similar to the genus one case \cite{Kaneko:1995}.
It is not clear from the definitions whether the ring $\widetilde{\mathcal{R}}(\Gamma)$  in Definition \ref{dfnringofquasimodularforms}
 is a differential ring.

\item
Secondly, it is not clear if the  ring $\widetilde{\mathcal{R}}(\Gamma)$ is  ``minimal"
in the sense that it contains no nontrivial sub-differential ring
that includes the ring $k(\Gamma)$ of modular functions.

\end{enumerate}
In this section we shall address these points.\\

First we shall show that the ring of quasi-modular forms as defined in Definition \ref{dfnringofquasimodularforms}
is closed under the differential $\partial_{\tau}$, where $\partial_{\tau}$
denotes collectively the partial derivatives $\partial_{\tau_{ij}}, i,j\in \{1,2\}$.
To be precise, we shall prove the following statements.

\begin{thm}[\cite{Bertrand:2000transcendence}]\label{thmdifferentialringofquasimodularforms}
Let $R=R(\Gamma), \widetilde{\mathcal{R}}=\widetilde{\mathcal{R}}(\Gamma)$ be the ring of scalar-valued meromorphic modular forms, the ring of meromorphic quasi-modular forms
 for $\Gamma$, respectively. Let
 $R(\Pi_{A}(\omega),  \Pi_{A}(\eta))$ be the fractional field of $\widetilde{\mathcal{R}}$. Then the following statements hold.
\begin{enumerate}\item
The fractional field $R(\Pi_{A}(\omega),  \Pi_{A}(\eta))$ of $\widetilde{\mathcal{R}}$ is stable under $\partial_{\tau}$.
\item
The fractional field $R(\Pi_{A}(\omega),  \Pi_{A}(\eta))$ of $\widetilde{\mathcal{R}}$ has transcendental degree $10$ over $\mathbb{C}$.

\item The presentation of the fractional field $R(\Pi_{A}(\omega),  \Pi_{A}(\eta))$ is the following. As a field extension of $R$,
the generators are the entries of $\Pi_{A}(\omega),\Pi_{A}(\eta)$, subject to the only relation given by the Riemann-Hodge bilinear relation
\begin{equation}\label{eqn:RiemanHodgebilinear1}
\left(\Pi_{A}(\omega)^{-1}\Pi_{A}(\eta)\right)^{t}
=
\Pi_{A}(\omega)^{-1}\Pi_{A}(\eta)
\,.
\end{equation}
\end{enumerate}
\end{thm}

The proof of Theorem \ref{thmdifferentialringofquasimodularforms} is essentially given in \cite{Bertrand:2000transcendence} (see also \cite{Zudilin2000:thetanulls}) which was used to prove the following statement.
Let $(K,\partial)$ be a differential field such that
the field of constant under $\partial$ is $k=\mathbb{C}$.
Let $L$ be another field extension of $k$.
Denote by
$K_{\partial}\langle L\rangle$ the differential field obtained by adjoining to $K$ all
$\partial$-differentials of elements from $L$.

As before, we take a set of three algebraically independent generators $\{t_1,t_2,t_3\}$ from $k(S)=k(\Gamma)$ and denote them
 collectively by $t$.
Hence up to algebraic extension $k(S)$ coincides with $k(t_1,t_2,t_3)$.
Denote
the differentials $\{\partial_{\tau_{ij}}, i,j,=1,2\}$ collectively by $\partial_{\tau}$
and similarly for $\partial_{t}$.

\begin{thm}[\cite{Bertrand:2000transcendence}]\label{thmdifferentialGaloisfield}
Consider the differential field
\begin{equation}
\mathcal{D}:=
k_{\partial_{\tau}} \langle k(S)\rangle\,,\quad
%\mathcal{F}'
%=k_{\partial_{\tau}}\langle \lambda, \tau\rangle
%=k_{\partial_{t}}\langle \lambda, \tau\rangle
%=k_{\partial_{t}} \langle \lambda, \tau\rangle  \,,
\quad
 k=\mathbb{C}\,,
\end{equation}
which up to algebraic extension coincides with
$\mathcal{F}:=k_{\partial_{\tau}}\langle  t \rangle$ defined by adjoining all
$\partial_{\tau}$-derivatives of the generators $t=\{t_1,t_2,t_3\}$ of $k(t)$ to the field $k$.
The differential field $\mathcal{F}$, and hence $\mathcal{D}=k_{\partial_{\tau}} \langle k(S)\rangle$,
is a finite extension of the field generated over $k(t)$  by the $\tau$-derivatives of $t$ of order $\leq 2$.
It has transcendental degree $10$ over $k$.
\end{thm}

%%%
\iffalse
As mentioned earlier, by passing to a subgroup if needed the algebraic extension does not cause trouble in our study of modularity.
\fi
%%%
The idea of the proofs is to translate the problem into a problem on period integrals, and then
utilize differential Galois theory.
Our proof of Theorem \ref{thmdifferentialringofquasimodularforms}
is a little more Hodge-theoretic reformulation of the proof of Theorem \ref{thmdifferentialGaloisfield} given in \cite{Bertrand:2000transcendence}.

\begin{proof}[Proof of Theorem \ref{thmdifferentialringofquasimodularforms}]

We prove this theorem in several steps.
%Then the Picard-Fuchs system for the periods $\Pi_{\gamma}(\omega)$
%follows from the triviality of the
%cohomology a $\mathbb{C}(\lambda)$-linear combination of $\omega, \nabla_{v_{0}}\omega,\nabla^2_{v_{0}}\omega$
%by dimension reasons.

\paragraph{Step 1. Picard-Vessiot extension.}
Let
\begin{equation}
\mathcal{E}':
=k(t)_{ \partial_{t}}
\langle \Pi_{A}(\omega),  \Pi_{B}(\omega),\Pi_{A}(\eta),\Pi_{B}(\eta) \rangle  \,
\end{equation}
be the Picard-Vessiot extension corresponding to the Picard-Fuchs system for the family $\pi:\mathcal{Z}\rightarrow S=\Gamma\backslash\mathcal{H}$.
Here again the notation $\Pi_{A}(\omega)$ denotes collectively its entries and similar notations are used for others.
The notation $k(t)_{\partial_{t}}\langle \Pi_{A}(\omega),  \Pi_{B}(\omega),\Pi_{A}(\eta),\Pi_{B}(\eta)\rangle$
denotes the differential closure of $k(t)( \Pi_{A}(\omega),  \Pi_{B}(\omega),\Pi_{A}(\eta),\Pi_{B}(\eta))$.
By differential Galois theory, the field $\mathcal{E}'$
has transcendental degree
$10$ over $k(t)$ and hence
transcendental degree $10+3$ over $k$.

Let
\begin{equation}
\mathcal{E}:
=k(t)_{ \partial_{t}}
\langle \Pi_{A}(\omega),  \Pi_{A}(\eta) \rangle  \,.
\end{equation}
Due to the universality of the family in \eqref{eqnetaisGaussManin}, it follows that
\begin{equation}
\mathcal{E}
=
k(t)_{ \partial_{t}}
\langle \Pi_{A}(\omega)\rangle
\subseteq
\mathcal{E}'
=
k(t)_{ \partial_{t}}
\langle \Pi_{A}(\omega),  \Pi_{B}(\omega)\rangle
\,.
\end{equation}
By the fact that the Picard-Fuchs system is of order two, we have
\begin{equation}
\mathcal{E}
=
k(t)(
 \Pi_{A}(\omega),\partial_{t}\Pi_{A}(\omega))
\subseteq
\mathcal{E}'
=k(t)
( \Pi_{A}(\omega), \Pi_{B}(\omega) ,\partial_{t} \Pi_{A}(\omega), \partial_{t} \Pi_{B}(\omega))\,.
\end{equation}
Denote for simplicity $\pi= \Pi_{A}(\omega)$ which is nondegenerate with
$ \Pi_{B}(\omega)= \tau \pi$. Then
\begin{eqnarray}\label{eqnfieldsEEprimeintermsofpi}
&&\mathcal{E}
=k(t)_{\partial_{t}}\langle \pi\rangle
=
k(t)(
 \pi,\partial_{t}\pi)\,,\nonumber\\
&&\mathcal{E}'=k(t)_{ \partial_{t}}
\langle \pi ,\tau \rangle =
k(t)(
\pi, \tau, \partial_{t} \pi, \partial_{t} (  \tau \pi))
=k(t) (
\pi, \tau, \partial_{t} \pi, \partial_{t} \tau )\,.
\end{eqnarray}
% Since $\mathcal{E}'$ is stable under $\partial_{t}$, we see that in particular $\partial_{t}^n \tau \in k(t) (\pi, \partial_{t} \pi)\,, n\geq 0$.

\paragraph{Step 2. Properties of $\partial_{\tau}t$ and $\pi$.}
The quantity $\pi$ is a modular form valued in the fundamental representation of $\mathbb{U}_{2}$
. This implies that $\det\pi$ lies in the graded ring $M$ of holomorphic scalar-valued modular forms
\begin{equation}
\det \pi\in M\,.
\end{equation}
Therefore, $k(t)(\pi)$ contains a weight-one scalar-valued modular form
and hence the fractional ring $R$ of the ring $M$.
That is
\begin{equation}\label{eqnRisincluded}
M\subseteq R\subseteq k(t)(\det \pi)\subseteq k(t)(\pi)\,.
\end{equation}

Furthermore, it is easy to check that  for any modular function $f\in k(t)$,
$d_{\tau} f\in \mathrm{Sym}^{\otimes 2} \underline{\omega}$ and
$\pi^{-1}\cdot (\partial_{\tau_{ij}})  f\cdot \pi^{-t} $ is a (matrix-valued) modular function. %See the discussions in Appendix \ref{appendix:reps} for details.
This tells that
\begin{equation}
\partial_{\tau}f\in k(t)(\pi)\,,
\quad \forall f\in k(t)\,.
\end{equation}
In particular one can take $f$ to be one of the generators $t_{1},t_{2},t_{3}$.
This leads to
\begin{equation}\label{eqnpartialtautincluded}
\partial_{\tau}t\in k(t)(\pi)\,.
\end{equation}
Similar to \eqref{eqnRisincluded}, we have
\begin{equation}\label{eqnRisincluded2}
M\subseteq R\subseteq \left( k(t)(\det \partial_{\tau}f)\right)^{\mathrm{alg}}\subseteq \left(k(t)(\partial_{\tau}f)\right)^{\mathrm{alg}}\,, \quad \forall f\in k(t)\,.
\end{equation}
Here for a field $K$, $K^{\mathrm{alg}}$ stands for its algebraic closure.

\paragraph{Step 3. Eliminating $\partial_{\tau}t$.}
According to \eqref{eqnpartialtautincluded}, we know that $\partial_{\tau}t\in \mathcal{E}\subseteq \mathcal{E}'$.
It follows from \eqref{eqnfieldsEEprimeintermsofpi}  that
\begin{equation}
\mathcal{E}
=k(t)_{\partial_{t}}\langle \pi\rangle
=k(t)_{\partial_{\tau}}\langle \pi\rangle\,,
\quad
\mathcal{E}'
=k(t)_{\partial_{t}}\langle \pi, \tau\rangle
=k(t)_{\partial_{\tau}}\langle \pi,\tau\rangle\,.
\end{equation}
We can also get rid of the generators $\partial_{\tau}t$ to obtain
\begin{eqnarray}\label{eqndifferentialGaloisinpartialtau}
 \mathcal{E}'
=k(t)
(\pi, \tau, \partial_{t} \pi)
=k(t) (
\pi, \tau, \partial_{\tau} \pi)
\,.
\end{eqnarray}
By definition we have $\partial_{t} \Pi_{A}(\omega)=\Pi_{A} (\nabla_{\partial_{t}} \omega)$,
with
$\nabla_{\partial_{t}} \omega$ being a $\mathcal{O}_{S}$-linear combination of
$\omega, \eta$.
Since $\mathcal{O}_{S}\subseteq R$, from \eqref{eqnfieldsEEprimeintermsofpi}, \eqref{eqndifferentialGaloisinpartialtau} we have
\begin{equation}\label{eqnEEprimequasiperiod}
\mathcal{E}
=k(t)
(\pi,   \partial_{t} \pi )=k(t)
(\Pi_{A}(\omega),   \Pi_{A}(\eta) )
\,,
\quad
\mathcal{E}'
=k(t) (
\Pi_{A}(\omega), \tau, \Pi_{A}(\eta))\,.
\end{equation}

\paragraph{Step 4. Computing transcendental degrees.}

By the Riemann-Hodge bilinear relations \eqref{eqn:RiemanHodgebilinear1}, the number
of algebraically independent entries in $\Pi_{A}(\omega), \tau, \Pi_{A}(\eta)$
is at most $4+3+3=10$.
According to the presentation in \eqref{eqnEEprimequasiperiod},
that $\mathcal{E}'$ has transcendental degree $13$ over $k$ immediately tells that
these generators are in fact algebraically independent.

The differential field  $\mathcal{E}$ is embedded into the fractional field of the ring
of convergent series in $e^{2\pi i \tau_{ij}},i,j\in \{1,2\}$ which is linearly disjoint from $k(\tau)$.
Hence
$\mathcal{E}'$ has  nonzero transcendental degree $3$ over $\mathcal{E}$.
Therefore $\mathcal{E}$ has transcendental degree $13-3=10$ over $k$.
Clearly this has generators
being $t_{1},t_{2},t_{3}$ and the entries of $\Pi_{A}(\omega),\Pi_{A}(\eta)$, subject to the only relation given by the Riemann-Hodge bilinear relation \eqref{eqn:RiemanHodgebilinear1}.
\iffalse
\begin{equation}
\Pi_{A}(\omega) \Pi_{A}(\eta)^{t}
=
\Pi_{A}(\eta) \Pi_{A}(\omega)^{t}
\,.
\end{equation}
%as reviewed in \eqref{eqnRiemannHodge}.
\fi

\paragraph{Step 5. Differential structure of the fractional field.}
Clearly
\begin{eqnarray*}
k(t)\subseteq  R\subseteq \widetilde{\mathcal{R}}
\subseteq R(\Pi_{A}(\omega), \Pi_{A}(\eta) )\,.
\end{eqnarray*}
From the relation $R\subseteq k(t)(\Pi_{A}(\omega))$ in  \eqref{eqnRisincluded}
we have
\begin{eqnarray*}
 R(\Pi_{A}(\omega), \Pi_{A}(\eta) )
\subseteq k(t)(\Pi_{A}(\omega), \Pi_{A}(\eta))\,.
\end{eqnarray*}
Combing the above two we obtain
\begin{eqnarray*}
 R(\Pi_{A}(\omega), \Pi_{A}(\eta) )
=k(t)(\Pi_{A}(\omega), \Pi_{A}(\eta))\,.
\end{eqnarray*}
From \eqref{eqnEEprimequasiperiod}, the right hand side of the above is
\begin{eqnarray*}
 k(t)(\Pi_{A}(\omega), \Pi_{A}(\eta))=\mathcal{E}=k(t)
(\Pi_{A}(\omega),   \partial_{t} \Pi_{A}(\omega))\,.
\end{eqnarray*}
Therefore, the fractional field $ R(\Pi_{A}(\omega), \Pi_{A}(\eta) )$ of the ring
$\widetilde{\mathcal{R}}$ is isomorphic to $\mathcal{E}$
and hence has transcendental degree $10$ over $k$.
This finishes the proof.
 \end{proof}

In fact, due to the universality of the family $\pi:\mathcal{Z}\rightarrow S$, the entries $\partial_{\tau}t$ form a basis
for modular forms valued in $\Omega_{S}^{1}\cong \mathrm{Sym}^{\otimes 2}\underline{\omega}$. That is,
all binomials in the entries of $\pi$
lie in the field $R(\partial_{\tau}t)$.
This says that $\pi$ is algebraic over $R(\partial_{\tau}t)$.
That is,
 \begin{equation}\label{eqnpiisincluded}
\pi \in  \left(R(\partial_{\tau}t)\right)^{\mathrm{alg}} \,,
\end{equation}
where $\left(R(\partial_{\tau}t)\right)^{\mathrm{alg}}$ denotes the algebraic closure of $R(\partial_{\tau}t)$.
Combining \eqref{eqnRisincluded2} and \eqref{eqnpiisincluded}, we obtain
\begin{equation}\label{eqnincludingpiinalgebraicclosure}
 \pi \in \left(R(\partial_{\tau}t)\right)^{\mathrm{alg}}
 \subseteq \left(k(t) (\partial_{\tau} t) \right)^{\mathrm{alg}} \,.
\end{equation}
%%%
\iffalse
Together with \eqref{eqnpartialtautincluded}, we obtain
\begin{equation}\label{eqntowerinclusion}
 \partial_{\tau}t\in k(t)(\pi)\subseteq k(t)(\partial_{\tau}t)^{\mathrm{alg}}\subseteq k(t)_{\partial_{\tau}}^{\mathrm{alg}}\,.
\end{equation}
\fi
%%%
This implies that
\begin{equation}\label{eqnpartialtpiisincluded}
\partial_{t}\pi\in  \left(k(t) ( \partial_{\tau} t,  \partial_{\tau}^2 t)\right)^{\mathrm{alg}}\subseteq \left(k_{\partial_{\tau}} \langle t\rangle \right)^{\mathrm{alg}}\,.
\end{equation}
%%%
\iffalse
This leads to
\begin{equation}
k( \partial_{\tau}t)\subseteq k(t)(\pi)\subseteq k(t)(\partial_{\tau}t)^{\mathrm{alg}}\subseteq k(t)_{\partial_{\tau}}^{\mathrm{alg}}\,.
\end{equation}
\fi
%%%
The two relations \eqref{eqnpiisincluded}, \eqref{eqnpartialtpiisincluded} tell that the $\partial_{\tau}$-stable field
$\mathcal{E}$ satisfies
\begin{equation}
\mathcal{F}= k_{\partial_{\tau}} \langle t\rangle \subseteq \mathcal{E}\
=k(t)_{\partial_{\tau}}\langle \pi\rangle
=k(t)_{\partial_{t}}\langle \pi\rangle
 \subseteq \mathcal{F}^{\mathrm{alg}}\,.
\end{equation}
This proves that $\mathcal{F}$
has the same transcendental degree with $\mathcal{E}$
and
is in fact a finite extension of the field generated over $k(t)$  by the $\tau$-derivatives of $t$ of order $\leq 2$.
Since $k(S)$ is a finite extension of $k(t)$,
by the proof for Theorem \ref{thmdifferentialringofquasimodularforms}
both of  $\mathcal{F}$ and $k_{\partial_{\tau}}(k(S))$ have transcendental degree $10$.
This proves Theorem \ref{thmdifferentialGaloisfield}.

\begin{rem}
In the proof of Theorem \eqref{thmdifferentialringofquasimodularforms} we proved $\Pi_{A}(\omega)$ is modular by looking at its transformation property.
Another argument provided in \cite{Bertrand:2000transcendence, Zudilin2000:thetanulls} is
to firstly pick a particular family (with modular group $\Gamma_{4,8}$) and
 explicit  basis $\omega$ constructed by theta functions, and then prove the statement by direct checking.
Again
an algebraic extension does not affect any of the
discussions used in the proof.
The relation in \eqref{eqnpiisincluded} is shown similarly.
Interesting computations on the presentation of the differential field
 for the particular subgroup  $\Gamma=\Gamma_{4,8}$ can be found in \cite{Bertrand:2000transcendence}.

\end{rem}

\begin{cor}\label{cordifferentialringofquasimodularforms}
The following statements hold for the rings
\begin{equation}\label{eqnormalizedrings}
\widetilde{\mathcal{R}}=R[\Pi_{A}(\omega),   \Pi_{A}(\eta)]\,,
\quad
\widetilde{\mathcal{R}}^{o}:=R[\Pi_{A}(\omega), \Pi_{A}(\omega)^{-1}\Pi_{A}(\eta)]
\end{equation}

\begin{enumerate}
\item
They are stable under $\partial_{t}$ and $\partial_{\tau}$.
\item
Their fractional fields have transcendental degree $10$ over $\mathbb{C}$.

\item The presentation of the former ring $\widetilde{\mathcal{R}}$ is the following. Over the ring $R$,
the generators are  the entries of $\Pi_{A}(\omega),\Pi_{A}(\eta)$, subject to the only relation given by the Riemann-Hodge bilinear relation
\begin{equation}
\Pi_{A}(\omega) \Pi_{A}(\eta)^{t}
=
\Pi_{A}(\eta) \Pi_{A}(\omega)^{t}
\,.
\end{equation}
The presentation of the latter ring $\widetilde{\mathcal{R}}^{o}$ is the following. Over the ring $R$,
the generators the entries of $\Pi_{A}(\omega),\Pi_{A}(\omega)^{-1}\Pi_{A}(\eta)$, subject to the only relation given by the Riemann-Hodge bilinear relation
\begin{equation}
\Pi_{A}(\omega)^{-1}\Pi_{A}(\eta)
=
\left(\Pi_{A}(\omega)^{-1}\Pi_{A}(\eta)\right)^{t}
\,.
\end{equation}
\end{enumerate}

\end{cor}

The reason we have chosen the generators $\Pi_{A}(\omega), \Pi_{A}(\omega)^{-1}\Pi_{A}({\eta}) $ in constructing the ring $\widetilde{\mathcal{R}}^{o}$
is that
in later applications it is  the quantity $\Pi_{A}(\omega)^{-1}\Pi_{A}({\eta}) $ that naturally enters the stage
through the Bergman kernel.
\begin{proof}

Recall that the Gauss-Manin connection is
an $\mathcal{O}_{S}$-derivation. This tells that
$\partial_{t}^{2}\pi$ is a $k(t)$-linear combination of $\pi ,\partial_{t}\pi$.
Therefore,
$k(t)[ \pi, \partial_{t}\pi]$ is closed under $\partial_{t}$.

The derivation of
\eqref{eqnpartialtautincluded} in the proof of Theorem \ref{thmdifferentialringofquasimodularforms} in fact shows
\begin{equation}
\partial_{\tau}t\in k(t)[\pi]\,.
\end{equation}
By the chain rule, this tells that $k(t)[ \pi, \partial_{t}\pi]$ is also $\partial_{\tau}$-stable.
Since $R\subseteq k(t)[\pi]$,
the ring $R[ \pi, \partial_{t}\pi]$ and hence
$R[\Pi_{A}(\omega), \Pi_{A}(\eta)]$ is also $\partial_{\tau}$-stable.
Applying \eqref{eqnetaisGaussManin}, we have proved the first statement.
The rest follow from Theorem \ref{thmdifferentialringofquasimodularforms}.

\end{proof}

Corollary \ref{cordifferentialringofquasimodularforms} says that
the ring $\widetilde{\mathcal{R}}$ of quasi-modular forms is
a differential ring. It is not only a
$k(t)_{\partial_{\tau}}$-module, but also up to algebraic extension it coincides with
$k(t)_{\partial_{\tau}}$.
That is, any ring including $k(t)$
and closed under differentiation $\partial_{\tau}$
must be $\widetilde{\mathcal{R}}$, up to an algebraic extension.
Therefore, $\widetilde{\mathcal{R}}$ is the ``minimal" differential ring
including $k(t)$.
On the other hand,
the presentation in Corollary \ref{cordifferentialringofquasimodularforms} gives an extremal simple dimension count \cite{Bertrand:2000transcendence}
\begin{equation}
3+4+4-1
\end{equation}
for the transcendental degree
of its factional field: the $3$ occurs as that for the field $k(S)$, the first $4$ from the $A$-cycle periods
forming a vector-valued modular form, the second $4$ from the $A$-cycle quasi-periods, and the $-1$
from the 2nd Riemann-Hodge bilinear relation.
This structure is universal in the sense that it is completely determined by differential Galois theory for the universal family and does not reply on the details on the congruence subgroup
$\Gamma$.

From the presentation of the
ring $\widetilde{\mathcal{R}}^{o}$
we see that all of
the ``quasi"-ness is completely encoded in the quasi-period matrix $\Pi_{A}(\omega)^{-1}\Pi_{A}(\eta)$.
For the case of a family of Jacobian varieties of genus two curves,
its explicit expression is given in \eqref{eqnausimodularintermsofthetaderivative}
in terms of special values of the derivatives of theta functions.
The structure of the ring $\widetilde{\mathcal{R}}^{o}$ from the differential field perspective gives another way of presenting
the quasi-modular generator up to algebraic extension:
$\widetilde{\mathcal{R}}^{o}$ can be generated over ${\mathcal{R}}$ from
the derivative of any scalar-valued modular form of nonzero weight.
The reason is that according to the transformation law, such a derivative can not be modular and therefore must be quasi-modular.
For example, one can take
$\partial_{\tau}\log \chi_{10}$
or $\partial_{\tau}\log \chi_{12}$.

\begin{ex}

Consider the quintic model \eqref{eqngenustwomodelquintic} or  \eqref{eqnhyperellipticquintic} with $b_0=1,b_6=0$.
The period matrix $\Pi_{A}(\omega)$ are given in \eqref{eqnperiodforquintic}.
The quasi-period matrix $\Pi_{A}(\eta)$ can be derived from \eqref{eqnausimodularintermsofthetaderivative}.
One has the following Rauch's variational formula for the periods and quasi-periods, see \cite{Enolski:2007periods} and references therein\footnote{In \cite{Enolski:2007periods}, it is assumed that
$b_{0}=4$ and
the basis $\omega$ therein is $2$ times that in this work.
Also the convention for the period matrix is different from ours. We have rewritem the results therein in our convention.},
\begin{eqnarray}\label{eqnRauch}
{\partial\over \partial e_k}
\begin{pmatrix}\Pi_A(\omega)&-\Pi_A(\eta)\\ \Pi_B(\omega)&-\Pi_B(\eta)\end{pmatrix}
=\begin{pmatrix}\Pi_A(\omega)&-\Pi_A(\eta)\\ \Pi_B(\omega)&-\Pi_B(\eta)\end{pmatrix}
\begin{pmatrix}\alpha_k^t &\gamma_k^t\\ \beta_k^t&-\alpha_k
\end{pmatrix}\,,
\quad
k=1,2,\cdots, 5\,.
\end{eqnarray}
The matrices $\alpha_k,\beta_k,\gamma_k$ are given as follows.
Let
\begin{equation*}
U(x)
=\begin{pmatrix}
1\\
x
\end{pmatrix}:=
\begin{pmatrix}
U_{1}(x)\\
U_{2}(x)
\end{pmatrix}\,,\quad
 V(x)=
 \begin{pmatrix}
  \sum_{\ell=1}^{4} \ell b_{3-\ell} x^{\ell}\\
   4  \sum_{\ell=2}^{3} (\ell-1) b_{2-\ell} x^{\ell} \
     \end{pmatrix}
     :=
     \begin{pmatrix}
V_{1}(x)\\
V_{2}(x)
\end{pmatrix}
     \,.
\end{equation*}
Then
\begin{eqnarray*}
\alpha_k&=&-{1\over 2}  \left( {1\over f'(e_{k})}  U(e_k) V(e_k)^t -
\begin{pmatrix}
0 & 0\\
1 & 0
\end{pmatrix}
\right)\,,\nonumber\\
\beta_k&=&-2  \left( {1\over f'(e_{k})}  U(e_k) U(e_k)^t
\right)\,,\nonumber\\
\gamma_k&=&{1\over 8}  \left( {1\over f'(e_{k})}  V(e_k) V(e_k)^t
-e_{k}
\begin{pmatrix}
{V_{1}(e_k)\over e_{k}} & {V_{1}(e_k)\over e_{k}}\\
{V_{1}(e_k)\over e_{k}} & {V_{2}(e_k)\over e_{k}^2}
\end{pmatrix}
-
\begin{pmatrix}
{V_{1}(e_k)\over e_{k}} & 0\\
0 & {V_{2}(e_k)\over e_{k}^2}
\end{pmatrix}
\right)\,.
\end{eqnarray*}
From this one can immediately derive the Picard-Fuchs equations for the periods and quasi-periods.
One also has
\begin{equation*}
{\partial\over \partial e_{k}}\tau= 2\pi i \Pi_{A}(\omega)^{-1} \beta_{k}^{t} (\Pi_{A}(\omega)^t)^{-1} \,.
\end{equation*}
These relations exhibit the differential ring structure of $\widetilde{\mathcal{R}}=R[\Pi_A(\omega),\Pi_A(\eta)]$.

\end{ex}

%%% the following could be commentted out, but useful

\subsubsection{Other differentials arising from covariant derivatives}
\label{subsecquasifromconnection}

We now discuss the transformation law of the derivatives of a modular form and of its some other differentials.
By construction,
for a modular function $f$ its derivative $\partial_{\tau}f$ lies in
$ \Omega_{S}^{1}\otimes k(S)\cong \mathrm{Sym}^{\otimes 2}\underline{\omega}\otimes k(S)$.
This says that we get a modular form in the symmetric square of the standard representation. In particular, this is a modular form instead of a quasi-modular form.
If however $f$ lies in some nontrivial representation $V_{\rho}$, then $\partial_{\tau}$
has to be replaced by the covariant derivative $D=\partial_{\tau}+a$ on $V_{\rho}$
so that $Df\in V_{\rho}\otimes \Omega_{S}^{1}\otimes k(S)$.
The existence of the nontrivial connection matrix $a$
tells that $\partial_{\tau} f$ is not modular.
In this work we shall mainly be interested in the special case $V_{\rho}=\underline{\omega}^{\otimes k}$ or
$ \mathrm{Sym}^{\otimes k} \underline{\omega}$ for some $k$.\\

By construction the differentiation by $\partial_{\tau}$ is nothing but the one induced by Gauss-Manin connection
\begin{equation}
\nabla: \mathcal{V}_{S}\rightarrow \mathcal{V}_{S}\otimes \Omega_{S}^{1} \,.
\end{equation}
For the modular form valued in $\underline{\omega}$,
the Gauss-Manin action $\nabla$ sends
it to a quantity
in
$ \mathcal{V}_{S}\otimes \Omega^{1}_{S}
\cong \mathrm{Sym}^{\otimes 2}\underline{\omega}\otimes \mathcal{V}_{S}$
which by Definition \ref{dfnalgebraicdefinition} yields a quasi-modular form.
A splitting of the Hodge filtration in \eqref{eqnHodgefiltration}
 induces a retraction
$r:\mathcal{V}\rightarrow \underline{\omega}$ and one can show that
$r\circ \nabla$ gives rise to a connection on
$\underline{\omega}$
\begin{equation}
D=r\circ \nabla:  \underline{\omega}\rightarrow
\mathcal{V}_{S}\rightarrow \mathcal{V}_{S}\otimes \Omega_{S}^{1}
\rightarrow \underline{\omega}\otimes \Omega_{S}^{1}
 \,.
\end{equation}
The resulting quantity is modular form valued in the representation
$ \underline{\omega}\otimes \Omega_{S}^{1} \cong  \underline{\omega}  \otimes \mathrm{Sym}^{\otimes 2} \underline{\omega}$.
Furthermore,
the above connection $r\circ \nabla$ induces
a connection $\det (r\circ \nabla)$ on
$(\det \underline{\omega}) ^{\otimes k}$ and in particular induces an
action on scalar-valued modular forms such as modular functions and Eisenstein series.

Note that the splitting might exist only  in an enlarged category
of varieties.   The one induced by Hodge decomposition is such an example.\\

To be more concrete, take a Hodge frame $\omega=(\omega_{1},\omega_{2}),\eta=(\eta_{1},\eta_{2})$
adapted to the Hodge filtration satisfying
\eqref{eqndeRhampairing}.
Suppose the Gauss-Manin connection $\nabla$ in the frame above is given by
\begin{equation}\label{eqnGaussManin}
\nabla
\begin{pmatrix}
\omega,
\eta
\end{pmatrix}
=
\begin{pmatrix}
\omega,
\eta
\end{pmatrix}
N
\,,
\quad N=
\begin{pmatrix}
N_{1} & N_{2}\\
N_{3} & N_{4}
\end{pmatrix}\in \mathfrak{sp}_{4}\otimes \Omega_{S}^{1}\,.
\end{equation}
A splitting of the Hodge filtration is equivalently given by a decomposition
\begin{equation}\label{eqnsplittingofHodgefiltration}
\mathcal{V}_{S}=\underline{\omega}\oplus \left(\mathcal{O}_{S} \sigma_{1}\oplus   \mathcal{O}_{S} \sigma_{2}\right)  \,,
\end{equation}
where $(\sigma_{1},\sigma_{2})$ is a set of sections of the form
\begin{equation}
\begin{pmatrix}
\sigma_{1},
\sigma_{2}
\end{pmatrix}
=
\begin{pmatrix}
{\eta}_{1},
{\eta}_{2}
\end{pmatrix}
+
\begin{pmatrix}
\omega_{1},
\omega_{2}
\end{pmatrix}
P
 \,,
\end{equation}
for some matrix-valued function $P$ on $S$.
When acting on a holomorphic section $s$  of $\underline{\omega}$ given by
\begin{equation}
s= \omega v\,,
\quad v:=
\begin{pmatrix}
v_{1} \\
v_{2}
\end{pmatrix}\,,
\end{equation}
we have
\begin{equation}\label{eqncovariantderivative}
D s=
r(\nabla s)
=r \left( \omega dv+ \omega N_{1}v+\eta N_3 v \right)
=
\omega (dv+N_{1} v-PN_3 v)\,.
\end{equation}

In terms of the frame $\{e,\beta\}$ one has
\begin{eqnarray*}
\nabla \omega&=&\omega N_{1}+\eta N_{3}
=e\cdot (\Pi_{A}(\omega) N_{1} +\Pi_{A} (\eta) N_{3} )
+\beta\cdot (\Pi_{B}(\eta)-\tau \Pi_{A}(\eta))N_{3}\,.
\end{eqnarray*}
Therefore, when acting on the section $s=\omega v$ one has
\begin{eqnarray*}
\nabla_{\partial_{\tau}} (\omega v) &=&
e\cdot \left(\Pi_A(\omega)\partial_{\tau}v +\Pi_{A}(\omega)N_{1}(\partial_{\tau})v +\Pi_{A} (\eta) N_{3}(\partial_{\tau}) v \right)
+\beta\cdot (\Pi_{B}(\eta)-\tau \Pi_{A}(\eta))N_{3}(\partial_{\tau})v\,.
\end{eqnarray*}
Similarly, the connection $D$ satisfies
\begin{eqnarray*}
D\omega&=& \omega N_{1}-\omega P N_{3}
=e\cdot \Pi_A(\omega)(N_{1}-PN_{3})\,.
\end{eqnarray*}
and hence
\begin{eqnarray*}
D_{\partial_{\tau}}(\omega v)&=& \omega\cdot
 (\partial_{\tau}v+N_{1}(\partial_{\tau})v-PN_{3}(\partial_{\tau})v)
 :=\omega D_{\partial_{\tau}} v\,.
\end{eqnarray*}

By Definition \ref{dfnalgebraicdefinition}, $\omega$ is a modular form of weight $1$ valued in
$\underline{\omega}$, while $\eta$ and $\sigma_{1},\sigma_{2}$ are components of the corresponding
quasi-modular forms $(\omega,\eta)$ and $(\omega,\sigma_{1},\sigma_{2})$
of weight $1$ and order $1$ valued in $\mathcal{V}_{S}$.
Hence by Definition \ref{dfnalgebraicdefinition}, one sees that
$\Pi_A(\omega)\partial_{\tau}v +\Pi_{A}(\omega)N_{1}(\partial_{\tau})v +\Pi_{A} (\eta) N_{3}(\partial_{\tau}) v $
transforms as a quasi-modular form while
$ \Pi_A(\omega)\left(\partial_{\tau}v+N_{1}(\partial_{\tau})v-PN_{3}(\partial_{\tau})v\right)$
%$ \partial_{\tau}v+N_{1}(\partial_{\tau})v-PN_{3}(\partial_{\tau})v$
as a modular form.
Therefore,
the difference term
\begin{equation}
(\Pi_{A}(\eta)+\Pi_A(\omega)P) N_{3}(\partial_{\tau})v
\end{equation}
measures the deviation between
 the quasi-modular quantity $\nabla_{\partial_{\tau}}s$
 and
 the modular quantity $D_{\partial_{\tau}}s$.

By construction, all of the terms $N_{1}(\partial_{\tau}),N_{3}(\partial_{\tau}),P$ are modular.
That is,
the quasi-modularity caused by
$\nabla s$
is completely encoded by the quasi-period matrix
$\Pi_{A}(\eta)$.\\

Mimicking the Ramanujan-Serre derivative for the genus one case, we consider the covariant derivative obtained by taking $P=0$.
When acting on the coordinate $v$ of a section $s=\omega v$ one has the coordinate of a modular form
\begin{eqnarray*}
\omega \cdot D_{\partial_{\tau}}v&=&\omega\cdot \left(\partial_{\tau}v +N_{1}(\partial_{\tau}) v
\right)\\
&=&e\cdot (\Pi_A(\omega)\partial_{\tau}v +\Pi_A(\omega)N_{1}(\partial_{\tau})v +\Pi_{A} (\eta) N_{3}(\partial_{\tau}) v )
-e\cdot \Pi_{A} (\eta) N_{3}(\partial_{\tau}) v\nonumber\\
&=& \mathrm{pr}\circ \nabla_{\partial_{\tau}} (\omega v)
-e\cdot \Pi_{A} (\eta) N_{3}(\partial_{\tau}) v\,,
\end{eqnarray*}
where $\mathrm{pr}$
is the projection $\mathcal{V}_{S}\rightarrow \underline{\omega}$
induced by the Hodge frame $\{e,\beta\}$.

One can also take $P=\widehat{\Pi_{A}(\eta)}$
which lives
in the smooth category.
The resulting covariant derivative is the analogue of Shimura-Maass derivative in genus one case.
See \cite{Urban:2014nearly} for the sheaf-theoretic description of this
operator which is induced by the Gauss-Manin connection on $\mathrm{Sym}^{\otimes m}
\mathcal{V}_{S}$ and yields a resulting operator
\begin{equation}
\mathrm{Sym}^{\otimes (k-m)}
\underline{\omega}\otimes \mathrm{Sym}^{\otimes m}
\mathcal{V}_{S}
\rightarrow
\mathrm{Sym}^{\otimes (k+2-(m+1))}
\underline{\omega}\otimes \mathrm{Sym}^{\otimes (m+1)}
\mathcal{V}_{S}
\,.
\end{equation}

\subsection{Quasi-Jacobi forms}
\label{secquasiJacobi}

\subsubsection{Differential function field of Jacobian variety}

We have mentioned in
Theorem
\ref{eqndifferentialfunctionfieldofhyperellipticJacobian}
in
Section \ref{subsecfunctionalfieldofJacobian}
that the functional field of the Jacobian $J(C)$ is a differential field stable
under $\partial_{u}$.

We now consider the relative version of the ring $k_{\partial_{u}}\langle \wp_{ij}\rangle$.
We again work with the family $\mathcal{J}(\mathcal{C})\rightarrow S=\Gamma\backslash\mathcal{H}$.
By Theorem
\ref{eqndifferentialfunctionfieldofhyperellipticJacobian}
we then have
\begin{equation}
k(S)_{\partial_{u}}\langle \wp_{ij}\rangle=k(S)( \wp_{ij}, \wp_{ijk})\,.
\end{equation}
It is easy to check that elements in the above ring are  meromorphic Jacobi forms for $\Gamma$, see \cite{Grant:1985theta}.
This then motivates the following definition.

\begin{dfn}\label{eqndefinitionofJ}
Let $\mathcal{R}, \widetilde{\mathcal{R}}, \widehat{\mathcal{R}}$ be the ring of modular, quasi-modular, almost-meromorphic modular forms, respectively.
Define
\begin{eqnarray*}
\mathcal{J}&:=&\mathcal{R}\otimes k(S)_{\partial_{u}}\langle \wp_{ij}\rangle
=\mathcal{R}\otimes  k(S)( \wp_{ij}, \wp_{ijk})\,,\\
\widetilde{\mathcal{J}}&:=&\widetilde{\mathcal{R}}^{o}\otimes k(S)_{\partial_{u}}\langle \wp_{ij}\rangle
=\widetilde{\mathcal{R}}^{o}\otimes  k(S)( \wp_{ij}, \wp_{ijk})\,,\\
\widehat{\mathcal{J}}&:=&\widehat{\mathcal{R}}^{o}\otimes k(S)_{\partial_{u}}\langle \wp_{ij}\rangle
=\widehat{\mathcal{R}}^{o}\otimes  k(S)( \wp_{ij}, \wp_{ijk})\,.
\end{eqnarray*}
Here the ring
$\widetilde{\mathcal{R}}^{o}$  is defined in
\eqref{eqnormalizedrings}, while the ring $\widehat{\mathcal{R}}^{o}$ is defined by replacing
$\Pi_{A}(\eta)$ in $\widetilde{\mathcal{R}}^{o}$ by $\widehat{\Pi_{A}(\eta)}$.
Elements in the ring $\mathcal{J}$ are meromorphic Jacobi forms for $\Gamma$.
We shall call elements in $\widetilde{\mathcal{J}},\widehat{\mathcal{J}}$ quasi-Jacobi and almost-meromorphic
Jacobi forms, respectively.
\end{dfn}

Note that
we should have used
a covariant derivative (raising operator in the context of Jacobi forms \cite{Eichler:1984}) on sections in $H^{0}(J(C), \mathcal{O}_{J(C)} (m\Theta))$.
However, for rational functions (Jacobi forms of index $m=0$), it turns out that the
connection matrix vanishes. This fact can be checked explicitly. Namely, the $u$-derivatives of index zero Jacobi forms
are Jacobi forms. Hence as long as we restrict to index zero objects,
the derivative $\partial_{u}$ respect the Jacobi form property.

\subsubsection{The Bergman kernel}
\label{secBergmankernel}

The Bergman kernel is characterized by some reproducing property
and normalization condition depending on the marking.
See for example \cite{Tyurin:1978periods, Takhtajan:2001free} for
details.
It can be constructed from the prime form and turns out to be, see  \cite{Baker:1898hyperelliptic, Buchstaber:1996hyperelliptic, Buchstaber:1997, Onishi:1998complex, Matsutani:2001hyperelliptic, Fay:2006theta, Buchstaber:2012multi},
\begin{equation}\label{eqndefofBergman}
\B(p,q)= H\log  \vartheta_{\delta}\left(\Pi_{A}(\omega)^{-t}\left(\phi(p)-\phi(q)\right)\right)
\cdot
d\phi(p)\boxtimes d\phi(q)\,,
\end{equation}
where $H$ is the Hessian operator $(\partial^{2}_{u_{i} u_{j}})$
and $\phi(p)=u(p+\infty)$.
The above is regarded as a matrix-valued two-form
lying in the space whose basis is given by the components of of $du\otimes du^{t}$.

According to \eqref{eqndefinitionofsigma} adapted to our convention \eqref{eqnperiodmatrix} for the matrix $\Pi$,
%using the previous convention for theta, hence the switched convention from the differential-geometric and representation-theoretic picture,
we see that it is related to the Weierstrass $\wp$-function by
\begin{equation}\label{eqnBergmanintermsofwp}
\B(p,q)=
\left(\wp (u(p)-u(-q))-  \Pi_{A}({\omega})^{-1} \Pi_{A}(\eta)
\right)
 \cdot d\phi(p)\boxtimes d\phi(q)\,.
\end{equation}
From \eqref{eqnbasisformTheta}, we see that
the quasi-period matrix $\Pi_{A}({\omega})^{-1}\Pi_{A}(\eta)$
is nothing but the degree zero term in the Laurent expansion of $\B$.
This is also consistent with the direct computation in \eqref{eqnausimodularintermsofthetaderivative}
where the quasi-periods are given in terms of particular values of derivatives of theta functions.\\

A different choice of odd characteristic corresponds to a different choice
of the marking that is related to the previous one
by an action in $\Gamma(1)=\mathrm{Sp}_{4}(\mathbb{Z})$, as mentioned before
in Section \ref{secothermodels}.
The embedding $\phi: C\rightarrow J(C)$ and the theta-characteristic $\delta$ in \eqref{eqnspecialtheta}
is fixed once and for all so that we can apply results in \cite{Grant1988:generalization, Grant:1990formal, Grant1991:generalization}
such as
\eqref{eqnexyintermsofwpij}.
Therefore we restrict ourselves to the modular subgroup $\Gamma(2)$
that fixes $\delta$,
and
pass to the modular group $\Gamma\cap\Gamma(2)$ in the constructions of
the fractional ring
$\mathcal{R},\widetilde{\mathcal{R}}$ of modular and quasi-modular forms
 in Section \ref{secquasimodularforms},
and in Definition \ref{eqndefinitionofJ}.
Under the action of $\gamma= (a,b;c,d)\in \Gamma\cap\Gamma(2)$ on the marking, one sees $\B$ transforms according to
\begin{equation}
\B\mapsto
\B+2\pi i    ( c\tau+d)^{-1}c \cdot d\phi(p) \boxtimes d\phi(q)\,,
\quad
\forall \gamma= (a,b;c,d)\in \Gamma\cap\Gamma(2)\,.
\end{equation}
%Here $\partial_{\tau}\log (c\tau+d):=c^{t} ( c\tau+d)^{-t}$.
%In the above the element $\gamma$ is assumed to lie in a subgroup preserving the characteristic $\delta$ of $\theta_{\delta}$.
This agrees with the transformation property implied by
the characterization of the Bergman kernel \cite{Tyurin:1978periods}.
Combining
\eqref{eqntransformationofquasiperiods}, it follows then that
$ \wp (u(p)-u(q))d\phi(p)\boxtimes d\phi(q) $ is modular invariant for the subgroup preserving
the characteristic $\delta$,
and that $\mathcal{B}\in \tilde{\mathcal{J}}$ is a quasi-Jacobi form.
If we replace the term $\Pi_{A}(\omega)^{-1}\Pi_{A}({\eta})$
in $\mathcal{B}$
by $\Pi_{A}(\omega)^{-1}\widehat{\Pi_{A}({\eta})}$, then the Bergman kernel $\mathcal{B}$ gets changed to the so-called Schiffer kernel
$\mathcal{S}$ which is an almost-meromorphic Jacobi form according to our definition.
See \cite{Tyurin:1978periods, Takhtajan:2001free} for details.\\

More intrinsically, from the fact that $\delta$ is odd, it is easy to see that
$\B(p,q)$
is an element
of
$H^{0}(C\times C, K_{C}\boxtimes K_{C}(2\Delta))^{\mathfrak{S}_{2}}$, where
$\Delta$ is the diagonal in $C\times C$.
In fact, consider the following sequence
of sheaves associated to the nonreduced divisor $2\Delta$
in $C\times C$
\begin{equation}
0\rightarrow
K_{C}\boxtimes K_{C}
\rightarrow
K_{C}\boxtimes K_{C}(2\Delta)
\rightarrow
K_{C}\boxtimes K_{C}(2\Delta)|_{2\Delta}
\rightarrow 0\quad\,.
\end{equation}
The 3rd map above is the polar part of an element in the middle
sheaf. It is also called the biresidue map or symbol map depending on the context.
Consider further the sequence of sheaves
\begin{equation}
0\rightarrow
(K_{C}\boxtimes K_{C})\otimes \mathcal{O}_{\Delta}(-\Delta)
\rightarrow
K_{C}\boxtimes K_{C}(2\Delta)|_{2\Delta}
\rightarrow
K_{C}\boxtimes K_{C}(2\Delta)|_{\Delta}
\rightarrow 0\quad\,.
\end{equation}
By the adjunction formula, one has
$K_{C}\boxtimes K_{C}(2\Delta)|_{\Delta}\cong \mathcal{O}_{\Delta}$.
Also one has
$(K_{C}\boxtimes K_{C})\otimes \mathcal{O}_{\Delta}(-\Delta)\cong K_{\Delta}$.
It can be shown
that
the restriction $K_{C}\boxtimes K_{C}(2\Delta)|_{2\Delta}$
has a one-dimensional space of sections invariant under the action by the symmetric group $\mathfrak{S}_{2}$,
and that there is a unique section which restricts to the canonical trivialization on $\Delta$.
The Bergman kernel is a lift of the above unique section along the biresidue map.
In particular, the extension of the latter to $3\Delta$
is determined by a projective structure which constitutes
a torsor over the space of quadratic differentials.
The extension data is the Bergman projective connection
that is the degree zero term in the Bergman kernel above.
See \cite{Biswas:1996projective, Ben:2003theta, Ben:2004opers} for details.\\

Therefore, in addition to the earlier characterization of $\Pi_{A}({\omega})^{-1}\Pi_{A}(\eta)$
as the coordinate of a section of $ \mathcal{V}_{S}$
with respect to the frame singled out by monodromy weight filtration, we have obtained yet another intrinsic description of it
as well as
the formula \eqref{eqndefofBergman}
which is given in terms of theta functions with odd characteristic.

We summarize the above discussions in the following.
\begin{cor}\label{corBergmankerneldegreezeroterm}
The Bergman kernel $\mathcal{B}$ is
a lift of its principal part in $K_{C}\boxtimes K_{C}(2\Delta)|_{2\Delta}$
that in turn is the unique symmetric element lifting the canonical
section of $K_{C}\boxtimes K_{C}(2\Delta)|_{\Delta}$.
The degree zero term of $\mathcal{B}$
is a projective connection
that is an extension of the above unique symmetric element in $K_{C}\boxtimes K_{C}(2\Delta)|_{2\Delta}$
to $K_{C}\boxtimes K_{C}(2\Delta)|_{3\Delta}$.
The ring $\widetilde{\mathcal{R}}^{o}$ of quasi-modular forms
is generated by the Bergman projective connection
over the ring $\mathcal{R}$ of modular forms.
Concretely, one has
\begin{equation}
\B(p,q)=H\log  \vartheta_{\delta}\left(\Pi_{A}(\omega)^{-t}\left(\phi(p)-\phi(q)\right)\right)
\cdot
d\phi(p)\boxtimes d\phi(q)\,,
\end{equation}
where $H$ is the Hessian operator $(\partial^{2}_{u_{i} u_{j}})$ and $\delta$ can be any odd theta-characteristic.
\end{cor}

As explained in Remark \ref{remprincipalpartofBergmankernel}, for the quintic model in \eqref{eqngenustwomodelquintic}
or  \eqref{eqnhyperellipticquintic} with $b_0=1$
\begin{equation}
y^2=f(x;b)=x^5+b_{1}x^4+b_{2}x^3+b_{3}x^2+b_{4}x+b_{5}\,,
\end{equation}
the Bergman kernel \eqref{eqnBergmanintermsofwp} admits an algebraic expression
using \eqref{eqnrelationbetweenwpijandrationalfunctions}
\begin{equation}\label{eqnhyperellipticBergmankernel}
\mathcal{B} (p,q)  ={G(x_{1}, x_{2}) + 2y_{1}y_{2}\over (x_{1}-x_{2})^2} {dx_{1}\over 2y_{1}} {dx_{2}\over 2y_{2}}
-\Pi_{A}(\omega)^{-1}\Pi_{A}(\eta)d\phi(p)\boxtimes d\phi(q)
\,,
\end{equation}
with $p=(x_1,y_1), q=(x_2,y_2)$ and $G$ is given in \eqref{eqnprincipalpartofBergmankernel}.
For the sextic model given in \eqref{eqnhyperellipticformfromcanonicalmap}
\begin{equation*}
Y^{2}= F(X;a)=\sum_{k=0}^{6}a_{k}X^{6-k}\,,
\end{equation*}
the algebraic form for the Bergman kernel is given by \cite[Page24]{Buchstaber:2012multi}. To be precise, let the affine coordinates of $p,q$ be $(X_1, Y_1,1)$ and $(X_2,Y_2,1)$ respectively, then
\begin{equation}\label{eqnhyperellipticBergmankernelsextic}
\mathcal{B}(p,q)
=\frac{G^{(6)}\left(X_{1}, X_{2}\right)+2 Y_{1} Y_{2}}{4\left(X_{1}-X_{2}\right)^{2}} \frac{\mathrm{d} X_{1}}{ Y_{1}} \frac{\mathrm{d} X_{2}}{ Y_{2}}-\Pi_A(\omega)^{-1}\Pi_{A}(\eta)d\phi(p)\boxtimes d\phi(q).
\end{equation}
where now
\begin{eqnarray} \label{eqnGforsextic}
G^{(6)}(X_{1}, X_{2}) &=&2  Y_{2}^{2}+2\left(X_{1}-X_{2}\right) Y_{2} \frac{\mathrm{d}  Y_{2}}{\mathrm{d} X_{2}}\nonumber \\
&&+(X_{1}-X_{2})^{2} \sum_{j=1}^{2} X_{1}^{j-1} \sum_{k=j}^{5-j}(k-j+1) a_{6-(k+j+1)} X_{2}^{k} \nonumber\\
 &=&2 a_{0} X_{1}^{3} X_{2}^{3} +\sum_{i=0}^{2} X_{1}^{i} X_{2}^{i}\left(2 a_{6-2 i}+a_{6-2 i-1}(X_{1}+X_{2})\right) \,.
\end{eqnarray}

As we shall see below, in topological recursion, all of the quasi-modularity in
the open GW potentials enter through $\mathcal{B}$.
Therefore, keeping track of the failure of modularity (in the built-in combinatorial structure in topological recursion) in these potentials is equivalent to keeping track of
the non-holomorphic dependence in the Schiffer kernel.

\begin{rem}
For the genus one case, we can take
$\omega={dx\over y}, {\eta}={xdx\over y}$  in the
Weierstrass normal form.
By uniformization, the lift of this frame is $\tilde{\omega}=dz, \tilde{\eta}=\wp dz$.
We normalize the period $\Pi_{A}(\tilde{\omega})$. Then
$\Pi_{A}(\tilde{\eta})=\int_{A} \wp dz=-\eta_{1}$, where $\eta_{1}=\zeta(z+1)-\zeta(z)$
is a quasi-period.
It turns out that
\begin{equation}
\B(u(p)-u(q))
=\partial_{p}\partial_{q}\ln\vartheta_{({1\over 2}, {1\over 2})}(u(p)-u(q))du(p)\boxtimes du(q)=(\wp+\eta_{1})du(p)\boxtimes du(q)\,,
\end{equation}
where $\eta_{1}=(\pi^2/3)E_{2}$.
It is well known that
$\eta_{1}$ is an elliptic quasi-modular form satisfying
\begin{equation}
\eta_{1}( {a\tau+b\over c\tau+d})=(c\tau+d)^2 \eta_{1}(\tau)-2\pi i c(c\tau+d)\,,
\quad
\forall
\begin{pmatrix}
a & b\\
c & d
\end{pmatrix}\in \mathrm{SL}_{2}(\mathbb{Z})\,.
\end{equation}
On the other hand, we have
\begin{equation}
{1\over \mathrm{Im } {a\tau+b\over c\tau+d}}=(c\tau+d)^2 {1\over \mathrm{Im}\tau}-2i c(c\tau+d)\,.
\end{equation}
Then $\hat{\eta}_{1}=\eta_{1}-2\pi i /(\tau-\bar{\tau})$ is modular and is called
an almost-holomorphic modular form \cite{Kaneko:1995}.
\end{rem}

\section{Topological recursion for genus two mirror curve families and applications}
\label{secapplications}

According to the proof of the Remodeling Conjecture \cite{BKMP2009, FLZ16},
the open GW potentials  of  a toric CY 3-fold $\mathcal{X}$ equal
the differentials $\{\omega_{g,n}\}_{g,n}$ constructed from Eynard-Orantin topological recursion  \cite{Eynard:2007invariants}
for the mirror curve,
under the so-called open and closed mirror maps.
In what follows we shall first study  the geometry of the mirror curve and discuss the construction of the mirror maps, focusing on
a particular example.
Then we apply results in Section \ref{secringofmodularforms} to prove modularity of Eynard-Orantin topological recursion.
Finally we prove the modularity of open and closed GW potentials, based on the proof of the Remodeling Conjecture.

\subsection{Geometry of genus two curves constructed from mirror symmetry}

Our approach in establishing modularity applies to a large class
of genus two mirror curves with hyperelliptic structure constructed from the brane structure.
 In this paper we shall illustrate our strategy
 by analyzing in detail the resolution of the orbifold $\mathbb{C}^3/\mathbb{Z}_6$ which is one of the simplest
 toric Calabi-Yau threefolds with genus two mirror curves.
 This example has been studied extensively in the context of mirror symmetry,
 see e.g., \cite{Katz:1996fh, Chiang:1999tz} and in particular \cite{Klemm:2015direct}
 with which our present work in this part is closely related.

\subsubsection{Construction from mirror symmetry}

%\subsubsection{Toric construction for mirror curve}

The local toric Calabi-Yau threefold $\mathcal X$ in the A-model
is the resolution of the  orbifold $\mathbb{C}^3/\mathbb{Z}_6$, with
\begin{equation}
h^2(\mathcal X)=3\,.
\end{equation}
The height one slice of its fan is depicted in Figure \ref{figuretoricdiagram} below.

\begin{figure}[h]
 \renewcommand{\arraystretch}{1.2}
\begin{displaymath}
  \resizebox{80mm}{!}{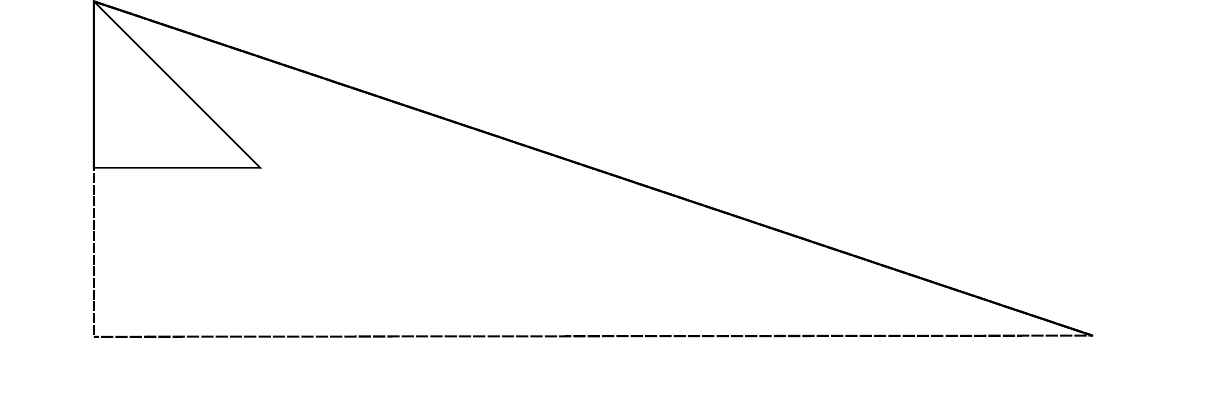}
  \label{fig:defining-polytope}
  \end{displaymath}
  \caption[toricdiagram]{Height one slice of the toric diagram.}
  \label{figuretoricdiagram}
\end{figure}

\iffalse
\begin{rem}
If we move the point $(6,-1)$ in the defining polytope to $(m,-1)$ for $m=0,\dots, 5$ (and adjust the triangulation accordingly), we still have a toric CY $3$-manifold with $h^4=2$ and $h^2=3$, i.e. the number of K\"ahler parameters is $3$ while the genus of the mirror curves is $2$.
\end{rem}
\fi
%%%
\iffalse
The linear relations among the rays, that is the charge vectors, are
\begin{align*}
l^{(1)}=&(1,-2,1,0,0,0)\,,\\
l^{(2)}=&(0,1,-2,1,0,0)\,,\\
l^{(3)}=&(0,0,0,-2,1,1)\,.
\end{align*}
\fi
The mirror curve $C^{o}\subset (\mathbb{C}^*)^2$ is then given by
\begin{equation}\label{eqnhyperellipticmirrorcurve}
1+X+Y+q_1 X^2 + q_2 X^3+ q_3 X^6 Y^{-1}=0\,.
\end{equation}
We denote the compactified mirror curve by $C$ and the resulting families by
$\pi^{o}:  \mathcal{C}^{o}\rightarrow S, \pi: \mathcal{C}\rightarrow S$, respectively.

%\subsubsection{Hyperelliptic structure from brane data}

As elaborated in \cite{Fang:2018open}, we shall choose a suitable Lagrangian with framing $(\mathcal{L},\mathfrak{f})$ in studying the
open GW theory of $\mathcal{X}$. In order to simplify the B-model computation we choose $\mathfrak{f}=0$ and $X: \mathcal{C}^\circ\rightarrow \mathbb{C}$ gives a hyperelliptic structure of $\mathcal{C}$.
Let
\begin{equation}
R=\{r\in \mathcal{C}^\circ~|~ dX|_{r}=0\}
\end{equation}
 be the set of ramification points, then $R$ consists of all $6$ Weierstrass points of $\mathcal{C}$.

\iffalse
By the proof of the Remodeling Conjecture \cite{BKMP2009, FLZ16}, the open GW potentials coincide with
the differentials constructed by topological recursion on the mirror curve $C$ with the differential $\lambda$
given by $\log Y\, {dX/X}$.
The differential $\lambda$ is understood to be a local expression for its Laurent expansion
around the so-called open GW point $\mathfrak{s}_0: (X,Y)=(0,-1)$.
\fi

\subsubsection{Closed and open moduli }

After
a simple change of coordinate $Y\mapsto Y+h(X)$,
the affine part of the mirror curve in \eqref{eqnhyperellipticmirrorcurve}
takes the hyperelliptic form.
To be more precise, let
\begin{equation}
h(X)={1\over 2}(1+X+q_{1}X^2+q_{2}X^3)\,.
\end{equation}
Let $\tilde Y=Y+h(X)$, and then \eqref{eqnhyperellipticmirrorcurve} is reduced to
\begin{eqnarray}\label{eqnhyperellipticmirrorcurvesextic}
{\tilde Y}^2:&=&(Y+h(X))^{2}\\
&=&-q_{3}X^{6}+h^2(X)\nonumber\\
&=&({1\over 4}q_{2}^2-q_{3})X^{6}+{1\over 2}q_{1}q_{2}X^5+({1\over 4}q_{1}^2+{1\over 2}q_{2})X^4\nonumber\\
&&+({1\over 2}q_{1}+{1\over 2}q_{2})X^3+({1\over 4}+{1\over 2}q_{1})X^2+{1\over 2}X+{1\over 4} \,.
\end{eqnarray}
We again denote the sextic on the right hand side of the above equation \eqref{eqnhyperellipticmirrorcurvesextic} by
\begin{eqnarray}\label{eqnhyperellipticmirrorcurvesextic2}
F(X;a)=\sum_{k=0}^{6}a_{k}X^{6-k}=a_{0}\prod_{k=0}^{5}(X-r_{k})\,,
\end{eqnarray}
where now the coefficients $a_{k},k=0,1\cdots,6$ are polynomials in $q_{1},q_{2},q_{3}$.
The following result concerns the modularity of the parameters $q_1,q_2,q_3$ in the  mirror curve family \eqref{eqnhyperellipticmirrorcurve}.

\begin{lem}\label{lemmodularityofparameters}
The parameters $q_{1},q_{2},q_{3}$ in the mirror curve family \eqref{eqnhyperellipticmirrorcurve} are algebraically independent modular functions with respect to a certain congruence subgroup $\Gamma<\Gamma(1)=\mathrm{Sp}_{4}(\mathbb{Z})$.
\end{lem}

\begin{proof}
Clearly the Igusa absolute invariants $j_{1},j_{2},j_{3}$ in \eqref{eqnabsoluteinvariantsintermsofbinaryinvariants}
are rational functions in $q_1,q_2,q_3$.
It is routine, for example by using a computer program, to check that $\mathbb{C}(q_1,q_2,q_3)$ is an algebraic extension of $\mathbb{C}(j_1,j_2,j_3)$ using the Jacobian criterion.

Since the field $\mathbb{C}(j_1,j_2,j_3)$ is identified with the field $\mathbb{C}(\mathfrak{j}_1,\mathfrak{j}_2,\mathfrak{j}_3)$ of modular functions
for the modular group $\Gamma(1)$, $\mathbb{C}(q_1,q_2,q_3)$ is an algebraical extension of $\mathbb{C}(\mathfrak{j}_1,\mathfrak{j}_2,\mathfrak{j}_3)$.
By looking at the transcendental degrees over $\mathbb{C}$ of these fields, one sees that in particular the generators $q_1,q_2,q_3$
 algebraically independent over $\mathbb{C}$.
 Consider the group homomorphism
\begin{equation}
\Psi: \Gamma(1)\rightarrow \mathrm{Gal}(\mathbb{C}(q_1,q_2,q_3)/\mathbb{C}(\mathfrak{j}_1,\mathfrak{j}_2,\mathfrak{j}_3))\,.
\end{equation}
Its kernel $\Gamma=\mathrm{Ker}\,\Psi$ is of finite index in $\Gamma(1)$
since the Galois group is finite due to the algebraic extension property.
By construction, elements in the field extension $\mathbb{C}(q_1,q_2,q_3)$
are invariant under the subgroup $\Gamma$.
From the classical fact that finite index subgroups of $\mathrm{Sp}_{4}(\mathbb{Z})$
are congruence subgroups, one concludes that the generators $q_1,q_2,q_3$ of $\mathbb{C}(q_1,q_2,q_3)$ are modular functions for the congruence subgroup $\Gamma$.

\end{proof}

The same argument in Lemma \ref{lemmodularityofparameters} above
also shows that the roots $r_{k},k=1,\cdots,6$ are modular functions, and one has the tower of fields  attached to the sextic in
\eqref{eqnhyperellipticmirrorcurvesextic2} as depicted in Figure \ref{figuretoweroffields}.

\begin{figure}[h]
  \renewcommand{\arraystretch}{1.2}
\begin{displaymath}
\xymatrixcolsep{5pc}
\xymatrix{
\mathbb{C}(r_{1},r_{2},r_{3},r_{4},r_{5},r_{0},a_{0})\ar@{-}[d]\\
\mathbb{C}(q_{1},q_{2},q_{3})\ar@{-}[d]\\
\mathbb{C}(a_{1},a_{2},a_{3},a_{4},a_{5},a_{6},a_{0})\ar@{-}[d]\\
\mathbb{C}(j_{1},j_{2},j_{3})\cong \mathbb{C}(\mathfrak{j}_{1},\mathfrak{j}_{2},\mathfrak{j}_{3})
  }
\end{displaymath}
  \caption[Tower of fields]{Tower of fields attached to the family of sextics with 3 parameters,
  where $a_{k},k=0,1,\cdots, 6$ are rational functions in $q_{1},q_{2},q_{3}$.}
  \label{figuretoweroffields}
\end{figure}

In mirror symmetry, one often needs to translate results between the A- and B- sides using the mirror map.
In our case, let
\begin{equation}
s_1=q_1\,,\quad s_2=q_1^{-2}q_2\,,\quad s_3=q_2^{-2}q_3\,.
\end{equation}
Define
\begin{align*}
&A_1=0\,,\\
&A_2=\sum_{\substack{d_1\geq 2d_2,\, d_2\geq 2d_3,\\
d_3\geq 0,\,d_1>0}} \frac{(-1)^{d_2-1}(2d_1-d_2-1)!s_1^{d_1}s_2^{d_2}s_3^{d_3}}{d_1!(d_1-2d_2)!(d_2-2d_3)!(d_3!)^2}\,,\\
&A_3=\sum_{\substack{0\leq 2 d_1\leq d_2,\\
d_2>0,\,d_2\geq 2d_3\geq 0}} \frac{(-1)^{d_1-1}(2d_2-d_1-1)!s_1^{d_1}s_2^{d_2}s_3^{d_3}}{d_1!(d_2-2d_1)!(d_2-2d_3)!(d_3!)^2}\,,\\
&A_4=-\sum_{d_3>0} \frac{(2d_3-1)!s_3^{d_3}}{(d_3!)^2}\,.
\end{align*}
Then the closed mirror map $\mathbf{Q}=(Q_{1},Q_{2},Q_{2})$, constructed by solving Picard-Fuchs system
\cite{Chiang:1999tz, FLZ16} for the
non-compact Calabi-Yau that is mirror to $\mathcal{X}$,
is given by
\begin{eqnarray}\label{eqnclosedmoduli}
Q_1&=&s_1 \exp(-2A_2+A_3)\,,\nonumber\\
Q_2&=&s_2 \exp(A_2-2A_3+A_4)\,,\nonumber\\
Q_3&=&s_3 \exp(-2A_4).
\end{eqnarray}
The open mirror map $\mathbf{\mathfrak{X}}=(\mathfrak{X}_1,\ldots, \mathfrak{X}_n)$ are linear functions in the rational functions $X_1,X_2,\cdots,X_n$ on $C$
whose coefficients are polynomials in the $A_{k}$'s above.
These quantities are understood as expansions near the point $s_1=s_2=s_3=0$ and the so-called open GW point $\mathfrak{s}_0=(0,-1)$.
For this reason, we shall call $\mathbf{Q}, \mathbf{\mathfrak{X}}$ the closed- and open-moduli respectively.
See \cite{Fang:2018open} for detailed discussions on the explicit expressions and modularity for them in the genus one case.

Lemma \ref{lemmodularityofparameters} above implies that the closed-moduli \eqref{eqnclosedmoduli}
are Siegel modular functions, modulo the complicated exponential factors
$Q_{k}s_{k}^{-1},k=1,2,3$.
In this work, we shall rarely need to carry out the expansions in terms of $\mathbf{Q},\mathbf{\mathfrak{X}}$, and hence shall loosely
call $\mathbf{q}=(q_1,q_2,q_3)$ and $\mathbf{X}=(X_1,\cdots,X_n)$ the closed- and open- moduli. This makes it more convenient to phrase statements about modularity.

The same argument in Section \ref{secothermodels} applying to \eqref{eqnhyperellipticmirrorcurvesextic2} enables us to
obtain explicit expressions (see \eqref{eqnexyintermsofwpij}) for the rational functions $X,\tilde{Y}$ on the mirror curve in terms of pull-backs of rational functions in
$\wp_{ij},\wp_{ijk}$ which are Jacobi forms.

\subsection{Modularity in topological recursion}

Recall that the Eynard-Orantin topological recursion \cite{Eynard:2007invariants} defines
differentials $\omega_{g,n}, 2g-2+n>0$ recursively as follows.
One starts with an affine curve $C^{\circ}\rightarrow\mathbb{A}^{1}$ equipped with a simply ramified cover structure
whose ramification locus is $R^{\circ }$.
Then one defines the following basic ingredients on $C^{\circ}$ and on its compactification $C$
\begin{eqnarray*}
\lambda =\ln Y\frac{d  X}{ X}\,,\quad
\omega_{0,1}=0\,,
\quad
\omega_{0,2}=\mathcal{B}\,.
\end{eqnarray*}
Let
$*:p\rightarrow p^*$ be the local involution near a simple ramification point $p_0$.
In the present hyperelliptic cover case, $C\rightarrow \mathbb{P}^1$ is a $2:1$ cover and the $*$ action is the hyperelliptic involution
$\sigma$ given in \eqref{eqnhyperellipticinvolution}.
In general, one defines
\begin{align}
  \label{eqn:EO-recursion}
  \omega_{g,n}(p_1,\ldots, p_n) &= \sum_{p_0\in R^{\circ}}\mathrm{Res}_{p \to p_0}
  \frac{\int_{\xi = p}^{{p}^*} \mathcal{B}(p_n,\xi)}{2(\lambda(p)-\lambda({p^*}))}
  \Big( \omega_{g-1,n+1}(p,{p^*},p_1,\ldots, p_{n-1}) \nonumber\\
  &\quad\quad  + \sum_{g_1+g_2=g}
  \sum_{ \substack{ J\cup K=\{1,..., n-1\} \\ J\cap K =\emptyset } } \omega_{g_1,|J|+1} (p,p_J)\omega_{g_2,|K|+1}({p^*},p_K)\Big)\,,\\
\omega_{g,0}&=\sum_{p_0\in R^{\circ}}\mathrm{Res}_{p \to p_0} d\lambda^{-1}(p)\omega_{g,1}(p)\,.
\end{align}

Thanks to the proof of the Remodeling Conjecture \cite{BKMP2009, FLZ16},
studying the open GW theory for the toric CY 3-fold with the given Lagrangian and framing data
is reduced to studying topological recursion on the curve $C$.
Following the same strategy in \cite{Fang:2018open} in establishing the
modularity and structure theorems for the open GW potentials,
we shall first work with the Schiffer kernel $\mathcal{S}$ instead of the Bergman kernel $\mathcal{B}$
for topological recursion to obtain globally defined differentials $\widehat{\omega}_{g,n}$. In the final step we then take the holomorphic limit of the resulting differentials $\widehat{\omega}_{g,n}$ to produce the open GW potentials $\omega_{g,n}$.

Due to the residue calculus nature in topological recursion
\cite{Eynard:2007invariants}, the algorithm is independent of the choice of coordinates which arises from biregular
morphism on the algebraic curve $C$.
This can be further checked by induction on $g,n$ in the recursive construction of $\widehat{\omega}_{g,n}$.
Hence we can freely choose coordinates to simplify computations.

\subsubsection{Expansions of basic ingredients}

We now study the Laurent expansions of the basic ingredients  such as the differential
$\lambda$ and the Schiffer kernel $\mathcal{S}$ near the ramification points.\\

The differential $\lambda$ used in the proof of the Remodeling Conjecture is given by
\begin{equation}\label{eqnlambdadifferential}
\lambda =\log Y {d X\over X}=
\log(\tilde Y-h(X)) {dX\over X}\,.
\end{equation}

%%%
\iffalse
There are a lot of choices in the coordinate transformations.
Near a ramification point in the set of ramification points $ R\subseteq C$, we can choose the local uniformizer to be the algebraic one $y$ around
$e_{k},k\neq 6$, $x^{-{1\over 2}}$ around $e_{6}$.
In the former case, we apply the relation
\begin{equation}\label{eqntransitionofdifferentials}
2y dy=f_{x}dx\,
\end{equation}
to obtain
the $y$-expansion of $\Lambda=\lambda-\lambda^{*}$
near a ramification point $e_{k}\in R, k\neq 6$ which are rational functions in $x,y$.
The expansion of the recursion kernel $K=d^{-1}\widehat \omega_{0,2}/(\lambda-\lambda^{*})$
can be discussed similarly.
\fi
%%%

As reviewed in Section \ref{sechyperellipticJacobian}, the image of the ramification points
are 2-torsion points on $J(C)$.
Rational functions of derivatives of $\wp_{ij}$, if not identically zero or infinity on $C$,
will be valued in Siegel modular forms at torsion points.
This implies that
 the Laurent  expansion in coordinate $\tilde{Y}$ of the quantity $\lambda(p)-\lambda(p^*)$,
that is of central importance in topological recursion,  has Siegel modular forms as coefficients.

\begin{rem}\label{remhyperellipticity}
In the original set-up of topological recursion \cite{Eynard:2007invariants},
it is assumed that a cover $C\rightarrow \mathbb{P}^{1}$ is a ramified cover with ramification index $2$ or $1$.
The local involution $*$ is the action which switches the two branches at a ramification point
with ramification index $2$ and fixes the rest of the branches,
For computational purpose,
one needs to realize the action of the local involution $*$ on the functional field of the algebraic curve  $C$.
This is easiest to do if the cover is Galois. The simplest case is the hyperelliptic cover.
Without the hyperellipticity, similar statements about the ring structure and holomorphic anomaly in topological recursion still hold, but now one needs a case by case analysis
to figure out the action on the functional field based on the
details of the equation for the curve.
\end{rem}

To expand the Schiffer kernel $\mathcal{S}$, we again use the algebraic local uniformizer.
From the algebraic formula \eqref{eqnhyperellipticBergmankernel} or \eqref{eqnhyperellipticBergmankernelsextic} for the principal part of the Bergman kernel,
it is easy to obtain
the expansion of
$\mathcal{S}(p, q)-\mathcal{S}(p, \sigma(q))$ around $q\in R$.
The Laurent coefficients are rational functions in the coordinates of the points $p, q$.
This can also be checked by first showing that the $X,\tilde Y$ coordinates of ramification points
are scalar-valued meromorphic modular forms using Lemma \ref{lemmodularityofparameters}
and then using \eqref{eqnhyperellipticBergmankernelsextic}.

\iffalse
\begin{rem}

Alternatively, we can expand it in an analytic local uniformizer.
As mentioned earlier in Section \ref{subsecfunctionalfieldofJacobian}, one can
choose $u_1-u_1(p)$ to be the local uniformizer around the ramification point $p\neq e_{0}$, or
$u_{2}$ around $p=e_{0}$. See \cite{Onishi:1998complex, Onishi:2002determinant} for details.
Then we
use the explicit formulae for $x,y$ in terms of the pull-back
of rational functions in $\wp_{ij}$ along the Abel-Jacobi map $\phi$, and
expand them in terms of the corresponding analytic local uniformizer.
On the Jacobian varieties, the derivatives of $\wp_{ij}$ are already polynomials in
$\wp_{ij}, \wp_{ijk}$ according to Theorem \ref{remprincipalpartofBergmankernel}.
However, the subtlety about the relation between differentiation and pull-back along $\phi$
makes the expansion in the analytic local uniformizer less convenient than that in the algebraic one.
This should be contrasted with the genus one case \cite{Fang:2018open}
which has the coincidence that the Jacobian variety is isomorphic to the curve itself
and thus the expansion in the analytic coordinate is equally effective as that in the algebraic coordinate.

\end{rem}
\fi

\subsubsection{Modularity of initial values}

The initial values for topological recursion are the differentials
$\widehat{\omega}_{g,n}$
satisfying $2g-2+n\leq 0$. These are the disk potential $\partial_{x}W$, the annuals potential
$\widehat{\omega}_{0,2}$ and the genus one potential $\widehat{F}_{1}$.
The open GW potentials will be the holomorphic limits of them.

The disk potential is given by
$\partial_{x}W=\log Y /X$ and
the annulus potential is nothing but the Bergman kernel $\mathcal{B}$.
The modularity (strictly speaking, quasi-Jacobiness) of both are evident, as shown above.
Specializing the above computations to the mirror curve family in
 \eqref{eqnhyperellipticmirrorcurvesextic},
the formula for $\widehat{F}_{1}$ in \cite{Eynard:2007invariants} reads
\begin{equation}
\widehat{F}_{1}=
-{1\over 2}\log \tau_{B}-{1\over 12}\log \prod_{r\in R^{\circ}} {d(Y+h(X))\over d(X-r_{k})^{1\over 2}}+{1\over 2}\ln \det (\mathrm{Im}\tau)^{-1}\,,
\end{equation}
where $R^{\circ}$ is the set of finite ramification points which in the current case is the set of
all Weierstrass points.\\

In general,
for the sextic model in  \eqref{eqnhyperellipticformfromcanonicalmap} or \eqref{eqnhyperellipticsextic},
the Bergman tau-function is given by\footnote{The formula is valid for the sextic model with $a_0=1$.}  \cite{Kokotov:2004tau}
\begin{equation}
\tau_{B}=2^2\det {\Pi}_{A}^{(6)}(\omega)  \prod_{0\leq i<j\leq 5} (r_{i}-r_{j})^{1\over 4}\,.
\end{equation}
Here again the superscript $(6)$ means the period matrix for the sextic model
and similarly $(5)$ for the quintic model.
Direct calculation shows that
\begin{equation}
\log \prod_{r\in R^{\circ}} {dY\over d(X-X_{r})^{1\over 2}}=
\log \left(a_0^3\prod_{0\leq i<j\leq 5} (r_{i}-r_{j})\right)\,,
\end{equation}
This yields
\begin{eqnarray}\label{eqnF1simplified1}
\widehat{F}_{1}&=&
-{1\over 2}\log \det \Pi_{A}^{(6)}(\omega)\nonumber\\
&&-{5\over 24}\log
\prod_{0\leq i<j\leq 5} (r_{i}-r_{j})-{1\over 4}\log (16a_0)+{1\over 2}\ln \det (\mathrm{Im}\tau)^{-1}\,.
\end{eqnarray}
We now rewrite the above quantities in the quintic model obtained from the coordinate transformation
\eqref{eqntransformationfromsextictoquintic}
which changes the sextic model to the quintic model in \eqref{eqngenustwomodelquintic} or \eqref{eqnhyperellipticquintic} with $b_0=1$.
Straightforward calculation tells that
\begin{equation}
\det \Pi_{A}^{(6)}(\omega)
=\det \Pi_{A}^{(5)}(\omega)
\cdot -{(c_{1}-c_{2})^2 r_{0}^{2}\over \prod_{k=1}^{5} (r_{0}-r_{k})}\,,
\end{equation}

\begin{equation}
 \prod_{0\leq i<j\leq 5} (r_{i}-r_{j})
 ={r_{0}^{15}(c_{2}-c_{1})^{15}\over \prod_{k=1}^{5} (c_{2}+e_{k})^5}
 \prod_{1\leq i<j\leq 5} (e_{j}-e_{i})
  \,.
\end{equation}
Note that
\begin{equation}
c_{2}+e_{k}={r_{0}\over r_{0}-r_{k}} (c_{2}-c_{1})\,.
\end{equation}
Simplifying  \eqref{eqnF1simplified1} we obtain
\begin{eqnarray}\label{eqnF1simplified2}
\widehat{F}_{1}
&=&-{1\over 2}\log \det \Pi_{A}^{(5)}(\omega)-{5\over 24}\log
 \prod_{1\leq i<j\leq 5} (e_{j}-e_{i})+{1\over 2}\ln \det (\mathrm{Im}\tau)^{-1}\nonumber\\
&& -{1\over 4}\log (16 a_0)+{13\over 24}\log {r_{0}^2 (c_{1}-c_{2})^{2} \over  \prod_{k=1}^{5} (r_{0}-r_{k})}\,.
\end{eqnarray}
For definiteness, we take the values for $c_1,c_2$ as in \eqref{eqnquintictoRosenhain}.
The period matrix $\Pi_{A}^{(5)}(\omega)$ is given in
\eqref{eqnperiodforquintic}
and its determinant in \eqref{eqnperiodforRosenhain}.

%%%
\iffalse
It is easy to check that
the contribution of the infinity ramification point to the product term above
for the quintic model can be defined similarly by choosing
a suitable local uniformizer. The contribution is
$b_{0}^{1/ 2}$ for the quintic with leading term $b_0 x^{5}$.
This is the same as the contribution of all of the ramification points, which are finite,
for the corresponding sextic.
Straightforward computations show that
\begin{equation}
2\log \prod_{r\in R^{\circ}} {dy\over d(x-x_{r})^{1\over 2}}=
\log D\,,
\end{equation}
where $D$ is the discriminant
\begin{equation}
D=b_{0}^{10}\prod_{i<j} (e_{i}-e_{j})^2\,.
\end{equation}
\fi
%%%

According to Lemma \ref{lemmodularityofparameters}, the last term in \eqref{eqnF1simplified2} is the logarithm of a modular function.
From \eqref{eqndiscriminantintermsoftheta} proved in \cite{Grant1988:generalization}, one has
\begin{equation}
D=\prod_{i<j} (e_{i}-e_{j})^2
=
\pm (\det ({\Pi_{A}^{(5)}(\omega) \over \pi ^2}))^{-10}
\prod_{\nu \,\,\mathrm{even}}\theta^2_{\nu}(0,\tau)\,.
\end{equation}
The product term in the above expression is nothing but essentially the cusp form $\chi_{10}$ \cite{Igusa:1967modular}
according to
\begin{equation}
-2^{14}\chi_{10}=
\prod_{\nu \,\,\mathrm{even}}\theta^2_{\nu}(0,\tau)\,.
\end{equation}
Combining these facts, we obtain that
$\widehat{F}_{1}$ is the logarithm of a modular form, modulo a constant and the
$\det \mathrm{Im}\tau$ term which is real-analytic modular

\begin{eqnarray}\label{eqnF1simplified3}
\widehat{F}_{1}
&=&{13\over 24}\log {
(\theta_{0110}\theta_{1000}\theta_{0011})^{3}
\over
\theta_{1100}\theta_{0010}\theta_{1001} \theta_{0000}\theta_{0001}\theta_{1111}\theta_{0100}}
-{5\over 48}\log
\chi_{10}+{1\over 2}\ln \det (\mathrm{Im}\tau)^{-1}\nonumber\\
&&-{1\over 4}\log a_0 +{13\over 24}\log {( r_{0}-r_{2} )^2 (r_{0}-r_{4}) ^2\over (r_{2}-r_{4})^2 } {1\over  \prod_{k=1}^{5} (r_{0}-r_{k})}\,.
\end{eqnarray}

%%%
\iffalse
We can now simplify $F_{1}$ into
\begin{eqnarray*}
F_{1}&=&
-{1\over 2}\tau_{B}-{1\over 24}\log D+{1\over 2}\ln \det (\mathrm{Im}\tau)^{-1}\\
&=&-{1\over 2}\tau_{B}-{1\over 24}\log
\left(\pm \det ({\Pi_{A}(\omega) \over \pi ^2}))^{-10}\cdot -
2^{14}\chi_{10}
\right)
+{1\over 2}\ln \det (\mathrm{Im}\tau)^{-1}
\end{eqnarray*}
The term $\det \Pi_{A}(\omega)$ can be obtained from
\eqref{eqnperiodforquintic} which involves the roots
$e_{k},e_{\ell}$.
\fi
%%%

We now summarize the above discussions in the following.

\begin{prop}\label{propmodularityofinitialvalues}
For the mirror curve family of the resolution of $\mathbb{C}^{3}/\mathbb{Z}_{6}$, we obtain the following
modularity for the initial values for topological recursion.
\begin{enumerate}

\item
The disk potential $\partial_{x}W$
is the logarithm of the pull-back of a meromorphic Jacobi form.

\item
The annulus  amplitude $ \omega_{0,2}=\mathcal{B}$
is the pull-back of a weight $2$, index $0$, meromorphic quasi-Jacobi form.
 It is symmetric in its arguments.
The recursion kernel $K=d^{-1}\widehat \omega_{0,2}/(\lambda-\lambda^{*})$
 is the pull-back of a formal almost-meromorphic Jacobi form of formal weight $0$.

\item
Up to addition by a constant, the genus one potential $F_1$ is
the logarithm of a modular form.
\end{enumerate}
\end{prop}
Note that the first and last statements in Proposition \ref{propmodularityofinitialvalues} above
match the results in \cite{Klemm:2015direct} obtained by other means.

\subsubsection{Modularity in topological recursion}

Exactly the same procedure as the genus one case
in \cite{Fang:2018open} yields the modularity and polynomial structure of the
differentials produced by topological recursion.
For example, the structure on the upper bound of the order of pole in
$\widehat{\omega}_{g,n}$
is a consequence of that for $\lambda, \widehat{\omega}_{0,2}$
and the combinatorial patter in the recursion algorithm, and
is independent of details such as the specific family one chooses.
Below we phrase the statements
whose proofs are the same for the genus one case \cite{Fang:2018open}
and are therefore omitted.

\begin{thm}
\label{thmhighergenusWgn}
The following statements hold for $\widehat \omega_{g,n}$ with $2g-2+n> 0$.

\begin{enumerate}
\item
The differential $\widehat \omega_{g,n}(\tilde{Y}_1,\cdots, \tilde{Y}_n), n\neq 0$ is symmetric in its arguments.
In each argument, it only has poles at the ramification points in $R^{\circ}$.
At any of the ramification point, the order of pole in any argument is at most $6g+2n-4$.
Furthermore, the sum of orders of poles over all arguments
in each term in $\widehat \omega_{g,n}(\tilde{Y}_1,\cdots, \tilde{Y}_n)$ is at most $6g+4n-6$.

\item The differential $\widehat \omega_{g,n}(\tilde{Y}_1,\cdots, \tilde{Y}_n),n\neq 0$
is a differential polynomial in $\mathcal{S}(\tilde{Y}_k-\tilde{Y}_r),$ $k=1,2,\cdots,n\,, r\in R$;
the coefficients are elements in the ring
of almost-meromorphic modular forms.
In particular, $\widehat \omega_{g,n}(\tilde{Y}_1,\cdots, \tilde{Y}_n),n\neq 0$
is the pull-back of an almost-meromorphic Jacobi form.

\item The quantities $\hat{F}_{g} ,g\geq 2$ are
 almost-meromorphic modular forms.
\end{enumerate}
\end{thm}

\begin{ex}
The following gives the
explicit expression for $\widehat{\omega}_{0,3}$.
Let the affine coordinates of $p_{i}$ on the mirror curve
\eqref{eqnhyperellipticmirrorcurvesextic} be $(X_{i},\tilde{Y}_{i}, 1), i=1,2,3$.
Then we have
\begin{eqnarray}
&&\widehat{\omega}_{0,3}(p_1,p_2,p_3)\nonumber\\
&=&\sum_{r_k\in R}{r_kh(r_k)\over 4}\prod_{i=1}^3\left({G^{(6)}(X_i, r_k)\over (X_i- r_k)^2}-\begin{pmatrix}1&X_i\end{pmatrix}
\cdot\Pi_A(\omega)^{-1}\widehat{\Pi_A(\eta)}
\cdot\begin{pmatrix}1\\ r_k \end{pmatrix}\right){dX_i\over 2\tilde Y_i}\,,
\end{eqnarray}
where $G^{(6)}(X_1,X_2)$ is as displayed in
\eqref{eqnGforsextic}.
\end{ex}

\subsection{Applications to open-closed Gromov-Witten theory}

The statements on modularity
and structure theorems for the differentials $\omega_{g,n}$
can be translated into those for
 open and closed GW potentials by the proof of the Remodeling Conjecture.
The Yamaguchi-Yau functional equation also follows from keeping track of
the non-holomorphic dependence in $\tau$ of the Schiffer kernel $\mathcal{S}$, in the same way
described in  \cite{Fang:2018open}.

We again only list the main results for completeness.
Interested readers are referred to  \cite{Fang:2018open} for more details.

\begin{thm}\label{thmholhighergenusWgn}
		The open GW potentials $\omega_{g,n},2g-2+n>0, n>0$
		%, as the holomorphic limits of the differentials $\widehat{\omega}_{g,n}$ which are
		%the pull-back of almost-meromorphic Jacobi forms,
		are pull-backs of quasi-meromorphic Jacobi forms.
		The structure as meromorphic quasi-Jacobi forms is as exhibited in Theorem
		\ref{thmhighergenusWgn}, with the Schiffer kernel $\mathcal{S}$ replaced by the Bergman kernel $\mathcal{B}$.
		The closed GW potentials $F_{g}=\omega_{g,0}, g\geq 2$
		%as the holomorphic limits of the differentials $\widehat{\omega}_{g,0}$ under closed mirror map \eqref{eqnclosedmoduli}
		%which are almost-meromorphic modular forms,
		are quasi-modular forms.
			\end{thm}

\begin{thm}\label{thmhae}

Let $T^{a}=\ln Q_{a},a=1,2,3$ be the coordinates on the base $S$ of the mirror curve family $\pi: \mathcal{C}\rightarrow S$ of the resolved manifold $\mathcal{X}$ of $\mathbb{C}^{3}/\mathbb{Z}_{6}$,
where $Q_{a},a=1,2,3$ are given in \eqref{eqnclosedmoduli}.
Let $F_{0}$ be the genus zero potential.
Define the quantity $P^{cd}$ to be a solution to
\begin{equation}\label{eqnclosedstringpropagator}
\bar{\partial}_{\bar{b}}P^{cd}=C^{cd}_{\bar{b}}:={(2 \mathrm{Im}\tau)^{-1,c\bar{c}}}\overline{\left ({\partial^3 F_0\over \partial  T^b \partial   T^c  \partial   T^d}\right)}
{(2 \mathrm{Im}\tau)^{-1,d\bar{d}}}\,.
\end{equation}
		which can alternatively be computed from the Weil-Petersson geometry of the moduli space of complex
		structures of the mirror CY 3-fold $\check{\mathcal{X}}$.
Let
\begin{equation}
\mathbold{\eta}=\Pi_{A}(\omega)^{-1}\Pi_{A}(\eta)\,,
\quad
\mathbold{\hat{\eta}}=\mathbold{\eta}+\mathbold{Y}\,,
\quad
\mathbold{Y}= 2\pi i \Pi_{A}(\omega)^{-1}(\tau-\bar{\tau})^{-1}\Pi_{A}(\omega)^{-t}\,.
\end{equation}
Then the open GW potentials $\omega_{g,n},2g-2+n>0, n>0$ satisfy the holomorphic anomaly equations
	\begin{eqnarray}
&&\left( {\partial \over \partial\mathbold{\eta}^{cd}}
+
\sum_{k,r}  {\partial \over \partial   \mathcal{B}_{kr}^{cd}}\right)
\omega_{g,I+1}  \nonumber\\
&=&
{\partial_{Y} P^{ab} \over \partial_{Y}\mathbold{\hat{\eta}}^{cd}} \cdot {1\over 2}
\left(
\partial_{T^a}\partial_{T^b}{\omega}_{g-1, I+1}
+\sum_{ \substack{g_{1}+g_{2}=g\,, I=J\sqcup K\\
(g_{1},J)\neq (0,\emptyset), (g,I )}}
\partial_{T^a}{\omega}_{g_{1},  J} \cdot \partial_{T^b}{\omega}_{g_{2}, K }
\right)\,,
\end{eqnarray}
where $\mathcal{B}_{kr}=\mathcal{B}(\tilde{Y}_{k}-\tilde{Y}_{r})$
and the indices $a,b,c,d$ label the components of the corresponding quantities.
		The closed GW potentials $F_{g}, g\geq 2$  satisfy
			\begin{equation}
		{\partial \over \partial\mathbold{\eta}^{cd}} F_{g}=
		{\partial_{Y} P^{ab} \over \partial_{Y}\mathbold{\hat{\eta}}^{cd}}\cdot {1\over 2}
		\left(
		\partial_{T^a}\partial_{T^b}F_{g-1}
		+\sum_{\substack{ g_{1}+g_{2}=g\\ g_{1}\neq 0,g}}
				\partial_{T^a}F_{g_{1}} \cdot \partial_{T^b}F_{g_{2} }
		\right)\,.
		\end{equation}
					\end{thm}

In summary,
we proved the modularity of the differentials constructed from topological recursion,
for a special family of genus two hyperelliptic curves.
Using the proof of the Remodeling Conjecture this gets translated to
modularity of the open-closed GW potentials under the mirror maps.

Modularity (quasi-Jacobiness, strictly speaking) in the open moduli is more or less automatic, using
results on the functional field of hyperelliptic Jacobians.
The modularity for the differentials relies on results on the
differential structure of the functional field.
Modularity in the closed moduli, as exhibited in Lemma \ref{lemmodularityofparameters},
requires that the base of the family is a cover of the
Siegel modular variety $\Gamma(1)\backslash \mathcal{H}$.
In particular, our approach does not carry over to a two-parameter family
of genus two mirror curves in a straightforward manner.

Besides the hyperellipticity mentioned in Remark \ref{remhyperellipticity}, the genus two condition is also heavily replied upon in this work.
At genus two, the ring of meromorphic modular forms
is graded by the representations $\mathrm{Sym}^{\otimes k}\underline{\omega}, \det^{\otimes k}\underline{\omega}$
of $\mathbb{U}_{2}$
according to the Weyl character formula.
For higher genus cases, the irreducible representation decomposition is more involved,
making the ring structure of quasi-modular forms less trackable.
Furthermore, for the genus two case the moduli space of smooth curves is essentially the moduli space of principally polarized Abelian varieties of dimension two.
Hence constructions on the latter yield automatically those on the former.
For higher genus cases, we need to pull back the theory of quasi-modular forms
from the moduli space of principally polarized Abelian varieties to the Jacobian locus.
The description of this locus (the Schottky problem) is still an open problem.
The lack of understanding of this problem makes the
generalization of our approach to higher genus curves  difficult.

\begin{appendix}

\section{Dependence of Weierstrass $\wp$-function on frame and marking}
\label{secdependenceappendix}

In this section, we review
the effect of the change of frame and marking
on the matrices of periods and quasi-periods and on
the Weierstrass $\wp$-function, following the discussions in \cite{Grant:1985theta}.

\subsection{Jacobian and Abel-Jacobi map}
\label{secAbelJacobimapappendix}

We recall the notations and conventions from Section \ref{secgenustwocurveandJacobian} as follows.
Consider the genus two hyperellipic curve $C$
of the following form
\begin{equation}
y^{2}=f(x;b)=x^{5}+b_{1}x^{4}+\cdots +b_{5}=\prod_{k=1}^{5}(x-e_{k})\,.
\end{equation}
The Abel-Jacobi map is studied  \cite{Grant:1990formal}, with a carefully chosen weak Torelli marking, the basis for the 1st kind Abelian differentials given by
\begin{equation}\label{eqnbasisof1stkindappendix}
\omega_{1}={dx\over 2y}\,,
\quad
\omega_{2}={xdx\over 2y}\,,
\end{equation}
and another two 2nd kind Abelian differentials given by
\begin{equation}
\eta_{1}={(3x^3+2b_1 x^2+b_{2}x)dx\over 2y}\,,
\quad
\eta_{2}={x^2 dx\over 2 y}\,.
\end{equation}
They form a symplectic basis for $H^{1}(C,\mathbb{C})$ with respect to the Poinca\'e residue paring
which is equivalent to the de Rham paring in singular cohomology up to a $2\pi i$ factor.

The convention for the period matrix and the quasi-period matrix in Section \ref{secgenustwocurveandJacobian}
 is
\begin{equation}\label{eqnperiodmatrixappendix}
\Pi:=\begin{pmatrix}
\Pi_{A}(\omega) & \Pi_{B}(\omega)\\
\Pi_{A}(\eta) & \Pi_{B}(\eta)
\end{pmatrix}
=
\begin{pmatrix}
\int_{A_{1}}\omega_{1} & \int_{A_{2}}\omega_{1} & \int_{B_{1}}\omega_{1} & \int_{B_{2}}\omega_{1} \\
\int_{A_{1}}\omega_{2} & \int_{A_{2}}\omega_{2} & \int_{B_{1}}\omega_{2} & \int_{B_{2}}\omega_{2} \\
\int_{A_{1}}\eta_{1} & \int_{A_{2}}\eta_{1} & \int_{B_{1}}\eta_{1} & \int_{B_{2}}\eta_{1} \\
\int_{A_{1}}\eta_{2} & \int_{A_{2}}\eta_{2} & \int_{B_{1}}\eta_{2} & \int_{B_{2}}\eta_{2}
\end{pmatrix}\,,
\quad
\tau=\Pi_{A}(\omega)^{-1} \Pi_{B}(\omega)\,.
\end{equation}
The Riemann-Hodge bilinear relations/Legendre period relation give rise to
\begin{equation}\label{eqnRiemannHodgeappendix}
\Pi \begin{pmatrix}
0 & \mathbb{1}\\
 - \mathbb{1} & 0
\end{pmatrix} \Pi^{t}=2\pi i
\begin{pmatrix}
0 & \mathbb{1}\\
 - \mathbb{1} & 0
\end{pmatrix}\,.
\end{equation}
It is easy to check that the above is also satisfied once we replace $\Pi$ by $\Pi^{t}$.\\

%%%
\iffalse
\begin{rem}
We now review how the above equations are derived.
The cohomology classes of the basis $\omega,\eta$ can be expressed in terms of
duals of the homologies $\alpha,\beta$, with coefficients being the corresponding periods. That is,
we have
\begin{equation}
\omega=
\begin{pmatrix}
\Pi_{A}(\omega) & \Pi_{B}(\omega)
\end{pmatrix}
\begin{pmatrix}
\alpha \\ \beta
\end{pmatrix}\,.
\end{equation}
By construction, the de Rham pairing $\langle-,-\rangle$
is given by the following
\begin{equation}
\langle-,-\rangle=\mathbb{1}\check{\omega}^{t}\otimes \check{\eta}^{t}-\mathbb{1}\check{\eta}^{t}\otimes \check{\omega}^{t}=\mathbb{1}\check{\alpha}^{t}\otimes \check{\beta}^{t}-\mathbb{1}\check{\beta}^{t}\otimes \check{\alpha}^{t}
\,.
\end{equation}
where $(\check{\omega}^{t},\check{\eta}^{t}),(\check{\alpha}^{t},\check{\beta}^{t})$ stand for the dual bases of $(\omega,\eta)^{t},(\alpha,\beta)^{t}$ respectively.
Evaluating the above paring on $\omega\otimes \eta$ and comparing the coefficients, one then gets the bilinear relations.
\begin{equation}
Q=
 \begin{pmatrix}
\omega \\ \eta
\end{pmatrix} \cdot \begin{pmatrix}
\omega^{t} &  \eta^{t}
\end{pmatrix}
=
Q
\,,
\quad
Q=
\begin{pmatrix}
0 & \mathbb{1}\\
 - \mathbb{1} & 0
\end{pmatrix}\,.
\end{equation}

\end{rem}
\fi
%%%

The curve $C$ is embedded into $J(C)$ via the induced map $\phi: p\mapsto \Phi(p,\infty)$, where
\begin{equation}\label{eqnAJsymmetric}
\Phi: (p_{1},p_{2})\mapsto
u=
\begin{pmatrix}
u_{1}\\
u_{2}
\end{pmatrix}=
\begin{pmatrix}
(\int_{\infty}^{p_{1}} +\int_{\infty}^{p_{2}} )\omega_{1}\\
(\int_{\infty}^{p_{1}} +\int_{\infty}^{p_{2}} )\omega_{2}
\end{pmatrix}\,.
\end{equation}

It turns out with the particularly chosen frame $\omega_{1},\omega_{2},\eta_{1},\eta_{2}$ and the marking above, $\phi(C)$
is the vanishing locus $\Theta$ of
the theta function $\vartheta_{\delta}$ with $\theta$-characteristic
\begin{equation}\label{eqncarefulthetacharacteristic}
\delta=\left(
\begin{pmatrix}
{1\over 2}\\
{1\over 2}
\end{pmatrix},
\begin{pmatrix}
0\\
{1\over 2}
\end{pmatrix}
\right)\,.
\end{equation}

\subsection{Dependence on choices}
\label{secdependonceonchoice}

\subsubsection{Dependence on frame}

A different choice of basis $\omega_{g}$ for the 1st kind Abelian differentials is related to
$\omega=(\omega_{1},\omega_{2})^{t}$ by
\begin{equation}\label{eqntransformationofHodgeframeappendix}
\omega_{g}=g \omega\,,\quad g\in \mathrm{GL}_{2}(\mathbb{C})\,.
\end{equation}
Correspondingly, $\eta$ is changed into
\begin{equation}
\eta_{g}=g^{-t} \eta\,.
\end{equation}
The periods and quasi-periods transform according to
\begin{equation}
\begin{pmatrix}
\Pi_{A} (\omega) & \Pi_{B} (\omega)\\
\Pi_{A} (\eta) & \Pi_{B} (\eta)
\end{pmatrix}
\mapsto
\begin{pmatrix}
g &0\\
0 & g^{-t}
\end{pmatrix}
\begin{pmatrix}
\Pi_{A} (\omega) & \Pi_{B} (\omega)\\
\Pi_{A} (\eta) & \Pi_{B} (\eta)
\end{pmatrix}
\,.
\end{equation}
The Riemann-Hodge bilinear relations \eqref{eqnRiemannHodgeappendix} stay unchanged.

Similarly, the Abel-Jacobi map transforms as
$u\mapsto gu$.
From this one can check that
$\sigma_{(a,b)}(u, \Pi_{A}(\omega),\Pi_{B}(\omega))$, which is defined by replacing the theta-characteristic $\delta$
in
\eqref{eqndefinitionofsigma} by a general one $(a,b)$,
is invariant.
In particular, this allows one to normalize $\omega$ such that $\Pi_{A}(\omega)=1$
in doing computations for a single curve with specified $\tau$, as is customarily  done.
Clearly the quantity
\begin{equation}
\sum_{i,j}du_i\wp_{ij}du_j=du^{t} \,(\wp_{ij})\,du
\end{equation}
is invariant. From this one can then read off the transformation
of the matrix $(\wp_{ij})_{i,j=1,2}$ according to
\begin{equation}
 (\wp_{ij})\mapsto g^{-t} \cdot (\wp_{ij})\cdot g^{-1}\,.
\end{equation}

\subsubsection{Dependence on marking}

The group $\mathrm{Sp}_{4}(\mathbb{Z})$ acts on the marking by
\begin{equation}\label{eqnsymplectictranformationoncyclesappendix}
\begin{pmatrix}
B_{1}\\
B_{2}\\
A_{1}\\
A_{2}
\end{pmatrix}
\mapsto
\gamma
\begin{pmatrix}
B_{1}\\
B_{2}\\
A_{1}\\
A_{2}
\end{pmatrix}\,,
\quad
\gamma=
\begin{pmatrix}
a & b\\
c & d
\end{pmatrix}\in \mathrm{Sp}_{4}(\mathbb{Z})\,.
\end{equation}
Accordingly, one has
\begin{equation}
\gamma:
(A,B)\mapsto (A^{*},B^{*}):=(A,B)
 \begin{pmatrix}
0 & \mathbb{1}\\
  \mathbb{1} & 0
\end{pmatrix}
\gamma^{t}
\begin{pmatrix}
0 & \mathbb{1}\\
  \mathbb{1} & 0
\end{pmatrix}
=(A,B)
\begin{pmatrix}
d^{t}& b^{t}\\
 c^{t} & a^{t}
\end{pmatrix}
\,.
\end{equation}
Here the condition $\gamma\in \mathrm{Sp}_{4}(\mathbb{Z})$ is equivalent to
\begin{equation}\label{eqnsymplecticgroupappendix}
\gamma
\begin{pmatrix}
0 & -\mathbb{1}\\
  \mathbb{1} & 0
\end{pmatrix}
 \gamma^{t}=
\begin{pmatrix}
0 & -\mathbb{1}\\
  \mathbb{1} & 0
\end{pmatrix}
\,.
\end{equation}
It is easy to check that the above is also satisfied once we replace $\gamma$ by $\gamma^{t}$.

Therefore, the transformation of the matrix \eqref{eqnperiodmatrixappendix} of periods and quasi-periods is
\begin{equation}\label{eqntransformationofPiappendix}
\Pi\mapsto \Pi
\begin{pmatrix}
d^{t}& b^{t}\\
 c^{t} & a^{t}
\end{pmatrix}\,.
\end{equation}
In particular, $\tau=\Pi_{A}(\omega)^{-1}\Pi_{B}(\omega)$ transforms as
\begin{equation}\label{eqntransformationoftauappendix}
\gamma:\tau\mapsto (d^{t}+\tau c^{t})^{-1} (b^{t}+\tau a^{t})\,.
\end{equation}
The induced action on $\tau^{t}$ is then the familiar action of $\mathrm{Sp}_{4}(\mathbb{Z})$
on the Siegel upper-half space
\begin{equation}
\tau^{t}\mapsto \gamma\tau^{t}= (a\tau^{t}+b)(c\tau^{t}+d)^{-1}\,.
\end{equation}
Since $\tau=\tau^{t}$,
we shall often denote the above action by
\begin{equation}\label{eqntransformationoftautransposedappendix}
\tau\mapsto \gamma\tau= (a\tau+b)(c\tau+d)^{-1}\,.
\end{equation}

\begin{rem}\label{remconventionbytranspose}
In the above consideration, the action of $\mathrm{GL}_{2}(\mathbb{C})$ is left multiplication
on $\omega$ while that of $\mathrm{Sp}_{4}(\mathbb{Z})$
on the making is right multiplication on $(A,B)$ which commutes with the former.
Had we switched the convention on the matrix $\Pi$ by making a transpose, the resulting action
would be
\begin{equation}
\tau^{t}\mapsto  ( a \tau^{t}+b) (c\tau^{t}+d)^{-1}\,.
\end{equation}
Again since $\tau^{t}=\tau$, the above takes the familiar form of the action of $\mathrm{Sp}_{4}(\mathbb{Z})$
on the Siegel upper-half space.
Correspondingly the Abel-Jacobi map $\Phi$ in \eqref{eqnAJsymmetric} is given by $p_{1}+p_{2}-2\infty\mapsto (u_1, u_2)$.
The previous convention treats cohomologies as row vectors and homologies as column vectors.
In this new convention, cohomologies are column vectors while homologies row vectors.
The new one is more convenient in discussing the representation-theoretic aspects, as we shall see below.
\end{rem}

It follows that
\begin{eqnarray}
&&\Pi_{A}(\eta)\Pi_{A}(\omega)^{-1}\mapsto \Pi_{A}(\eta)\Pi_{A}(\omega)^{-1}+2\pi i \Pi_{A}(\omega)^{-t}(c\tau+d)^{-1}c\Pi_{A}(\omega)^{-1}\,,
\quad \nonumber\\
&&v:=\Pi_{A}(\omega)^{-1}u\mapsto (d^{t}+\tau c^{t})^{-1} \Pi_{A}(\omega)^{-1} u=(c\tau+d)^{-t}v\,.
\end{eqnarray}
Recall the standard fact \cite{Mumford:1983tata}
\begin{equation}
\theta_{(x^{*}, y^{*})}(v^{*},\tau^{*})
=
\chi \det (c\tau+d)^{1\over 2}
e^{\pi i v^{t} (c\tau+d)^{-1} c v }
\theta_{(x, y)}(v,\tau)\,,
\end{equation}
with
\begin{eqnarray}
&& v^{*}=\gamma v=(c\tau+d)^{-t}v\,,
\quad
 \tau^{*}=\gamma\tau=(a\tau+b)(c\tau+d)^{-1}\,,\nonumber
\\
&&  \begin{pmatrix}
x^{*}\\
y^{*}
 \end{pmatrix}
=
 \begin{pmatrix}
 d & -c\\
 -b & a
 \end{pmatrix}
   \begin{pmatrix}
x\\
y
 \end{pmatrix}
 +{1\over 2}
  \begin{pmatrix}
\mathrm{diag} (c^{t}d ) \\
 \mathrm{diag} (a^{t}b)
 \end{pmatrix}\,.
\end{eqnarray}
Here $\chi$ is a multiplier system, $ \mathrm{diag}(M)$
means the row consisting of diagonal entries of $M$.
\begin{eqnarray}\label{eqntransformationofwpappendix}
 && \left( \partial^{2}_{v_iv_j} \log
\sigma_{(x^{*},y^{*})}   \right) (v, \tau^{*})=(c\tau+d) \cdot  \left( \partial^{2}_{v_iv_j} \log
\sigma_{(x,y)}   \right) (v, \tau)  \cdot (c\tau+d)^{t}\,.
\end{eqnarray}
The $\wp$-function, derived from the normalized coordinate, transforms in the following way under an action by $\gamma\in \Gamma(1)$
that  additionally satisfies $(x^{*},y^{*})=(x,y)$:
\begin{eqnarray}
  \left( \wp_{ij} \right) ( \gamma v, \gamma \tau)
&=&
(c\tau+d) \cdot  \left(\wp_{ij}  \right) (v, \tau)  \cdot (c\tau+d)^{t}\,,
\end{eqnarray}
where
\begin{eqnarray}
\gamma v=(c\tau+d)^{-t}v\,,\quad \gamma\tau=({a\tau+b})(c\tau+d)^{-1}\,.
\end{eqnarray}
The $\wp_{ij}$-function derived from $u$-coordinate transforms in the following way under the action
\begin{eqnarray}
\wp_{ij} ( \Pi_{A^*} (\omega)^{-1}u, \gamma \tau)
&=& \wp_{ij}  (\Pi_{A}(\omega)^{-1}u, \tau)\,,
\end{eqnarray}

\section{Modularity and quasi-modularity from geometric perspective}
	\label{appendixqgeometricperspective}

\subsection{Quasi-periods and quasi-modularity}\label{appendixquasiperiods}

The modular group is induced by group action on the weak Torelli marking $\{A_{1},A_{2}, B_{1},B_{2}\}$ which is equivalent to a basis $\{\alpha_{1},\alpha_{2}, \beta_{1},\beta_{2}\}$ of $H^{1}(C,\mathbb{Z})$.
We also fix a symplectic basis $(\omega,\eta)$ for the cohomology $H^{1}(C,\mathbb{C})$ adapted to the Hodge filtration.
Here we have used the short-hand notations $\omega=(\omega_{1},\omega_{2}), \eta=(\eta_{1},\eta_{2})$.\\

We consider the vectors in the cohomology space $H^{1}(C,\mathbb{C})$ to be column vectors.
Let
\begin{equation}
	\gamma:=
	\begin{pmatrix}
a &b \\
c & d
\end{pmatrix}\in \mathrm{Sp}_{4}(\R)\,
\end{equation}
act by $(\beta,\alpha)\mapsto (\beta,\alpha) \gamma^{-1}$.
This means that on the basis of the homologies considered to be row vectors, we have
the left multiplication as in \eqref{eqnsymplectictranformationoncyclesappendix}
\begin{equation}\label{eqnactiononcycles}
\begin{pmatrix}
B_{1} \\
B_{2}\\
A_{1}\\
A_{2}
\end{pmatrix} \mapsto \gamma
\begin{pmatrix}
B_{1} \\
B_{2}\\
A_{1}\\
A_{2}
\end{pmatrix} \,.
\end{equation}
Therefore, it is more natural to use the following new convention  for the matrix of periods and quasi-periods instead of the one in \eqref{eqnperiodmatrixappendix}
\begin{equation}
\Pi:=\begin{pmatrix}
\Pi_{B}(\omega) & \Pi_{B}(\eta)\\
\Pi_{A}(\omega) & \Pi_{A}(\eta)
\end{pmatrix}\,,
\quad
\tau= \Pi_{B}(\omega)\Pi_{A}(\omega)^{-1}\,.
\end{equation}
This new convention is the one explained in
Remark \ref{remconventionbytranspose}.\\

The following is a standard result for symplectic groups.

\begin{lem}\label{lemsymplecticgroup}
Let $\gamma=\begin{pmatrix}
a &b \\
c & d
\end{pmatrix}\in \mathrm{Sp}_{4}(\R)$, then
$\gamma^{t}\in \mathrm{Sp}_{4}(\R)$. In particular, one has
\begin{eqnarray*}
 a b^t= b a^t\,, \quad cd^t=dc^t\,, \quad ad^t-bc^{t}=1\,,\\
  a^t c= c^t a\,, \quad b^t d=d^t b\,, \quad  a^t d-c^{t} b=1\,.
\end{eqnarray*}
Consequentially
\begin{eqnarray*}
(c\tau+d)^t (a\tau+b)=(a\tau+b)^t (c\tau+d)\,.
\end{eqnarray*}
\end{lem}

From Lemma \ref{lemsymplecticgroup} we have the following result.
\begin{lem}\label{lemtautransformation}
Let $\gamma=\begin{pmatrix}
a &b \\
c & d
\end{pmatrix}\in \mathrm{Sp}_{4}(\R)$. Then
under the action by $\gamma$, one has
\begin{eqnarray*}
d\tau
&\mapsto&
(c\tau+d)^{-t} \cdot  d\tau \cdot (c\tau+d)^{-1}\,,\\
(\tau-\bar{\tau})
&\mapsto&
(c\tau+d)^{-t }  (\tau-\bar{\tau}) (c\bar{\tau}+d)^{-1}\,.
\end{eqnarray*}
Consequentially, one has
\begin{eqnarray*}
(\gamma\tau-\overline{\gamma\tau})^{-1}&=&
(c\bar{\tau}+d)  (\tau-\bar{\tau})^{-1}(c\tau+d)^t\\
&=&(c\tau+d)  (\tau-\bar{\tau})^{-1}(c\tau+d)^t
- c(c\tau+d)^t\\
(\int_{A} \omega )^{-1}(\tau-\overline{\tau})^{-1}(\int_{A} \omega )^{-t}
&\mapsto&
(\int_{A} \omega )^{-1}(\tau-\overline{\tau})^{-1}(\int_{A} \omega )^{-t}\\
&&- (\int_{A} \omega )^{-1} (c\tau+d)^{-1} c (\int_{A} \omega )^{-t}
\,.
\end{eqnarray*}
\end{lem}
For the symplectic basis $(\omega,\eta)$,
the Riemann-Hodge bilinear relation is equivalent to
\begin{equation}\label{eqnRiemannHodgePi}
\Pi \begin{pmatrix}
0 & \mathbb{1}\\
 - \mathbb{1} & 0
\end{pmatrix} \Pi^{t}=2\pi i
\begin{pmatrix}
0 & \mathbb{1}\\
 - \mathbb{1} & 0
\end{pmatrix}\,.
\end{equation}
This gives in particular
\begin{equation}\label{eqnRiemannHodge}
\left(\Pi_{A}(\omega)^{-1}\Pi_{A}(\eta)\right)^{t}
=
\Pi_{A}(\omega)^{-1}\Pi_{A}(\eta)\,,
\quad
(\int_{A} \omega)^{t}
(\int_{B} \eta)
-
(\int_{B} \omega)^{t}
(\int_{A} \eta)=2\pi i \,.
\end{equation}
By the convention in \eqref{eqnactiononcycles} and the relation in
\eqref{eqnRiemannHodge},
an easy computation gives the following transformation law for the quasi-periods
$(\int_{A}\omega)^{-1} (\int_{A}\eta)$.
\begin{lem}\label{lemquasiperiodstransformation}
Let $\gamma=\begin{pmatrix}
a &b \\
c & d
\end{pmatrix}\in \mathrm{Sp}_{4}(\R)$. Then
under the action by $\gamma$, one has
\begin{equation}
(\int_{A}\omega)^{-1} (\int_{A}\eta)
\mapsto
(\int_{A}\omega)^{-1} (\int_{A}\eta)
+(\int_{A}\omega)^{-1} (c\tau+d)^{-1} 2\pi i c (\int_{A}\omega)^{-t}\,.
\end{equation}
\end{lem}
Further details for the computations can be found in \cite{Grant:1985theta, Buchstaber:1997}.
The above tells that
the obstruction of the modularity of $(\int_{A}\omega)^{-1} (\int_{A}\eta)$ given above is
cohomological: it is due to the second Riemann-Hodge bilinear relation in \eqref{eqnRiemannHodge}.
This observation has recently found interesting applications in two dimensional conformal field theories on  Riemann surfaces \cite{Li:2021}.
\\

Combining Lemma \ref{lemtautransformation} and Lemma \ref{lemquasiperiodstransformation}, one obtains
\begin{lem}\label{lemalmostmodular}
Let $\gamma=\begin{pmatrix}
a &b \\
c & d
\end{pmatrix}\in \mathrm{Sp}_{4}(\R)$, then
\begin{equation*}
(\int_{A}\omega)^{-1} (\int_{A}\eta)+ 2\pi i (\int_{A} \omega )^{-1}(\tau-\overline{\tau})^{-1}(\int_{A} \omega )^{-t}
\end{equation*}
is modular invariant.
\end{lem}

Due to the first relation in the Riemann-Hodge bilinear relations \eqref{eqnRiemannHodge},
the above relation in Lemma \ref{lemalmostmodular} remains under the action by transpose.

\subsection{Representation-theoretic perspective of Hodge bundles}\label{appendix:reps}

We now explain the representation-theoretic perspective of the vector bundles $\underline{\omega},\Omega_{S}^{1} ,\mathcal{V}_{S}$ from Hodge theory,
where $\underline{\omega}$ is the Hodge bundle, $\Omega_{S}^{1}$ the cotangent bundle and
$\mathcal{V}_{S}$ the flat bundle corresponding to the local system of the 1st cohomology.
We shall use the convention mentioned in Section \ref{appendixquasiperiods}

\iffalse
As already mentioned in Remark \ref{remconventionbytranspose}, by
choosing the different convention for $\Pi$ from the one in \eqref{eqnperiodmatrixappendix}, we have
\begin{equation}
 \tau=(\int_{B}\omega ) (\int_{A}\omega)^{-1}\,.
 \end{equation}
The action of $\mathrm{Sp}_{4}(\mathbb{Z})$ is then given by
\begin{equation}
\tau\mapsto \gamma\tau=(a\tau+b)(c\tau+d)^{-1}\,.
\end{equation}
Hereafter we shall work with this new convention.\\
\fi

Assuming $S=\Gamma\backslash \mathcal{H}$ for some $\Gamma<\Gamma(1)$.
We now lift
any vector bundle on $S$ to the corresponding trivial vector bundle on the universal cover $\mathcal{H}$ of $S$.
Take the frame for the lift $ \widetilde{\mathcal{V}_{S}}$ of $\mathcal{V}_{S}$  to be $\varepsilon=(\beta,\alpha)$.
A section $s$ of $\underline{\omega}\rightarrow S$ is described by a $\Gamma$-equivariant section $\tilde{s}$ of the lifted bundle $\widetilde{\underline{\omega}}$ over $\mathcal{H}$.
The lifted bundle has a local trivialization given by $e(\tau)=\varepsilon (\tau, \mathbb{1})^t$ transforming according to
\begin{equation}
e (\gamma\tau)=e(\tau)(c\tau+d)^{-1}\,.
\end{equation}
Let $\tilde{s}=e\tilde{v}$, then the coordinate  $v$ with respect to the trivialization transforms as
\begin{equation}
\tilde{v} (\gamma\tau)=  (c\tau+d) \tilde{v}(\tau)\,.
\end{equation}
The bundle $\mathrm{Sym}^{\otimes 2}\underline{\omega}$
can be conveniently described by using the local frame corresponding to the entries of $e^t\otimes e$.
The action by $\Gamma$ on this local frame is
\begin{equation}
e^t\otimes  e\mapsto
(c\tau+d)^{t,-1}
e^t\otimes e
(c\tau+d)^{-1}\,.
\end{equation}
Comparing with Lemma
\ref{lemtautransformation}, we obtain the
 Kodaira-Spencer map
$\Omega_{S}^{1}\rightarrow \mathrm{Sym}^{\otimes 2}\underline{\omega}$ by examining the representations
that define the corresponding vector bundles.
For universal families,
the above map gives an isomorphism
$\Omega_{S}^{1}\cong \mathrm{Sym}^{\otimes 2}\underline{\omega}$.
In particular, from the discussion above, one can see that the Eisenstein series $E_{2k}, k\geq 2$ are holomorphic sections of $(\det \underline{\omega})^{\otimes 2k}$.\\

We can further use the homogeneous space structure on $\mathcal{H}=\mathrm{Sp}_{4}(\R)/\mathbb{U}_{2}$.
Recall that a homogeneous vector bundle $V_{\rho}$ on $\mathcal{H}$ associated to a representation
$\rho$ of $\mathbb{U}_{2}$ is defined to be
$\mathrm{Sp}_{4}(\R)\times_{\rho}V$, where the identification is
$(g, v)\sim (gu, \rho(u^{-1}) v), \forall g\in \mathrm{Sp}_{4}(\R), u\in \mathbb{U}_{2}, v\in V$.
The bundles $\widetilde{\underline{\omega}}, \widetilde{\mathcal{V}_{S}}$
are determine by the fundamental representation and trivial representation, respectively.
The Baily-Borel embedding gives an embedding from the period domain
$D=\mathcal{H}$ to its compact dual $\check{D}$ which is a flag variety.
The homogeneous vector bundles $\widetilde{\underline{\omega}}, \widetilde{\mathcal{V}_{S}}$
are the pull-backs of the corresponding tautological bundles on  $\check{D}$.
\\

The benefit of the above discussions is that
all of the relevant vector-valued representations in the usual definition of Siegel modular forms
become geometric. The representation-theoretic perspective allows to use machinery such as Hodge theory and Hermitian geometry to
study modular forms.

\end{appendix}

\providecommand{\bysame}{\leavevmode\hbox to3em{\hrulefill}\thinspace}
\providecommand{\MR}{\relax\ifhmode\unskip\space\fi MR }
\providecommand{\MRhref}[2]{%
  \href{http://www.ams.org/mathscinet-getitem?mr=#1}{#2}
}
\providecommand{\href}[2]{#2}

\bigskip{}

\noindent{\small Institute of Advanced Study for Mathematics, Zhejiang University, Hangzhou, P. R. China}

\noindent{\small  Email: \tt yongbin.ruan@yahoo.com}

\medskip{}

\noindent{\small Department of Mathematics, University of Michigan, 2074 East Hall, 530 Church Street,
	Ann Arbor, MI 48109, USA}

\noindent{\small Email: \tt zyingchu@umich.edu}

\medskip{}

\noindent{\small Yau Mathematical Sciences Center, Tsinghua University, Beijing 100084, P. R. China}

\noindent{\small Email: \tt jzhou2018@mail.tsinghua.edu.cn}

\end{document}